\documentclass[a4paper,10pt,twoside,english]{scrartcl}
\usepackage[utf8x]{inputenc}
\usepackage{amsmath,amssymb,amsthm}
\usepackage{mathrsfs}  
\usepackage{dsfont} 
\usepackage[english]{babel}
\usepackage[unicode=true]{hyperref}
\usepackage[all]{hypcap}
\usepackage[T1]{fontenc}
\usepackage{lmodern}
\usepackage{graphicx}
\usepackage{caption}
\usepackage{subcaption}
\usepackage{enumerate}
\usepackage{setspace}
\usepackage{xspace} 
\usepackage{bm}
\usepackage{blkarray}
\usepackage{color}
\usepackage[top=1in, bottom=1in, left=2cm, right=2cm]{geometry}

\newcommand{\titel}{Location of the spectrum of Kronecker random matrices}
\title{\titel} 	
\author{
Johannes Alt\footnote{\hspace{0.15cm}Partially funded by ERC Advanced Grant RANMAT No. 338804.\newline Date: \today} 
\addtocounter{footnote}{-1}\addtocounter{Hfootnote}{-1}\\
{\small \begin{tabular}{c}{IST Austria}\\{johannes.alt@ist.ac.at} \end{tabular}} 
\and László Erd\H{o}s\footnotemark \addtocounter{footnote}{-1}\addtocounter{Hfootnote}{-1}\\
{\small \begin{tabular}{c} IST Austria\\ lerdos@ist.ac.at\end{tabular}} 
\and Torben Krüger\footnotemark \addtocounter{footnote}{-1}\addtocounter{Hfootnote}{-1}\\ 
{\small \begin{tabular}{c} IST Austria\\ torben.krueger@ist.ac.at\end{tabular}}
\and Yuriy Nemish\footnotemark\\
{\small \begin{tabular}{c} IST Austria\\ yuriy.nemish@ist.ac.at\end{tabular}}
}
\date{}

\hypersetup{
	pdftitle={\titel},
	pdfauthor={Johannes Alt, László Erd\H{o}s, Torben Krüger, Yuriy Nemish}
}

\numberwithin{equation}{section}

\newcommand{\R}{\mathbb{R}}  
\ifdefined\C\renewcommand{\C}{\mathbb{C}}\else\newcommand{\C}{\mathbb{C}}\fi 
\renewcommand{\Im}{\mathrm{Im}\,} 
\renewcommand{\Re}{\mathrm{Re}\,} 
\newcommand{\N}{\mathbb{N}}  
\newcommand{\E}{\mathbb{E}}  
\renewcommand{\P}{\mathbb{P}}  
\newcommand{\di}{\mathrm{d}} 
\newcommand{\eps}{\varepsilon} 
\newcommand*{\defeq}{\mathrel{\vcenter{\baselineskip0.5ex \lineskiplimit0pt\hbox{\scriptsize.}\hbox{\scriptsize.}}}=}

\newcommand{\pt}{\partial}

\DeclareMathOperator{\supp}{supp}

\DeclareMathOperator{\linspan}{span}

\newcommand{\cf}{\boldsymbol c}
\newcommand{\df}{\boldsymbol d}

\newcommand{\gf}{\boldsymbol g}
\newcommand{\yf}{\boldsymbol y}
\newcommand{\xf}{\boldsymbol x}
\newcommand{\af}{\boldsymbol a}

\newcommand{\mf}{\boldsymbol m}

\newcommand{\rf}{\ensuremath{\boldsymbol r}}

\newcommand{\Af}{\boldsymbol A}
\newcommand{\Bf}{\boldsymbol B}

\newcommand{\Df}{\boldsymbol D}

\newcommand{\Ff}{\boldsymbol F}
\newcommand{\Gf}{{\boldsymbol G}}
\newcommand{\Hf}{\boldsymbol H}
\newcommand{\Lf}{{\mathscr{L}}}
\newcommand{\Mf}{{\boldsymbol M}}
\newcommand{\Qf}{\boldsymbol Q}
\newcommand{\Rf}{{\boldsymbol R}}
\newcommand{\Sf}{{\mathscr{S}}}
\newcommand{\Tf}{{\boldsymbol T}}
\newcommand{\Uf}{{\boldsymbol U}}
\newcommand{\Vf}{\boldsymbol V}
\newcommand{\Wf}{{\boldsymbol W}}
\newcommand{\Xf}{{\boldsymbol X}}

\newcommand{\smin}{\sigma_{\mathrm{min}}}

\newcommand{\Hb}{\mathbb H}
\newcommand{\Hbin}{{\mathbb H}_{\mathrm{stab}}}
\newcommand{\Hboutone}{{\mathbb H}_{\mathrm{out}}^{(1)}}
\newcommand{\Hbouttwo}{{\mathbb H}_{\mathrm{out}}^{(2)}}

\newcommand{\id}{\mathds{1}}
\newcommand{\Id}{\mathrm{Id}}
\renewcommand{\char}{\ensuremath{\chi}} 
	
\newcommand{\bs}[1]{\boldsymbol{{#1}}} 
\renewcommand{\rm}{\mathrm} 

\newcommand{\mC}{\ensuremath{{\mathcal C}}}
\newcommand{\mF}{\ensuremath{{\boldsymbol{\mathcal F}}}}

\newcommand{\normtwo}[1]{\lVert #1 \rVert_{2}}

\newcommand{\normhs}[1]{\lVert #1 \rVert_{\mathrm{hs}}}
\newcommand{\normsp}[1]{\lVert #1 \rVert_{\mathrm{sp}}}

\newtheoremstyle{test}
  {}
  {}
  {\itshape}
  {}
  {\bfseries}
  {.}
  { }
  {}

\theoremstyle{test}
\newtheorem{defi}{Definition}[section]
\newtheorem{assums}[defi]{Assumptions}

\newtheorem{exam}[defi]{Example}
\newtheorem{rem}[defi]{Remark}

\newtheorem{thm}[defi]{Theorem}
\newtheorem{lem}[defi]{Lemma}
\newtheorem{coro}[defi]{Corollary}
\newtheorem{pro}[defi]{Proposition}
\newtheorem*{rem*}{Remark}   
\newtheorem*{ex*}{Example}   
\newtheorem*{pro*}{Proposition} 
\newtheorem*{def*}{Definition}
\newtheorem*{coro*}{Corollary}
\newtheorem*{thm*}{Theorem}

\theoremstyle{test}
    \newtheorem{theorem}[defi]{Theorem}
    \newtheorem{proposition}[defi]{Proposition}
    \newtheorem{corollary}[defi]{Corollary}
    \newtheorem{lemma}[defi]{Lemma}
    \newtheorem{definition}[defi]{Definition}
    
    \newtheorem{convention}[defi]{Convention}
    \newtheorem{remark}[defi]{Remark}

\newcommand{\bels}[2] {
        \begin{equation} \label{#1} \begin{split} 
                #2 
        \end{split} \end{equation}
        }

\renewcommand{\bf}[1]{\boldsymbol{\mathrm{#1}}} 
\renewcommand{\cal}{\mathcal} 
\newcommand{\scr}{\mathscr} 
 
\newcommand{\ul}[1]{\underline{#1} \!\,} 
\newcommand{\ol}[1]{\overline{#1} \!\,} 

\newcommand{\wt}{\widetilde}

\renewcommand{\P}{\mathbb{P}}

\newcommand{\ii}{\mathrm{i}} 
\newcommand{\dd}{\mathrm{d}}

\newcommand{\p}[1]{({#1})}
\newcommand{\pb}[1]{\bigl({#1}\bigr)}
\newcommand{\pB}[1]{\Bigl({#1}\Bigr)}
\newcommand{\pbb}[1]{\biggl({#1}\biggr)}

\newcommand{\sbb}[1]{\biggl[{#1}\biggr]}

\newcommand{\cB}[1]{\Bigl\{{#1}\Bigr\}}
\newcommand{\cbb}[1]{\biggl\{{#1}\biggr\}}

\newcommand{\abs}[1]{\lvert #1 \rvert}
\newcommand{\absb}[1]{\big\lvert #1 \big\rvert}
\newcommand{\absB}[1]{\Big\lvert #1 \Big\rvert}
\newcommand{\absbb}[1]{\bigg\lvert #1 \bigg\rvert}

\newcommand{\norm}[1]{\lVert #1 \rVert}

\newcommand{\avg}[1]{\langle #1 \rangle}
\newcommand{\avgb}[1]{\big\langle #1 \big\rangle}

\newcommand{\scalar}[2]{\langle{#1} \mspace{2mu}, {#2}\rangle}

\DeclareMathOperator{\Tr}{Tr}

\DeclareMathOperator{\im}{Im}

\DeclareMathOperator{\dist} {dist}                
\DeclareMathOperator{\diam} {diam}                
						
\DeclareMathOperator{\spec}{Spec}				

\newcommand{\2} {\mspace{2 mu}}

\newcommand{\genarg} {{\,\cdot\,}}  

\newcommand{\sopb}{\boldsymbol{\mathcal S}}
\newcommand{\topb}{\boldsymbol{\mathcal T}}
\newcommand{\lopb}{\boldsymbol{\mathcal L}}

\newcommand{\blockdiag}{{{\mathcal{D}}}}
\newcommand{\tenspace}{{{\mathcal{M}}}}
\newcommand{\supnon}{{\mathbb{D}}}
\newcommand{\supmeas}{{\supp \rho}}

\newcommand{\disvz}[1][]{{d_{\rho}^{#1}(z)}}
\newcommand{\drho}[1][]{{d_{\rho}^{#1}}}

\begin{document}
\maketitle

\vspace{-0.7cm}

\begin{abstract}
\noindent
\bf{Abstract: } 
 For a general class of large non-Hermitian random block matrices $\Xf$  we  prove that  
 there are no eigenvalues away from a  deterministic set 
with very high probability. This set is  obtained from the Dyson equation of the Hermitization of $\Xf$
as the self-consistent approximation of the pseudospectrum. We 
demonstrate 
that the analysis of the  matrix Dyson equation from \cite{AjankiCorrelated} offers a unified
treatment of many structured matrix ensembles. 
\end{abstract}


\section{Introduction}

Large random matrices tend to exhibit deterministic patterns due to the cumulative effects of 
many independent random degrees of freedom. 
The Wigner semicircle law  \cite{Wigner1955}  describes the deterministic limit of the empirical density of eigenvalues
of Wigner matrices, i.e., \emph{Hermitian} random matrices with i.i.d. 
entries (modulo the Hermitian symmetry).  For \emph{non-Hermitian} matrices with i.i.d. entries, the limiting 
density is Girko's circular law,  i.e., the  uniform distribution in a disc centered around zero in the complex plane,
see  \cite{bordenave2012} for a review. 

For more complicated ensembles, no simple formula exists for  the limiting behavior, but second order
perturbation theory predicts that it may be obtained from the solution to a nonlinear equation, called the \emph{Dyson equation.}
While simplified forms of the Dyson equation  are present in practically every work on random matrices, its full scope
has only recently been analyzed systematically, see \cite{AjankiCorrelated}.
 In fact, the  proper Dyson equation describes not only the density of states but
the entire resolvent of the random matrix. Treating it as a genuine \emph{matrix equation} unifies many previous 
works that were specific to certain structures imposed on the random matrix.  These additional structures
often masked a fundamental  property of the Dyson equation, its stability against small perturbations,
that plays a key role in proving the expected  limit theorems, also called \emph{global laws}.
 Girko's monograph \cite{girko2012theory} is the most systematic
collection of many possible ensembles, yet it  analyzes them on a case by case basis. 

 In this paper, using  
the setup of the \emph{matrix Dyson equation (MDE)} from \cite{AjankiCorrelated},
  we demonstrate a unified treatment  for  a large class  of random matrix ensembles
that contain or generalize many of Girko's models. For brevity, we focus only on two basic problems: (i) obtaining  the  global law and (ii)
locating the spectrum.
 The global law, typically formulated as a weak convergence
of linear statistics of the eigenvalues, describes only the overwhelming majority of the eigenvalues. Even local versions
of this limit theorem,  commonly called \emph{local laws} (see e.g.  \cite{EJP2473,Bourgade2014,AltInhomCirc}
and references therein) are typically not sensitive to individual eigenvalues and 
they do not exclude that a few eigenvalues are located far away from the support of the density of states.

Extreme eigenvalues have nevertheless been controlled  in some simple  cases. In particular, 
for the i.i.d. cases, it is known
that with a very high probability all eigenvalues lie in an $\varepsilon$-neighborhood of the support of
the density of states.
These results can be proven with the moment method,  see \cite[Theorem 2.1.22]{anderson2010introduction} 
for the Hermitian (Wigner) case, and  \cite{geman1986} for the non-Hermitian i.i.d. case; see also \cite{Bai1986,bai1988} for  the optimal moment condition.
More generally, norms of polynomials in  large independent random matrices can be computed via free probability;
for GUE or GOE Gaussian matrices it was achieved in  \cite{Haagerup2005} 
and generalized to polynomials of general Wigner and Wishart type matrices in  \cite{anderson2013,Capitaine2007}. 
These results have  been  extended recently to polynomials 
that  include deterministic matrices with the goal of studying outliers, see  \cite{Belinschi2016} and references therein.

All these works concern  Hermitian matrices either directly or indirectly by
 considering  quantities, such as norms of non-Hermitian polynomials, that can be deduced
from related Hermitian problems.  
For general Hermitian random matrices, the density of states may be supported on several intervals. 
In this situation, excluding eigenvalues outside of the convex hull of this support is typically easier 
than excluding possible eigenvalues lying inside the gaps of the support.
This latter problem, however, is especially important for studying the spectrum of non-Hermitian random matrices $\Xf$,
since the eigenvalues of $\Xf$ around a complex parameter $\zeta$ can be understood by studying 
the spectrum of the Hermitized matrix 
\begin{equation}\label{Hzeta def}
   \Hf^\zeta = \begin{pmatrix} 0 & \Xf-\zeta\cr \Xf^*-\bar \zeta & 0\end{pmatrix}
\end{equation}
around 0. 
Note that for $\zeta \in\C$ away from the spectrum of $\Xf$, zero  will typically fall inside a gap of  the spectrum of $\Hf^\zeta$ by its symmetry.

In this paper, we consider a very general class of structured block matrices  $\Xf$ that we call \emph{Kronecker random matrices} 
 since their structure is reminiscent to the Kronecker product of matrices. 
They have  $L\times L$ large blocks and   each block
 consists of  a linear combination of  random $N\times N$ matrices with centered, independent,
not necessarily identically distributed entries;  see \eqref{eq:def_Xf}
later for the precise definition. 
We will keep $L$ fixed and let $N$ tend to infinity.
The matrix $\bf{X}$ has a correlation structure that stems from allowing the same $N \times N$ matrix to appear in different blocks.
This introduces an arbitrary linear dependence among the blocks, while keeping
independence inside the blocks.  The dependence is thus described by $L\times L$ deterministic \emph{structure matrices}.

Kronecker random 
 ensembles  occur in many real-world applications of random matrix theory, especially in  evolution of ecosystems 
\cite{Hastings1992} and  neural networks \cite{PhysRevLett.97.188104}. 
 These evolutions are described by a large system of ODE's with random coefficients and the spectral radius
of the coefficient matrix  determines the long time stability, see  \cite{may1972will} for the original idea. More recent results are found in 
\cite{PhysRevE.91.012820,Aljadeff2015,PhysRevLett.114.088101}
and references therein.  The ensemble we study here  is even more general as it allows for  linear dependence among the blocks
described by arbitrary structure matrices. This level of generality is essential for another
 application; to study spectral properties of polynomials of random matrices. These 
 are often studied via the ``linearization trick'' and the linearized matrix is exactly a Kronecker random matrix.
 This application is presented  in \cite{Erdos2017Polynomials}, where
the    results of the current paper are directly used.

We present general results that  exclude eigenvalues of Kronecker random matrices away from a  deterministic set  $\supnon$
with a very high probability. 
The set $\supnon$ is determined by solving the self-consistent Dyson equation.
In the Hermitian case, $\supnon$ is the \emph{self-consistent spectrum} defined as the  support of the \emph{self-consistent 
density of states} $\rho$ which is defined as the imaginary part of the solution 
to the Dyson equation when restricted to the real line. 
We also address the
general  non-Hermitian setup,  where the eigenvalues are not confined to the real line.
In this case,  the set $\supnon=\supnon_\eps$ contains an additional cutoff parameter $\eps$ 
and it is the \emph{self-consistent $\eps$-pseudospectrum}, 
 given via  the Dyson equation for the Hermitized problem $\Hf^\zeta$, see \eqref{eq:def_support_non_hermitian}
later.
The $\eps\to0 $ limit of the sets $\supnon_\eps$
 is expected not only to contain but to  coincide with the support of the 
 density of states  in the non-Hermitian case as well, but this  has been proven
only in some special cases.  We provide numerical examples to support this conjecture.

We point out that the global law and the location of the spectrum for $A+X$, where $X$ is an  i.i.d. centered random
matrix  and $A$ is a  general deterministic  matrix (so-called \emph{deformed ensembles}), have been extensively 
studied, see \cite{Bai2012,Bordenave2016,Bordenave_spec_rad2016,Tao2013,tao2010}.
For more references, we refer to  
the review \cite{bordenave2012}. In contrast to these papers,  the main focus of our work is to allow for general
 (not necessarily identical) distributions of the matrix elements.

In this paper, we first study   arbitrary  Hermitian Kronecker matrices $\Hf$; the Hermitization $\Hf^\zeta$ 
of a general Kronecker matrix is itself a Kronecker matrix and therefore 
just a special case. Our first
result is the global law, i.e., we show that the empirical density of states of  $\Hf$ is asymptotically given 
by the self-consistent density of states $\rho$ determined by
the Dyson equation. 
 We then also prove an optimal local law  for spectral parameters away from the instabilities of the Dyson equation.
The Dyson equation for Kronecker matrices is
 a system  of $2N$ nonlinear equations for $L\times L$ matrices, see
\eqref{eq:Dyson_non_Hermitian} later.   In case of identical distribution of the entries within each $N\times N$ matrix,
the system reduces to a single equation for a $2L\times 2L$ matrix -- a computationally feasible problem. 
 This  analysis provides not only the limiting  density of states but also 
 a full understanding of the resolvent for spectral parameters $z$ very close to the real line, down to scales $\im z\gg 1/N$. 
Although the optimal local law down to scales $\im z\gg 1/N$ cannot capture individual eigenvalues inside the support of $\rho$,
the key point is that outside of this support  a stronger estimate in the local law may be proven that  actually detects 
individual eigenvalues, or rather lack thereof.  
This observation  has been used for simpler models before, in particular \cite[Theorem 2.3]{EJP2473} already contained this stronger form 
of the local semicircle law for generalized Wigner matrices, see also \cite{Ajankirandommatrix} for Wigner-type matrices, 
\cite{AltGram} for Gram matrices and \cite{Erdos2017Correlated}   for correlated matrices with a uniform lower bound on
the variances. 
In particular, by running the stability analysis twice, this allows for an extension of the local law
for any $\im z>0$  outside of the support of $\rho$.

Finally, applying the local law to the Hermitization $\Hf^\zeta$ of a non-Hermitian Kronecker matrix $\Xf$, 
 we  translate
local spectral information on $\Hf^\zeta$ around 0 into information about the location of the spectrum of $\Xf$.
This is possible since $\zeta\in \mbox{Spec} (\Xf)$ if and only if $0\in \mbox{Spec} (\Hf^\zeta)$.  In practice, we give a
good approximation to the  
$\eps$-pseudospectrum of $\Xf$ by considering the set of those $\zeta$ values in $\C$ for which 0 is at least $\eps$ distance away from 
the support of the self-consistent density of states for $\Hf^\zeta$.

In the main part of the  paper,  we give a short, self-contained proof that directly aims at locating the  Hermitian  spectrum under the weakest conditions 
for the most general setup.  
We split the proof into two well-separated parts; a random and a deterministic one.
 In Section~\ref{sec:Hermitian_Kronecker_matrices} and~\ref{sec:FA} as well as Appendix~\ref{app:proof_local_law} we give
 a model-independent probabilistic proof of the main technical result, the local law 
(Theorem~\ref{thm:no_eigenvalues} and Lemma~\ref{lem:local_law}), 
assuming only two explicit conditions, boundedness and stability,  on the solution of the Dyson equation that can be checked separately for concrete models.
In Section~\ref{subsec:stability_mde}  we prove that these two conditions  are satisfied
for Kronecker matrices away from the self-consistent spectrum. 
The  key inputs behind the stability are (i) a matrix version of the Perron-Frobenius theorem and (ii) 
a sophisticated symmetrization procedure that is much more transparent
 in the matrix formulation. 
In particular, the global law is an immediate consequence of this approach. Moreover, the analysis 
 reveals  that outside of the spectrum the stability holds without any lower bound on the variances,
 in contrast to local laws inside the bulk spectrum that typically require some non-degeneracy condition on the matrix of variances. 
 
 We stress that only the first part involves randomness and we follow
  the Schur complement method and  concentration estimates for linear and quadratic functionals
of independent random variables. Alternatively, we could have used 
the cumulant expansion method that is typically  better suited for ensembles with correlation \cite{Erdos2017Correlated}.
We opted for the former path to demonstrate that correlations stemming from the block structure
can still be handled with the more direct Schur complement method as long as the non-commutativity
of the $L\times L$ structure matrices is properly taken into account. 
Utilizing a powerful \emph{tensor matrix structure} generated by the correlations between blocks 
resolves this issue automatically.

\paragraph{Acknowledgement} 
The authors are grateful to  David Renfrew  for several discussions and for calling their attention to references
on applications of non-Hermitian models.

\subsection{Notation} \label{subsec:notation}

Owing to the tensor product structure of Kronecker random matrices (see Definition \ref{def:kronecker_matrix} below), 
we need to introduce different spaces of matrices. In order to make the notation more transparent to the reader, 
we collect the conventions used on these spaces in this subsection. 

For $K, N \in \N$, we will consider the spaces $\C^{K\times K}$, $(\C^{K\times K})^N$ and $\C^{K\times K} \otimes \C^{N\times N}$, i.e., 
we consider $K\times K$ matrices, $N$-vectors of $K\times K$ matrices and $N\times N$ matrices with $K \times K$ matrices as entries. 
For brevity, we denote $\tenspace \defeq \C^{K\times K} \otimes \C^{N\times N}$. 
Elements of $\C^{K\times K}$ are usually denoted by small roman letters, elements of $(\C^{K\times K})^N$ by small boldface roman letters and 
elements of $\tenspace$ by capitalized boldface roman letters. 
 
For $\alpha \in \C^{K\times K}$, we denote by $\abs{\alpha}$ the matrix norm of $\alpha$ induced by the Euclidean distance on $\C^K$.
Moreover, we define two different norms on the $N$-vectors of $K\times K$ matrices. For any $\rf =(r_1, \ldots,r_N) \in (\C^{K\times K})^N$
we define  $\norm{\rf}  \defeq \max_{i=1}^N \abs{r_{i}}$,  
and   
\begin{equation} \label{eq:def_normhs_matrix_vector}
\normhs{\rf}^2 \defeq \frac{1}{NK} \sum_{i=1}^N \Tr(r_i^*r_i).
\end{equation}
 These  are the analogues of the maximum norm and the Euclidean norm for vectors in $\C^N$ which corresponds to $K=1$. 
Note that $\normhs{\rf} \leq \norm{\rf}$.   

For any function $f\colon U \to \C^{K\times K}$ from $U \subset \C^{K\times K}$ to $\C^{K\times K}$,
we lift $f$ to $U^N$ by defining $f(\rf) \in (\C^{K\times K})^N$ entrywise for any $\rf=(r_1, \ldots, r_N) \in U^N \subset (\C^{K\times K})^N$, i.e.,  
\begin{equation} \label{eq:def_function_vector_matrices}
 f(\rf) \defeq (f(r_1), \ldots, f(r_N)). 
\end{equation}
We will in particular apply this definition for $f$ being the matrix inversion map and the imaginary part. Moreover, for $\xf=(x_1, \ldots, x_N)$, $\yf = (y_1, \ldots, y_N) \in (\C^{K\times K})^N$ 
we introduce their entrywise product $\xf\yf \in (\C^{K\times K})^N$ through 
\begin{equation} \label{eq:def_product_vector_matrices}
 \xf\yf \defeq (x_1y_1, \ldots, x_Ny_N) \in (\C^{K\times K})^N. 
\end{equation}
Note that for $K \neq 1$, in general, $\xf \yf \neq \yf \xf$. 

If $a \in \C^{K\times K}$ or $\Af \in \tenspace$ are positive semidefinite matrices, then we write $a \geq 0$ or $\Af \geq 0$, respectively. 
Similarly, for $\af \in (\C^{K\times K})^N$, we write $\af\geq 0$ to indicate that all components of $\af$ are positive semidefinite matrices in $\C^{K\times K}$. 
The identity matrix in $\C^{K\times K}$ and $\tenspace$ is denoted by $\id$. 

We also use two norms on $\tenspace$. These are the operator norm $\normtwo{\genarg}$ induced by the Euclidean distance on $\C^{KN}\cong \C^K \otimes \C^N$ 
and the norm $\normhs{\genarg}$  induced  by the scalar product $\scalar{\genarg}{\genarg}$ on $\tenspace$ defined through 
\begin{equation}  \label{eq:def_normHS}
\scalar{\Rf}{\Tf} \defeq \frac{1}{NK}\Tr\left(\Rf^* \Tf\right), \qquad \normhs{\Rf}  \defeq \sqrt{\scalar{\Rf}{\Rf}}, 
\end{equation}
for $\Rf, \Tf \in \tenspace$. In particular, all orthogonality statements on $\tenspace$ are understood with respect to this 
scalar product. Furthermore, we introduce  $\avg{\Rf} \defeq \scalar{\id}{\Rf}$,   the normalized trace  for $\Rf \in \tenspace$.

We also consider linear maps on $(\C^{K\times K})^N$ and $\tenspace$, respectively. 
We follow the convention that the symbols $\Sf$, $\Lf$ and $\mathscr{T}$ label linear maps $(\C^{K\times K})^N \to (\C^{K\times K})^N$ 
and $\sopb$, $\lopb$ or $\topb$ denote linear maps $\tenspace\to\tenspace$.
The symbol $\Id$ refers to the identity map on $\tenspace$. 
For any linear map $\mathscr{T} \colon (\C^{K\times K})^N \to (\C^{K\times K})^N$, let $\norm{\mathscr{T}}$ denote the operator norm of $\mathscr{T}$ induced by~$\norm{\genarg}$ 
and let $\normsp{\mathscr{T}}$ denote the operator norm induced by $\normhs{\genarg}$. 
Similarly, for a linear map $\topb \colon \tenspace\to\tenspace$, we write $\norm{\topb}$ for the operator norm induced by $\normtwo{\genarg}$ on $\tenspace$ and  
$\normsp{\topb}$ for its operator norm induced by $\normhs{\genarg}$ on $\tenspace$. 

We use the notation $[n] \defeq \{1, \ldots, n\}$ for $n \in\N$. For $i,j \in [N]$, we introduce
 the matrix $E_{ij} \in \C^{N \times N}$ which has a one at its $(i,j)$ entry and only zeros otherwise, i.e., 
\begin{equation} \label{eq:def_E_ij}
 E_{ij} \defeq (\delta_{ik} \delta_{jl} )_{k,l=1}^N. 
\end{equation}
For $i,j \in [N]$, the linear map $P_{ij}\colon \tenspace \to \C^{K \times K}$ is defined through 
\begin{equation} \label{eq:def_P_ij}
 P_{ij} \Rf = r_{ij}, 
\end{equation}
for any $\Rf = \sum_{i,j=1}^N r_{ij} \otimes E_{ij} \in \tenspace$ with $r_{ij} \in \C^{K\times K}$.

\section{Main results}

Our main object of study are Kronecker random matrices which we define first. To that end, we recall the definition of $E_{ij}$ from~\eqref{eq:def_E_ij}.

\begin{defi}[Kronecker random matrix] \label{def:kronecker_matrix}
A random matrix $\Xf \in \C^{L\times L}\otimes \C^{N\times N}$ is called \emph{Kronecker random matrix} if
it is of the form
\begin{equation} \label{eq:def_Xf}
 \Xf = \sum_{\mu=1}^{\ell} \wt{\alpha}_\mu \otimes X_\mu + \sum_{\nu=1}^{\ell}(\wt{\beta}_\nu \otimes Y_\nu + \wt{\gamma}_\nu \otimes Y_\nu^*) + \sum_{i=1}^N \wt{a}_{i}\otimes E_{ii}, \qquad \ell\in\N,
\end{equation}
where $X_\mu = X_\mu^*\in\C^{N\times N}$ are Hermitian  random matrices
with centered independent entries (up to the Hermitian symmetry) and $Y_\nu\in\C^{N\times N}$ are random matrices 
with centered  independent entries; furthermore $X_1, \ldots, X_{\ell}, Y_1, \ldots, Y_{\ell}$ are independent.
The ``coefficient'' matrices $\wt{\alpha}_\mu, \wt{\beta}_\nu, \wt{\gamma}_\nu \in\C^{L\times L}$ are deterministic 
and they are called \emph{structure matrices}. Finally,  $\wt{a}_1, \ldots, \wt{a}_N \in \C^{L\times L}$
 are also deterministic. 
\end{defi}
We remark  that the number of $X_\mu$ and $Y_\nu$ matrices effectively present in $\Xf$ may differ by choosing some structure matrices zero.
Furthermore, note that $\E \Xf = \sum_{i=1}^N \wt{a}_i \otimes E_{ii}$, i.e., the deterministic
matrices $\wt{a}_i$ encode the expectation of $\Xf$.

Our main result asserts that  all eigenvalues of a Kronecker random matrix $\Xf$  
are contained in  the \emph{self-consistent $\eps$-pseudospectrum} for any  $\eps>0$,
with a very high probability if $N$ is sufficiently large.  
The self-consistent $\eps$-pseudospectrum, $\supnon_\eps$, is 
a deterministic subset of the complex plane that can be  defined and computed  via
the self-consistent solution  to the \emph{Hermitized Dyson equation.} 
Hermitization entails doubling the dimension and studying the matrix $\Hf^\zeta$ defined in \eqref{Hzeta def}
for any spectral parameter $\zeta\in \C$ associated with  $\Xf$. We introduce an additional spectral 
parameter $z\in \Hb\defeq \{ w \in \C \colon \Im w >0 \}$ that will be associated with the Hermitian matrix $\Hf^\zeta$.
The Hermitized Dyson equation is used to study the resolvent $(\Hf^\zeta-z)^{-1}$.

We first introduce some notation necessary to write up the Hermitized Dyson equation. 
For $\mu,\nu \in [\ell]$, we define
\begin{equation}
 \alpha_\mu \defeq \begin{pmatrix} 0 & 1\\ 0 & 0 \end{pmatrix} \otimes \wt{\alpha}_\mu + \begin{pmatrix} 0 & 0\\ 1 & 0 \end{pmatrix} \otimes \wt{\alpha}_\mu^*, 
\qquad \beta_\nu \defeq \begin{pmatrix} 0 & 1 \\ 0 & 0 \end{pmatrix} \otimes \big(\wt{\beta}_\nu + \wt{\gamma}_\nu^*\big) . \label{eq:mapping_non_hermitian_to_hermitian1}
\end{equation}
We set 
\begin{equation} \label{eq:def_variances}
s_{ij}^\mu \defeq  \E\, \abs{x_{ij}^\mu}^2, \qquad t_{ij}^\nu \defeq \E\, \abs{y_{ij}^\nu}^2,
\end{equation}
where $x_{ij}^\mu$ and $y_{ij}^\nu$ are the (scalar) entries of the random matrices $X_\mu$ and $Y_\nu$, respectively, i.e., $X_\mu=(x_{ij}^\mu)_{i,j=1}^N$ and $Y_\nu = (y_{ij}^\nu)_{i,j=1}^N$. 
 We define a linear map $\Sf$ on $(\C^{2 \times 2} \otimes \C^{L\times L})^N$, i.e., on $N$-vectors of  $(2L)\times (2L)$ 
matrices as follows. For any $\rf = (r_1, \ldots, r_N) \in (\C^{2 \times 2} \otimes \C^{L\times L})^N$ we set
$$ 
     \Sf[\rf] = \big( \Sf_1[\rf], \Sf_2[\rf], \ldots , \Sf_N[\rf]\big) \in (\C^{2 \times 2} \otimes \C^{L\times L})^N,
$$
where the $i$-th component is given by 
\begin{equation} \label{eq:sopb_l}
 \Sf_i[\rf] \defeq \sum_{k=1}^N\left( \sum_{\mu=1}^{\ell} s_{ik}^\mu \alpha_\mu r_k \alpha_\mu + \sum_{\nu=1}^{\ell} \left( t_{ik}^\nu \beta_\nu r_k \beta_\nu^* + t_{ki}^\nu \beta_\nu^* r_k \beta_\nu \right) \right) \in  \C^{2 \times 2} \otimes \C^{L\times L} , \qquad  i\in [N].  
\end{equation}
For $j \in [N]$ and $\zeta \in\C$, we define $a_j^\zeta \in \C^{2\times 2}\otimes \C^{L\times L}$ through
\begin{equation} \label{eq:mapping_non_hermitian_to_hermitian2}
 a_j^\zeta \defeq \begin{pmatrix} 0 & 1 \\ 0 & 0 \end{pmatrix} \otimes \wt{a}_j + \begin{pmatrix} 0 & 0 \\ 1 & 0 \end{pmatrix} \otimes \wt{a}_j^* - \begin{pmatrix} 0 & \zeta \\ \bar \zeta & 0 \end{pmatrix} \otimes \id. 
\end{equation}
 The Hermitized Dyson equation is the following system of equations
\begin{equation} \label{eq:Dyson_non_Hermitian}
- \frac{1}{m_j^\zeta(z)} = z\id  - a_j^\zeta + \Sf_j[\mf^\zeta(z)], \qquad  j=1,2,\ldots N, 
\end{equation}
for the vector 
$$
  \mf^\zeta(z)= \big( m_1^\zeta(z), \ldots , m_N^\zeta(z)\big) \in (\C^{2 \times 2} \otimes \C^{L\times L})^N.
$$
Here, $\id$ denotes the identity matrix in $\C^{2\times 2} \otimes \C^{L\times L}$ and $\zeta \in \C$ as well as $z\in \Hb$ are spectral parameters 
associated to $\Xf$ and $\Hf^\zeta$, respectively.

\begin{lem} \label{lem:Dyson_non_hermitian} For any $z\in\Hb$ and $\zeta\in \C$ there exists a unique
solution to \eqref{eq:Dyson_non_Hermitian} with the additional condition that  the matrices 
$\Im m_j^\zeta( z )\defeq \frac{1}{2i}(m_j^\zeta(z  ) - m_j^\zeta( z  )^*)$ are positive definite for all $j\in[N]$. 
Moreover, for $j\in[N]$, there are measures $v_j^\zeta$ on $\R$ with values in the positive semidefinite matrices in $\C^{2\times 2}\otimes \C^{L\times L}$ such that 
 \[ m_j^\zeta(z) = \int_\R \frac{v_j^\zeta(\di \tau)}{\tau-z} \]
for all $z \in \Hb$ and $\zeta \in \C$. 
\end{lem}

Lemma \ref{lem:Dyson_non_hermitian} is proven after Proposition \ref{pro:exist_unique_M} below. 
 Throughout the paper 
 $\mf^\zeta$ will always denote the  unique solution to the Hermitized Dyson equation 
  defined in Lemma \ref{lem:Dyson_non_hermitian}.    
The \emph{self-consistent density of states} $\rho^\zeta$ of $\Hf^\zeta$ is given by 
\[ \rho^\zeta(\di \tau) \defeq \frac{1}{2LN} \sum_{j=1}^N \Tr v_j^\zeta(\di\tau)\] 
(cf. Definition~\ref{def:self-consistent_density_of_states} below). 
The \emph{self-consistent spectrum} of $\Hf^\zeta$ is the set 
  $\supp \rho^\zeta = \bigcup_{j=1}^N \supp v_j^\zeta$.
  Finally, for any $\eps>0$ the \emph{self-consistent $\eps$-pseudospectrum of $\Xf$}
  is defined by
  \begin{equation} \label{eq:def_support_non_hermitian}
\supnon_\eps \defeq \big\{\zeta\in\C \colon \dist(0,\supp \rho^\zeta) \leq \eps \big\}.
\end{equation}
 The eigenvalues of $\Xf$ will concentrate on the set $\supnon_\eps$ for any fixed $\eps>0$ if $N$ is large.
 The motivation for this definiton \eqref{eq:def_support_non_hermitian} is that $\zeta$ is in the 
$\eps$-pseudospectrum of $\Xf$ 
if and only if  $0$ is in the $\eps$-vicinity of the 
  spectrum of $\Hf^\zeta$, i.e., $\dist(0, \spec(\Hf^\zeta))\le \eps$. 
We recall that the 
\emph{$\eps$-pseudospectrum} $\spec_{\eps}(\Xf)$ of $\Xf$ 
is defined through
\begin{equation} \label{eq:def_pseudospectrum}
 \spec_\eps(\Xf) \defeq \spec(\Xf) \cup \big\{ \zeta \in \C\setminus\spec(\Xf) \colon \normtwo{(\Xf-\zeta)^{-1}} \geq \eps^{-1}\big\}. 
\end{equation}
In accordance with Subsection \ref{subsec:notation}, $\normtwo{\cdot}$ denotes the operator norm on $\C^{L\times L}\otimes \C^{N\times N}$ induced by the Euclidean norm on $\C^{L}\otimes\C^{N}$. 

The precise statement is given in Theorem \ref{thm:no_eigenvalues_non_hermitian} below whose conditions we collect next.

\begin{assums} \label{assums:non_hermitian}
\begin{enumerate}[(i)]
\item (Upper bound on variances) There is $\kappa_1>0$ such that 
\begin{equation}
\label{eq:upper_bound_variances} 
s_{ij}^\mu \leq \frac{\kappa_1}{N}, \qquad t_{ij}^\nu \leq \frac{\kappa_1}{N}
\end{equation}
for all $i,j \in [N]$ and $\mu,\nu \in [\ell]$.
\item (Bounded moments) For each $p \in \N$, $p \geq 3$, there is $\varphi_p >0$ such that 
\begin{equation}
\label{eq:moments_bounds} 
\E \abs{x_{ij}^\mu}^p \leq \varphi_p N^{-p/2},\qquad  \E \abs{y_{ij}^\nu}^p \leq \varphi_pN^{-p/2}
\end{equation}
for all $i,j \in [N]$ and $\mu, \nu \in [\ell]$.
\item (Upper bound on structure matrices) There is $\kappa_2 >0$ such that
\begin{equation} \label{eq:bound_coefficients_non_hermitian} 
\max_{\mu \in[\ell]} \abs{\wt{\alpha}_{\mu}}\leq \kappa_2, \qquad \max_{\nu \in[\ell]} \abs{\wt{\beta}_{\nu}} \leq \kappa_2, 
\end{equation}
where $\abs{\alpha}$ denotes the operator norm induced by the Euclidean norm on $\C^L$. 
\item (Bounded expectation) Let $\kappa_3>0$ be such that the matrices $\wt{a}_i \in\C^{L\times L}$ satisfy
\begin{equation} \label{eq:A_two_bounded_non_hermitian}
\max_{i=1}^N \abs{\wt{a}_i} \leq \kappa_3.
\end{equation}
\end{enumerate}
\end{assums}
The constants $L$, $\ell$, 
$\kappa_1$, $\kappa_2$, $\kappa_3$ and $(\varphi_p)_{p\in \N}$ are called \emph{model parameters}. 
Our estimates will be uniform in all models possessing the same model parameters, in particular the bounds
will be uniform in $N$, the large parameter in our problem. Now we can formulate our main result:

\begin{thm}[All eigenvalues of $\Xf$ are inside self-consistent $\eps$-pseudospectrum] \label{thm:no_eigenvalues_non_hermitian}
Fix $L\in \N$. 
Let $\Xf$ be a  Kronecker  random matrix as in \eqref{eq:def_Xf} such that 
the bounds \eqref{eq:upper_bound_variances}~--~\eqref{eq:A_two_bounded_non_hermitian} are satisfied.

Then for each $\eps>0$ and $D>0$, there is a constant $C_{\eps,D} >0$ such that 
\begin{equation} \label{eq:no_eigenvalues_non_hermitian}
\P \big( \spec(\Xf) \subset \supnon_\eps \big) \geq 1- \frac{C_{\eps,D}}{N^D}.
\end{equation}
The constant $C_{\eps, D}$ in \eqref{eq:no_eigenvalues_non_hermitian} only depends on the model parameters in addition to $\eps$ and $D$. 
\end{thm}

\begin{rem}
\begin{enumerate}[(i)]
\item 
Theorem~\ref{thm:no_eigenvalues_non_hermitian} follows from the slightly stronger Lemma \ref{lem:pseudospectrum_subset_self_consistent} below; we show that not only the spectrum of $\Xf$ but also its $\eps/2$-pseudospectrum 
lies in the self-consistent $\eps$-pseudospectrum.
\item By carefully following the proof of Lemma~\ref{lem:pseudospectrum_subset_self_consistent}, one can see that $\eps$ can be replaced by $N^{-\delta}$ with a small universal constant $\delta >0$.
The constant $C$ in \eqref{eq:no_eigenvalues_non_hermitian} will depend only on $D$ and the model parameters. 
\item (Only finitely many moments)
If \eqref{eq:moments_bounds} holds true only for $p \leq P$ and some $P \in \N$ then there is a $D_0(P) \in \N$ such that 
the bound \eqref{eq:no_eigenvalues_non_hermitian} is valid for all $D \leq D_0(P)$.  
\item 
The self-consistent $\eps$-pseudospectrum $\supnon_\eps$ from \eqref{eq:def_support_non_hermitian} is defined in terms of the support of the self-consistent density of states of the Hermitized 
Dyson equation \eqref{eq:Dyson_non_Hermitian}. In particular, to determine $\supnon_\eps$ one needs to solve the Dyson equation for spectral parameters $z$ in a neighborhood of $z=0$. 
There is an alternative definition for a deterministic $\eps$-regularized set that is comparable to $\supnon_\eps$ and requires to solve the Dyson equation solely on the imaginary axis $z=\ii\eta$, namely 
\bels{def of wt D eps}{
\wt{\supnon}_\eps\,&=\, \cB{\zeta: \limsup_{\eta \downarrow 0}\frac{1}{\eta} \max_j\abs{\im m_j^\zeta(\ii \eta)} 
\geq \frac{1}{\eps}}\,.
}
Hence, \eqref{eq:no_eigenvalues_non_hermitian} is true if $\supnon_\eps$ is replaced by $\wt{\supnon}_\eps$. 
For more details we refer the reader to Appendix \ref{app:alt_def_pseudo}. 
\item (Hermitian matrices)  If $\Xf$ is  a Hermitian random matrix, $\Xf=\Xf^*$, i.e., $\wt{\alpha}_\mu = \wt{\alpha}_\mu^*$ and $\wt{\beta}_\nu^*=\wt{\gamma}_\nu$ 
for all $\mu,\nu\in[\ell]$ and $\wt{a}_i^*=\wt{a}_i$ for all $i \in [N]$,  then the Hermitization is superfluous
and the Dyson equation may be formulated  directly for $\Xf$. 
One may easily show that the support of the self-consistent density of states $\rho$  
is the intersection of all self-consistent $\eps$-pseudospectra: 
\[ \supmeas = \bigcap_{\eps>0} \supnon_\eps. \]
\item Theorem \ref{thm:no_eigenvalues_non_hermitian} as well as its stronger version for the Hermitian case, 
Theorem~\ref{thm:no_eigenvalues}, 
 identify a deterministic superset of the spectrum of $\Xf$.  In fact, it is expected that for a large class
 of Kronecker matrices  the set $\bigcap_{\eps>0} \supnon_\eps$ is the smallest deterministic set
 that still contains the entire $\spec(\Xf)$ up to a negligible distance. For $L=1$ this  has been proven 
 for many Hermitian ensembles  and for the  circular ensemble.
 Example~\ref{numerics}  below presents numerics for the $L\ge 2$ case.
\end{enumerate}
\end{rem}

\begin{exam}\label{numerics}
Fix $L \in \N$. Let $\zeta_1, \ldots, \zeta_L \in \C$ 
and $a \in \C^{L\times L}$ denote the diagonal matrix with $\zeta_1, \ldots, \zeta_L$ on its diagonal. 
We set $\Xf \defeq  a\otimes \id + \Wf$, where $\Wf$  has centered i.i.d. entries with  variance $1/(NL)$.
Clearly, $\Xf$ is a Kronecker matrix. 
In this case  the Dyson equation can be directly solved  and one easily  finds that 
\begin{equation}\label{set}
 \bigcap_{\eps>0} \supnon_\eps = \Big\{ \zeta \in \C \colon \sum_{i=1}^L \frac{1}{\abs{\zeta_i - \zeta}^2} \geq  L \Big\} 
 \end{equation}
(To our knowledge, the formula on the r.h.s. first appeared in \cite{Khoruzhenko1996}).
Figure \ref{fig:figure_numerics} shows the set \eqref{set} and the actual eigenvalues of $\Xf$ for  $N=8000$ and different matrices $a$.
\end{exam}

\newlength\fwidth
\begin{figure}[ht!]
\begin{subfigure}{.5\textwidth}
\centering
\includegraphics[width=\textwidth]{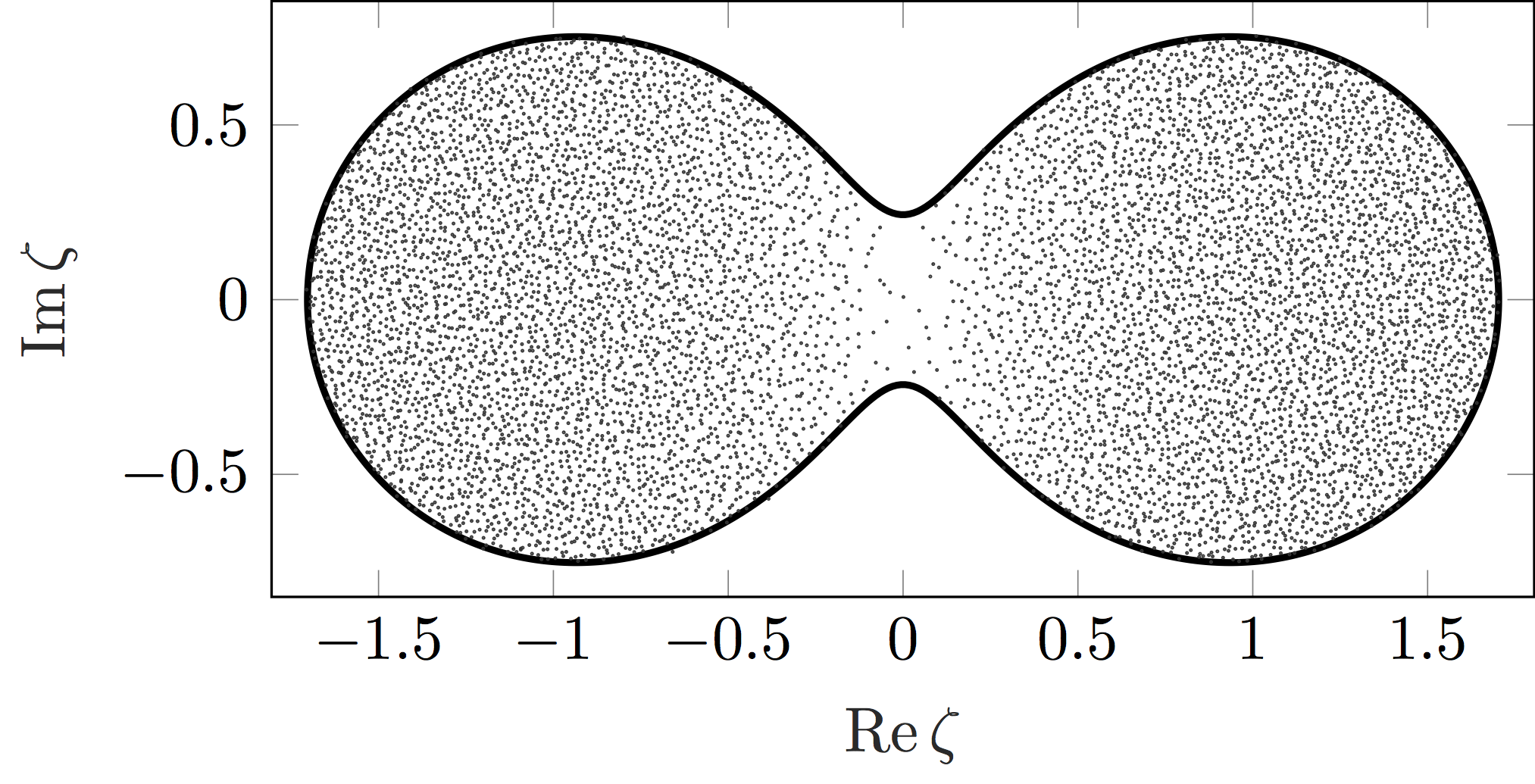}
\caption{$\{\zeta_1, \zeta_2\} = \{ \pm 0.97\}$}
\end{subfigure}
\begin{subfigure}{.5\textwidth}
\centering
\includegraphics[width=\textwidth]{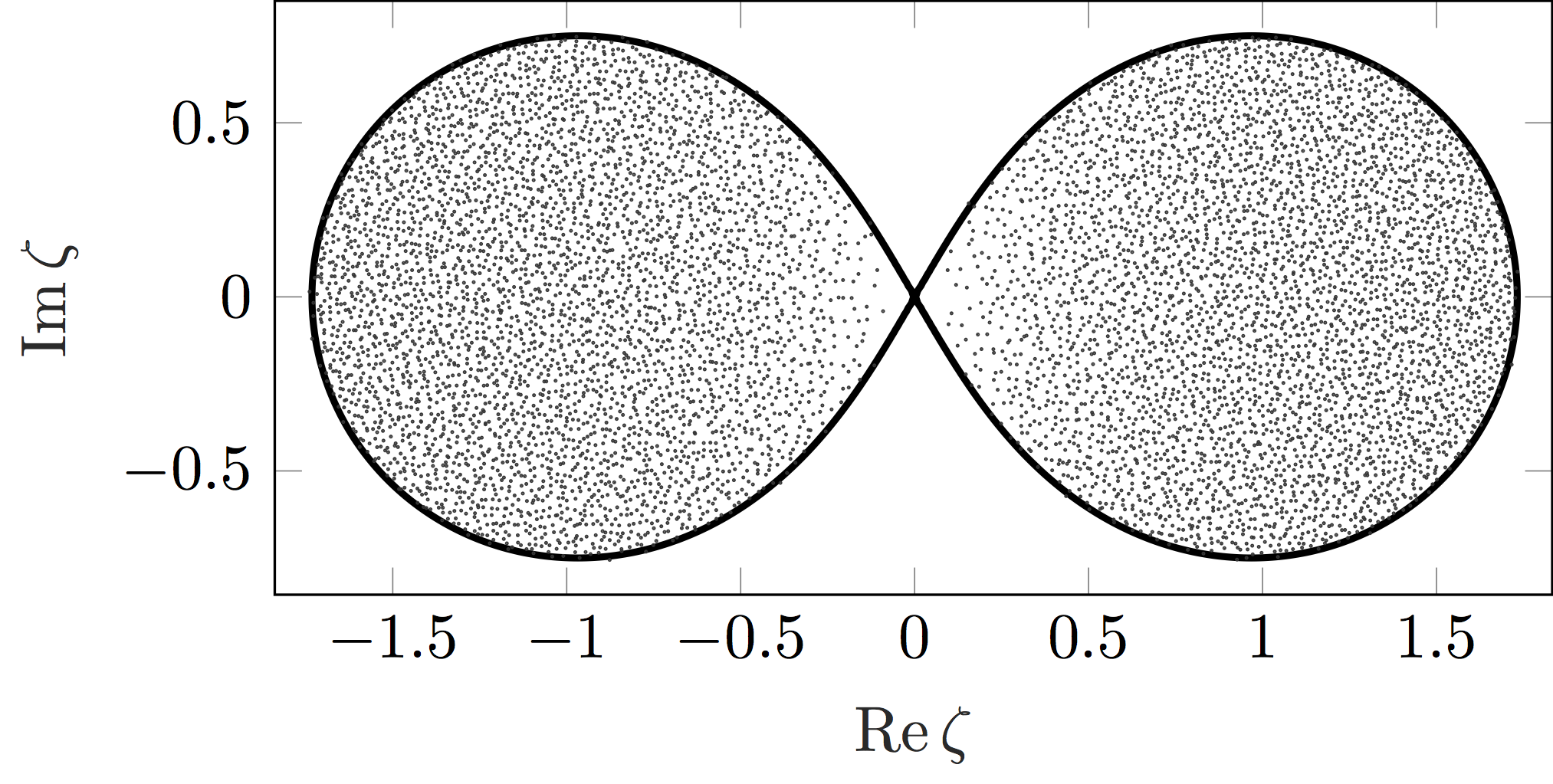}
\caption{$\{\zeta_1, \zeta_2\} = \{ \pm 1.0\}$}
\end{subfigure}

\vspace{0.5cm}

\begin{subfigure}{.5\textwidth}
\centering
\includegraphics[width=\textwidth]{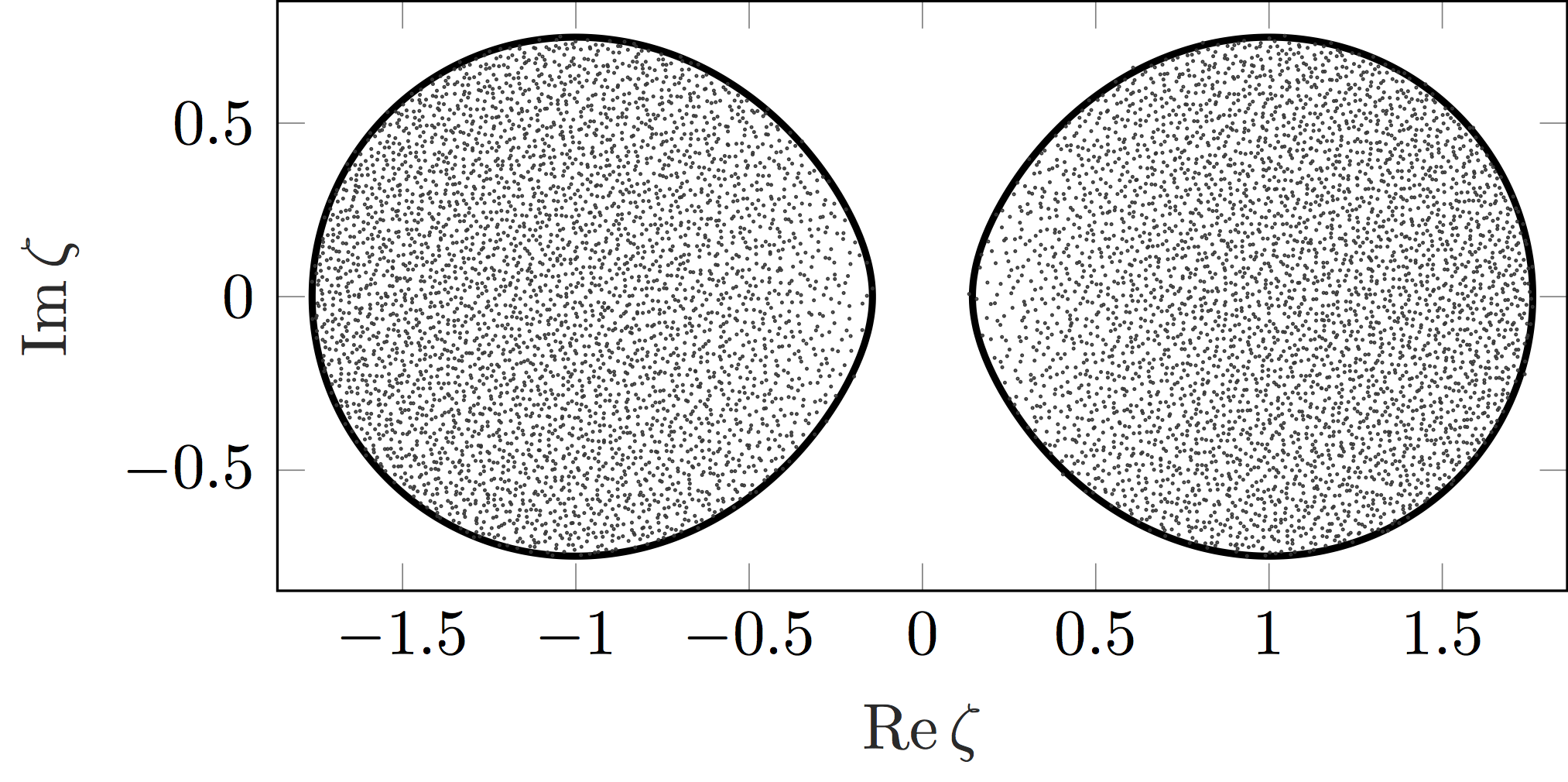}
\caption{$\{\zeta_1, \zeta_2\} = \{ \pm 1.03\}$}
\end{subfigure}
\begin{subfigure}{.5\textwidth}
\centering
\includegraphics[width=\textwidth]{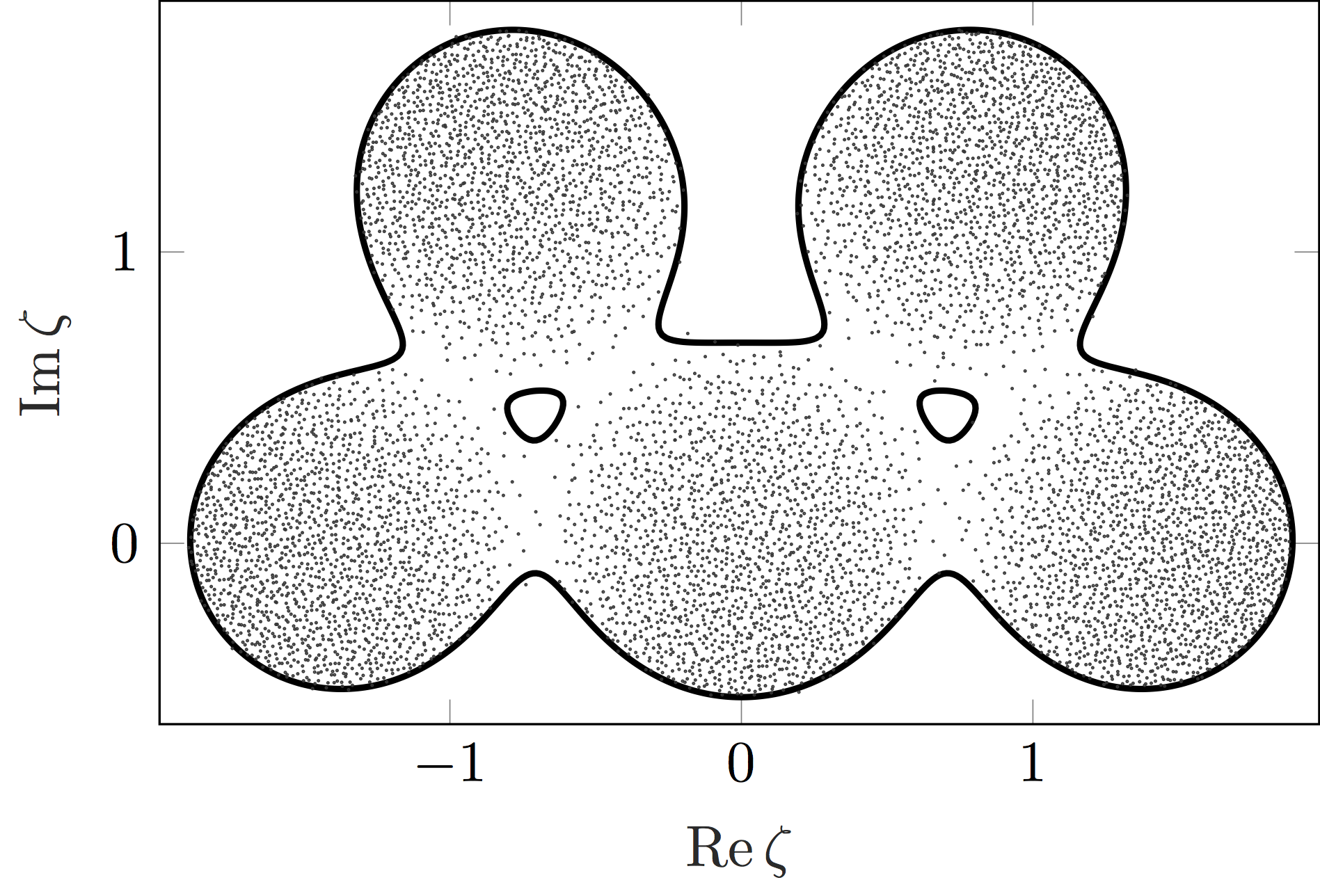}
\caption{$\{\zeta_1,\ldots, \zeta_5\} = \{0, \pm 1.4, \pm 0.8 + \ii 1.26\}$}
\end{subfigure}
\caption{Eigenvalues of sample random matrix with $N=8000$ and $\cap_{\eps>0} \supnon_\eps$.}
\label{fig:figure_numerics}
\end{figure}

The \emph{empirical density of states} of a Hermitian matrix $\Hf \in \C^{L\times L} \otimes \C^{N\times N}$  
is defined through 
\begin{equation} \label{eq:def_empirical_density_of_states}
\mu_{\Hf}(\di \tau) \defeq  \frac{1}{NL} \sum_{\lambda \in\spec(\Hf)} \delta_\lambda(\di\tau).
\end{equation}
\begin{thm}[Global law for Hermitian Kronecker matrices] \label{thm:global_law}
Fix $L\in \N$. 
For $N \in \N$, let $\Hf_N \in \C^{L\times L} \otimes \C^{N\times N}$ be a  Hermitian Kronecker  random matrix 
as in \eqref{eq:def_Xf} such that 
the bounds \eqref{eq:upper_bound_variances}~--~\eqref{eq:A_two_bounded_non_hermitian} are satisfied.
Then there exists a sequence of deterministic probability 
measures $\rho_N$ on $\R$ 
such that the difference of $\rho_N$ and the empirical spectral measure $\mu_{\Hf_N}$,
defined in \eqref{eq:def_empirical_density_of_states}, of $\Hf_N$ 
converges to zero weakly in probability, i.e., 
\begin{equation} \label{eq:global_law}
 \lim_{N\to\infty} \int_\R f(\tau)(\mu_{\Hf_N}- \rho_N)(\di\tau) = 0 
\end{equation}
for all $f \in C_0(\R)$ in probability. Here, $C_0(\R)$ denotes the continuous functions on $\R$
vanishing at infinity. 

Furthermore, there is a compact subset of $\R$ which contains the supports of all $\rho_N$. 
This compact set depends only on the model parameters. 
\end{thm}

Theorem~\ref{thm:global_law} is proven in Appendix \ref{app:proof_local_law}. 
The measure $\rho_N$, the self-consistent density of states, can be obtained by solving the 
corresponding Dyson equation, see Definition \ref{def:self-consistent_density_of_states} later.
If the function $f$ is sufficiently regular then our proof combined with the Helffer-Sjöstrand formula 
yields an effective convergence rate of order $N^{-\delta}$ in \eqref{eq:global_law}.

\section{Solution and stability of the Dyson equation} \label{sec:Dyson_equation}

The general matrix Dyson equation (MDE) has been extensively studied in 
\cite{AjankiCorrelated}, but under conditions that exclude general Kronecker random matrices. 
Here, we relax these conditions and show how to extend some key results of \cite{AjankiCorrelated} 
to our current setup. 
Our analysis of the MDE on the space of $n\times n$ matrices, $\tenspace = \C^{n\times n}$, will then be applied 
to \eqref{eq:Dyson_non_Hermitian} with $n=2LN=KN$. 
On $\tenspace = \C^{n\times n}$, we use the norms as defined in Subsection \ref{subsec:notation} and 
 require the pair $(\Af, \sopb)$ to have the following properties:
\begin{definition}[Data pair] \label{def:Data pair} We call $(\Af, \sopb)$ a \emph{data pair} if 
\begin{itemize}
\item The imaginary part $\im \Af = \frac{1}{2\ii}(\Af-\Af^*)$ of the matrix $\Af \in \C^{n \times n}$  is negative semidefinite.
\item The linear operator $\sopb: \C^{n\times n} \to \C^{n\times n}$ is self-adjoint with respect to the scalar product
\[
\scalar{\bs{R}}{\bs{T}}\,\defeq \,  \frac{1}{n} \Tr[\bs{R}^* \bs{T}]\,,
\]
and preserves the cone of positive semidefinite matrices, i.e. it is positivity preserving. 
\end{itemize}
\end{definition}

For any data pair $(\Af, \sopb)$,
the MDE then takes the form
\begin{equation} \label{eq:Dyson_matrix}
- \Mf^{-1}(z) = z\id  - \Af + \sopb[\Mf(z)], \quad z \in \Hb,
\end{equation}
for a solution matrix $\Mf(z) \in \C^{n \times n}$. It was shown in this generality that 
the MDE, \eqref{eq:Dyson_matrix}, has a unique solution under the constraint that the imaginary part 
$\Im \Mf(z) \defeq (\Mf(z)-\Mf(z)^*)/(2\ii)$ is positive definite~\cite{Helton01012007}. 
We remark that $\Im \Af$ being negative semidefinite is the most general condition for which our analysis is applicable. 
Furthermore, in \cite{AjankiCorrelated}, properties of the solution of \eqref{eq:Dyson_matrix} and the stability of \eqref{eq:Dyson_matrix} against small perturbations 
were studied in the general setup with Hermitian $\bs{A}$ and under the so-called \emph{flatness} assumption, 
\bels{matrix flatness}{
 \frac{c}{n} \Tr(\Rf) \leq \sopb[\Rf] \leq \frac{C}{n} \Tr(\Rf)\,, 
 }
 for all positive definite $\Rf \in\C^{n \times n}$ with some constants $C > c>0$. Within Section~\ref{sec:Dyson_equation} we will generalize certain results from \cite{AjankiCorrelated} by dropping the flatness assumption \eqref{matrix flatness} and the Hermiticity of $\bs{A}$. 
The results in this section, apart from \eqref{eq:matrix support_A} below, 
follow by combining and modifying several arguments from \cite{AjankiCorrelated}.
We will only explain the main steps and refer to \cite{AjankiCorrelated} for details.
At the end of the section we translate these general results back to the setup of Kronecker matrices with the associated Dyson equation~\eqref{eq:Dyson_non_Hermitian}. 

\subsection{Solution of the Dyson equation}
\label{subsec:Solution of the Dyson equation}

According to Proposition~2.1 in \cite{AjankiCorrelated} the solution $\Mf$ to \eqref{eq:Dyson_matrix} has a Stieltjes transform representation 
\bels{matrix Stieltjes transform rep}{
\Mf(z) \,=\, \int_\R \frac{\bs{V}(\dd \tau)}{\tau-z}\,, \qquad z \in \mathbb{H}\,,
}
where $\bs{V}$ is a compactly supported measure on $\R$ with values in positive semidefinite $n \times n$-matrices such that $\bs{V}(\R)=\id$, provided $\Af$ is Hermitian. The following lemma strengthens the conclusion about the support properties for this measure compared to Proposition~2.1 in \cite{AjankiCorrelated}.

\begin{lemma} \label{lmm:matrix support of V} Let $(\Af,\sopb)$ be a data pair as in Definition~\ref{def:Data pair} and $\Mf: \mathbb{H} \to \C^{n \times n}$ be the unique solution to \eqref{eq:Dyson_matrix} with positive definite imaginary part. Then
\begin{enumerate}
\item[(i)] There is a unique measure $\bs{V}$ on $\R$ with values in positive semidefinite matrices and $\bs{V}(\R)=\id$ such that \eqref{matrix Stieltjes transform rep} holds true.
\item[(ii)] If $\Af$ is Hermitian, then 
\begin{subequations} 
\begin{align}
\label{eq:matrix support_V}
\supp \bs{V} & \, \subset \, \spec \bs{A} + [-2 \norm{\sopb}^{1/2}, 2 \norm{\sopb}^{1/2}], \\
\label{eq:matrix support_A}
\spec \bs{A} & \, \subset \, \supp \bs{V} + [-\norm{\sopb}^{1/2}, \norm{\sopb}^{1/2}].
\end{align}
\end{subequations}
\end{enumerate}
\end{lemma}

\begin{proof}[Proof of Lemma \ref{lmm:matrix support of V}] The representation \eqref{matrix Stieltjes transform rep} follows exactly as in the proof of Proposition~2.1 in \cite{AjankiCorrelated}
even for $\Af$ with negative semidefinite imaginary part.
We now prove \eqref{eq:matrix support_V} motivated by the same proof in \cite{AjankiCorrelated}. For a matrix $\Rf \in \C^{n \times n}$, its smallest singular value is denoted by $\smin(\Rf)$. 
Note that $\smin(z-\Af) = \dist(z, \spec \Af)$ since $\Af$ is Hermitian. In the following, we fix $z \in\Hb$ such that $\dist(z,\spec\Af) = \smin(z-\Af) > 2\norm{\sopb}^{1/2}$.

Under the condition $\normtwo{\Mf(z)} \leq \smin(z-\Af)/(2\norm{\sopb})$, we obtain from \eqref{eq:Dyson_matrix} 
\begin{equation} \label{eq:local_spectrum_aux1}
\normtwo{\Mf(z)}= \frac{1}{\smin(z-\Af + \sopb[\Mf(z)])} \leq \frac{1}{\smin(z-\Af) - \norm{\sopb}\normtwo{\Mf(z)}} 
\leq \frac{2}{\dist(z,\spec\Af)}.   
\end{equation}
Therefore, using $\smin(z-\Af) > 2\norm{\sopb}^{1/2}$, we find a gap in the values $\normtwo{\Mf(z)}$ can achieve
\[ \normtwo{\Mf(z)} \notin \Big(\frac{2}{\smin(z-\Af)}, \frac{\smin(z-\Af)}{2\norm{\sopb}}   \Big). \]
For large values of $\eta= \Im z$, $\normtwo{\Mf(z)}$ is smaller than the lower bound of this interval. 
Thus, since $\normtwo{\Mf(z)}$ is a continuous function of $z$ and the set $\{ w\in\Hb\colon \dist(w, \spec \Af ) > 2\norm{\sopb}^{1/2}\}$ is path-connected,
we conclude that \eqref{eq:local_spectrum_aux1} holds true for all $z \in \Hb$ satisfying $\dist(z,\spec \Af)> 2\norm{\sopb}^{1/2}$. 

We take the imaginary part of \eqref{eq:Dyson_matrix} and use $\Af=\Af^*$ to obtain $\Im \Mf = \eta \Mf^* \Mf + \Mf^* \sopb[\Im\Mf] \Mf.$
Solving this relation for $\Im \Mf$ and estimating its norm yields 
\[ \normtwo{\Im \Mf} \leq \frac{\eta\normtwo{\Mf}^2}{1- \norm{\sopb} \normtwo{\Mf}^2} \leq \frac{4\eta}{\dist(z,\spec\Af)^2 - 4 \norm{\sopb}}. \]
Here, we employed $\normtwo{\Mf}^2 \norm{\sopb} <1$ by \eqref{eq:local_spectrum_aux1} and $\dist(z,\spec \Af)> 2\norm{\sopb}^{1/2}$.
Hence, $\Im \Mf$ converges to zero locally uniformly on the set $\big\{ z\in\Hb\colon \dist(z, \spec\Af) > 2 \norm{\sopb}^{1/2}\big\}$ for $\eta \downarrow 0$. 
Therefore, $E \notin \supp \bs{V}$ if $\dist(E, \spec\Af) > 2 \norm{\sopb}^{1/2}$. 
This concludes the proof of~\eqref{eq:matrix support_V}.

We now prove \eqref{eq:matrix support_A}.
From \eqref{eq:Dyson_matrix}, we obtain 
\begin{equation} \label{eq:spec_A_supp_rho_aux1}
 \Af-z\id = \Mf^{-1}( \id + \Mf \sopb[\Mf]) 
\end{equation}
for $z \in\Hb$. Since $\bs{V}(\R) = 1$, we have
\begin{equation} \label{eq:spec_A_supp_rho_aux2}
 \normtwo{\Mf} \leq \frac{1}{\dist(z,\supp \bs{V})}. 
\end{equation}
Therefore, taking the inverse in \eqref{eq:spec_A_supp_rho_aux1} and applying \eqref{eq:spec_A_supp_rho_aux2} yield 
\begin{equation} \label{eq:spec_A_supp_rho_aux3}
 \normtwo{(\Af-z\id)^{-1}} \leq \frac{1}{\dist(z,\supp \bs{V})(1-\norm{\sopb}\dist(z, \supp \bs{V})^{-2})} 
\end{equation}
for all $z \in \Hb$ satisfying $\dist(z, \supp\bs{V})^2 > \norm{\sopb}$. Taking $\Im z \downarrow 0$ in \eqref{eq:spec_A_supp_rho_aux3}, we see that the matrix $\Af - E\id$ 
is invertible for all $E \in\R$ satisfying $\dist(E, \supp\bs{V})^2 > \norm{\sopb}$, showing~\eqref{eq:matrix support_A}.
\end{proof}

In accordance with Definition~2.3 in \cite{AjankiCorrelated} we define the self-consistent density of states as the unique measure whose Stieltjes transform is $n^{-1} \Tr \Mf$.

\begin{defi}[Self-consistent density of states] \label{def:self-consistent_density_of_states}
The measure 
\begin{equation} \label{eq:matrix def_self_cons_dens_states}
 \rho(\di \tau) \defeq \frac{1}{n} \Tr \Vf(\di \tau)\,=\, \avg{\Vf(\di \tau)}
\end{equation}
is called the \emph{self-consistent density of states}. Clearly, $\supmeas = \supp \Vf$.  For the following lemma, we also 
 define the harmonic extension of the self-consistent density of states $\rho\colon \Hb \to \R_+$ through
\begin{equation} \label{eq:def_rho}
 \rho (z) \defeq \frac{1}{\pi}\avg{\Im \Mf(z)}. 
\end{equation}

\end{defi}

In the following we will use the short hand notation 
\[\disvz \defeq \dist(z,\supmeas)\,.\]

\begin{lem}[Bounds on $\Mf$ and $\Mf^{-1}$] \label{lem:bounds_Mf}
Let $(\Af, \sopb)$ be a data pair as in Definition \ref{def:Data pair}. 
\begin{enumerate}[(i)]
\item For $z\in\Hb$, we have the bounds
\begin{subequations} 
\begin{align}
\normtwo{\Mf} \leq &\, \frac{1}{\disvz},  \label{eq:Mf_upper_bound_eta} \\
(\Im z)\normtwo{\Mf^{-1}}^{-2} \id \leq \Im \Mf \leq  &\,\frac{\Im z}{\disvz[2]} \id, \label{eq:Im_Mf_upper_bound_away_support}  \\
 \normtwo{\Mf^{-1}} \le & \, \abs{z} + \normtwo{\Af} +  \norm{\sopb}\normtwo{\Mf}.  \label{eq:bound_Mf_inverse}
\end{align}
\end{subequations}
\item For $z\in\Hb$, we have the bound
\begin{equation} \label{eq:bound_rho_dist_support}
 \rho(z) \leq \frac{\Im z}{\pi\disvz[2]}.
\end{equation}
\end{enumerate}
\end{lem}

\begin{proof}
Using \eqref{matrix Stieltjes transform rep} immediately 
yields \eqref{eq:Mf_upper_bound_eta} and the upper bound in \eqref{eq:Im_Mf_upper_bound_away_support} since  $\bs{V}(\R) = \id$. 
With $\eta = \Im z$ and taking the imaginary part of \eqref{eq:Dyson_matrix}, we obtain 
\[ \Im \Mf = \eta \Mf^* \Mf - \Mf^*(\Im \Af)\Mf+ \Mf^* \sopb[\Im \Mf] \Mf \geq \eta \Mf^* \Mf\]
as $\Im \Af \leq 0$, $\Im \Mf\geq 0$ and $\sopb$ is positivity preserving. Since $\Rf^*  \Rf \geq \normtwo{\Rf^{-1}}^{-2}\id$ for any $\Rf \in\C^{n \times n}$ 
the lower bound in \eqref{eq:Im_Mf_upper_bound_away_support} follows. 
From \eqref{eq:Dyson_matrix}, we obtain \eqref{eq:bound_Mf_inverse}. Since $\rho(z) = \pi^{-1} \avg{\Im \Mf(z)}$ the upper bound in \eqref{eq:Im_Mf_upper_bound_away_support} implies~\eqref{eq:bound_rho_dist_support}. 
\end{proof}

\subsection{Stability of the Dyson equation} \label{subsec:stability_mde} 

The goal of studying the stability of the Dyson equation in matrix form, \eqref{eq:Dyson_matrix}, is to show that if some $\Gf$ satisfies 
\begin{equation} \label{eq:resolvent_perturbed_Dyson}
-\id = (z\id - \Af + \sopb[\Gf])\Gf + \Df
\end{equation}
for some small $\Df$, then  $\Gf$ is close to $\Mf$.
It turns out that to a large extent this is a question about the invertibility of the stability operator $\lopb \defeq \Id - \Mf \sopb[\genarg] \Mf$ acting on $\C^{n \times n}$.
From \eqref{eq:Dyson_matrix} and \eqref{eq:resolvent_perturbed_Dyson}, we obtain the following equation
\begin{equation} \label{eq:stability_equation}
 \lopb[\Gf- \Mf] = \Mf\Df + \Mf \sopb[\Gf-\Mf](\Gf-\Mf) 
\end{equation}
relating the difference $\Gf-\Mf$ with $\Df$. 
We will call \eqref{eq:stability_equation} the \emph{stability equation}.
Under the assumption that $\Gf$ is not too far from $\Mf$, the question whether $\Gf - \Mf$ is comparable with $\Df$ is determined by the invertibility of $\lopb$ in \eqref{eq:stability_equation} and the boundedness of 
the inverse.

In this subsection, we show that $\norm{\lopb^{-1}}$ is bounded, provided $\dist(z,\supp\bs{V})$ is bounded away from zero.  
In order to prove this bound on $\lopb^{-1}$, we follow the symmetrization procedure for $\lopb$ introduced in \cite{AjankiCorrelated}.
We introduce the operators $\mC_\Rf\colon \C^{n \times n} \to \C^{n \times n}$ and $\mF \colon \C^{n \times n} \to \C^{n \times n}$ through
\[ \mC_\Rf[\Qf] = \Rf\Qf\Rf,\qquad \qquad \mF \defeq \mC_\Wf \mC_{\sqrt{\Im \Mf}} \sopb  \mC_{\sqrt{\Im \Mf}}\mC_\Wf, \]
for $\Qf \in \C^{n\times n}$. 
Furthermore, the matrix $\Tf \in \C^{n\times n}$, the unitary matrix $\Uf \in\C^{n \times n}$ and the positive definite matrix $\Wf\in\C^{n \times n}$ are defined through
\[ \Tf \defeq \mC_{\sqrt{\Im \Mf}}^{-1} [ \Re \Mf] - \ii \id, \qquad \Uf \defeq \frac{\Tf}{\abs{\Tf}},  \qquad \Wf \defeq \abs{\Tf}^{1/2}.  \]
With these notations, a direct calculation yields 
\begin{equation}\label{eq:decomposition_Lf}
 \lopb = \Id - \mC_\Mf \sopb = \mC_{\sqrt{\Im \Mf}} \mC_\Wf \mC_{\Uf^*} \big( \mC_\Uf - \mF\big) \mC_\Wf^{-1} \mC_{\sqrt{\Im \Mf}}^{-1}, 
\end{equation}
as in (4.39) of \cite{AjankiCorrelated}.

We remark that $\mC_\Rf$ for $\Rf\in \C^{n \times n}$ is invertible if and only if $\Rf$ is invertible and $\mC_\Rf^{-1} = \mC_{\Rf^{-1}}$ in this case.  
Similarly, $\mC_\Rf^* = \mC_{\Rf^*}$. 

Our goal is to verify $\normsp{\mF} \leq 1- c$ for some positive constant $c$ which yields $\normsp{(\mC_\Uf - \mF)^{-1}} \leq c^{-1}$ as $\normsp{\mC_\Uf} = 1$. 
Then the boundedness of the other factors in \eqref{eq:decomposition_Lf} implies the bound on the inverse of the stability operator $\lopb$.

\begin{convention}[Comparison relation] \label{conv:matrix_Comparison_Relation}
For nonnegative scalars or vectors $f$ and $g$, we will use the notation $f \lesssim g$ if there is a constant $c>0$, depending only on 
$\norm{\sopb}_{\rm{hs}\to\norm{\genarg}} $ such that $f \leq cg$ and 
$f \sim g$ if $f \lesssim g$ and $f \gtrsim g$ both hold true. 
If the constant $c$ depends on an additional parameter (e.g. $\eps>0$), then we will indicate this dependence by a subscript (e.g. $\lesssim_\eps$). 
\end{convention}

\begin{lem} \label{lem:properties_F_operator}
Let $(\Af, \sopb)$ be a data pair as in Definition \ref{def:Data pair}. 
\begin{enumerate}[(i)]
\item Uniformly for any $z \in \Hb$, we have 
\begin{equation}\label{eq:lower_bound_Wf}
 \disvz[4] \normtwo{\Mf^{-1}}^{-2} \id \lesssim \Wf^4 (\Im z)^2  \lesssim \normtwo{\Mf}^2\normtwo{\Mf^{-1}}^4\id.
\end{equation}
\item 
There is a positive semidefinite $\Ff \in\C^{n \times n}$ such that $\normhs{\Ff} = 1$ and 
 $\mF[\Ff]= \normsp{\mF} \Ff$. Moreover, 
\begin{equation}\label{eq:norm_mF}
 1 -\normsp{\mF} = (\Im z)\frac{\scalar{\Ff}{\mC_{\Wf}[\Im \Mf]}}{\scalar{\Ff}{\Wf^{-2}}} .
\end{equation}
\item Uniformly for $z \in \Hb$, we have
\begin{equation} \label{eq:norm_mF_estimate}
 1 -\normsp{\mF} \gtrsim \disvz[4] \normtwo{\Mf^{-1}}^{-4}.  
\end{equation}
\end{enumerate}
\end{lem}

The proof of this lemma is motivated by the proofs of Lemma 4.6 and Lemma 4.7 (i) in \cite{AjankiCorrelated}.

\begin{proof}
We set $\eta \defeq \Im z$.
We rewrite the definition of $\Wf$ and use the upper bound in \eqref{eq:Im_Mf_upper_bound_away_support} to obtain 
\begin{align*}
 \Wf^4 = \mC^{-1}_{\sqrt{\Im \Mf}}(\mC_{\Im \Mf}+\mC_{\Re \Mf}) [(\Im \Mf)^{-1}] \geq &\,\eta^{-1} \disvz[2] \mC^{-1}_{\sqrt{\Im \Mf}}[\Mf\Mf^* + \Mf^*\Mf] \\
\gtrsim & \, \normtwo{\Mf^{-1}}^{-2}  \eta^{-2} \disvz[4]\id. 
\end{align*}
Here, we also applied $\Mf\Mf^* + \Mf^*\Mf \geq 2 \normtwo{\Mf^{-1}}^{-2} \id$ and the upper bound in \eqref{eq:Im_Mf_upper_bound_away_support} again. 
This proves the lower bound in \eqref{eq:lower_bound_Wf}. Similarly, using $\Mf\Mf^* + \Mf^*\Mf \leq 2 \normtwo{\Mf}^2 \id$ and the lower bound in \eqref{eq:Im_Mf_upper_bound_away_support}
we obtain the upper bound in \eqref{eq:lower_bound_Wf}.

For the proof of (ii), we remark that $\mF$ preserves the cone of positive semidefinite matrices. Thus, by a version of the Perron-Frobenius theorem 
of cone preserving operators there is a positive semidefinite $\Ff$ such that $\normhs{\Ff} = 1$ and $\mF\Ff= \normsp{\mF}\Ff$. 
Following the proof of (4.24) in \cite{AjankiCorrelated} and noting that this proof uses neither the uniqueness
of $\Ff$ nor its positive definiteness, we obtain \eqref{eq:norm_mF}.

The bound in \eqref{eq:norm_mF_estimate} is obtained by plugging the lower bound in \eqref{eq:lower_bound_Wf} and the lower bound in \eqref{eq:Im_Mf_upper_bound_away_support} into \eqref{eq:norm_mF}. 
We start by estimating the numerator in \eqref{eq:norm_mF}. Using $\Ff\geq 0$, the cyclicity of the trace, \eqref{eq:Im_Mf_upper_bound_away_support} and the lower bound in \eqref{eq:lower_bound_Wf}, we get
\begin{equation} \label{eq:bound_numerator}
 \scalar{\Ff}{\mC_\Wf[\Im \Mf]} \geq \eta\avgb{\sqrt{\Ff}\Wf^2 \sqrt{\Ff}} \normtwo{\Mf^{-1}}^{-2} \gtrsim \normtwo{\Mf^{-1}}^{-3} \disvz[2] \avg{\Ff}.
\end{equation}
Similarly, we have
\begin{equation} \label{eq:bound_denominator}
 \scalar{\Ff}{\Wf^{-2}} = \avgb{\sqrt{\Ff}\Wf^{-2} \sqrt{\Ff}} \lesssim \frac{\eta}{\disvz[2]} \normtwo{\Mf^{-1}} \avg{\Ff}. 
\end{equation}
Combining \eqref{eq:bound_numerator} and \eqref{eq:bound_denominator} in \eqref{eq:norm_mF} yields \eqref{eq:norm_mF_estimate} and concludes the proof of the lemma.
\end{proof}

\begin{lem}[Bounds on the inverse of the stability operator]\label{lem:Lf_invertible} 
Let $(\Af, \sopb)$ be a data pair as in Definition~\ref{def:Data pair}. 
\begin{enumerate}[(i)]
\item \label{enu_item:Lf_invertible}
The stability operator $\lopb$ is invertible for all $z \in \Hb$. For fixed $E \in\R$ and uniformly for $\eta\geq \max\{1,\abs{E},\normtwo{\Af}\}$, we have
\begin{equation} \label{eq:norm_Lf_large_eta}
\norm{\lopb^{-1}(E+\ii\eta)} \lesssim 1.
\end{equation} 
\item Uniformly for $z\in\Hb$, we have
\begin{equation} \label{eq:lopb_bounded_sp}
\normsp{\lopb^{-1}(z)}\lesssim  \frac{\normtwo{\Mf(z)}\normtwo{\Mf^{-1}(z)}^9}{\disvz[8]}.  
\end{equation}
\item 
Uniformly for $z\in\Hb$, we have 
\begin{equation} \label{eq:lopb_bounded}
 \norm{\lopb^{-1}(z)} + \norm{(\lopb^{-1}(z))^*}
\lesssim 1 + \normtwo{\Mf(z)}^2 + \normtwo{\Mf(z)}^4 \normsp{\lopb^{-1}(z)}. 
\end{equation}
\end{enumerate}
\end{lem}

\begin{proof}
We start with the proof of \eqref{eq:lopb_bounded_sp}.
From the upper and lower bounds in \eqref{eq:lower_bound_Wf} and \eqref{eq:Im_Mf_upper_bound_away_support}, respectively, we obtain 
\begin{subequations} \label{eq:estimates_aux_operators_stab_operator}
\begin{align}
\norm{\mC_\Wf} \lesssim & \, \frac{1}{\eta} \normtwo{\Mf}\normtwo{\Mf^{-1}}^2, & \qquad \norm{\mC_\Wf^{-1}} \lesssim & \, \frac{\eta}{\disvz[2]} \normtwo{\Mf^{-1}},  \\
 \norm{\mC_{\sqrt{\Im \Mf}}}\lesssim & \, \frac{\eta}{\disvz[2]},&  \qquad \norm{\mC_{\sqrt{\Im \Mf}}^{-1}} \lesssim  & \,\frac{1}{\eta}\normtwo{\Mf^{-1}}^2.
\end{align}
\end{subequations}
Since $\normsp{\mC_\Tf} \leq \norm{\mC_\Tf}$ for Hermitian $\Tf \in\C^{n \times n}$ we conclude from \eqref{eq:estimates_aux_operators_stab_operator}, \eqref{eq:norm_mF_estimate} and \eqref{eq:Mf_upper_bound_eta}
\[ \normsp{\lopb^{-1}} \lesssim \frac{\normtwo{\Mf}\normtwo{\Mf^{-1}}^5}{\disvz[4]} \normsp{(\mC_\Uf - \mF)^{-1}} \lesssim \frac{\normtwo{\Mf}\normtwo{\Mf^{-1}}^9}{\disvz[8]}. \]

For the proof of \eqref{eq:lopb_bounded}, we remark that $\norm{\sopb}_{\rm{hs}\to\norm{\genarg}} \lesssim 1$ implies
$\norm{\sopb}_{\norm{\genarg}\to\rm{hs}}\lesssim 1$. 
Therefore, exactly as in the proof of (4.53) in \cite{AjankiCorrelated}, we obtain
the first bound in~\eqref{eq:lopb_bounded}. We similarly conclude the second bound from $\normsp{(\lopb^{-1})^*} = \normsp{\lopb^{-1}}$. 

We conclude the proof of Lemma \ref{lem:Lf_invertible} by remarking that \eqref{eq:norm_Lf_large_eta} is a consequence of \eqref{eq:lopb_bounded_sp}, \eqref{eq:Mf_upper_bound_eta}, \eqref{eq:lopb_bounded}
and~\eqref{eq:bound_Mf_inverse}.
\end{proof}

\begin{coro}[Lipschitz-continuity of $\Mf$] \label{coro:Mf_lipschitz}
If $(\Af, \sopb)$ is a data pair as in Definition \ref{def:Data pair} then  
there exists $c>0$ such that for each (possibly $N$-dependent) $\eps\in(0,1]$ we have 
\begin{equation} \label{eq:Mf_Lipschitz_continuous}
 \normtwo{\Mf(z_1)-\Mf(z_2)} \lesssim (\eps^{-c} + \normtwo{\Af}^c) \abs{z_1-z_2} 
\end{equation}
for all $z_1, z_2 \in \Hb$ such that $\Im z_1, \Im z_2 \geq \eps$.
\end{coro}

\begin{proof}
We differentiate \eqref{eq:Dyson_matrix} with respect to $z$ and obtain $\lopb[\pt_z \Mf] = \Mf^2$.
We invert $\lopb$, use \eqref{eq:lopb_bounded_sp}, \eqref{eq:Mf_upper_bound_eta} and \eqref{eq:bound_Mf_inverse} and follow the proof of \eqref{eq:lopb_bounded}. 
This yields \eqref{eq:Mf_Lipschitz_continuous} and hence concludes the proof of Corollary \ref{coro:Mf_lipschitz}.  
\end{proof}

\subsection{Translation to results for Kronecker matrices}
Here we translate the results of Subsections~\ref{subsec:Solution of the Dyson equation} and \ref{subsec:stability_mde} into results about \eqref{eq:Dyson_non_Hermitian}.
In fact, we study \eqref{eq:Dyson_non_Hermitian} in a slightly more general setup.
Motivated by the identification $\C^{2\times 2} \otimes \C^{L\times L} \cong \C^{2L \times 2L}$, 
we consider \eqref{eq:Dyson_non_Hermitian} on $\C^{K\times K}$ for some $K \in \N$ instead.
The results of Subsections~\ref{subsec:Solution of the Dyson equation} and \ref{subsec:stability_mde} are 
applied with $n=KN$.
Moreover, the special $a_j^\zeta$ defined in \eqref{eq:mapping_non_hermitian_to_hermitian2} are replaced by general $a_j\in \C^{K\times K}$. 
Therefore, the parameter $\zeta$ will not be present throughout this subsection.
We thus look at the \emph{Dyson equation in vector form} 
\begin{equation} \label{eq:Dyson}
- \frac{1}{m_j(z)} = z\id  - a_j + \Sf_j[\mf(z)],
\end{equation}
where $z\in\Hb$, $m_j(z)\in\C^{K\times K}$ for $j \in [N]$, $\mf(z)\defeq (m_1(z), \ldots m_N(z))$ and $\Sf_j$ is defined as in \eqref{eq:sopb_l}. 

Recall that the definition of $\Sf_j$ involves coefficients $s_{ij}^\mu$ and $t_{ij}^\nu$ as well as matrices $\alpha_\mu$ and $\beta_\nu$. Next, we formulate 
assumptions on $\Sf$ in terms of these data as well as assumptions on $a_1, \ldots, a_N$. 

\begin{assums} \label{assums:qualitative_dyson}
\begin{enumerate}[(i)]
\item 
For all $\mu, \nu \in [\ell]$ and $i, j \in [N]$, we have nonnegative scalars $s_{ij}^\mu \in \R$ and $t_{ij}^\nu\in\R$ satisfying \eqref{eq:upper_bound_variances}. 
Furthermore, $s_{ij}^\mu = s_{ji}^\mu$ for all $i,j \in [N]$ and $\mu \in [\ell]$. 
\item 
For $\mu,\nu \in [\ell]$, we have $\alpha_\mu, \beta_\nu \in \C^{K\times K}$ and $\alpha_\mu$ is Hermitian.
There is $\alpha^* >0$ such that
\begin{equation}
\label{eq:bound_coefficients} 
\max_{\mu\in [\ell]} \abs{\alpha_\mu}\leq \alpha^*, \qquad 
\max_{\nu \in [\ell]} \abs{\beta_\nu} \leq \alpha^*.
\end{equation}
\item The matrices $a_1, \ldots, a_N \in\C^{K\times K}$ have a negative semidefinite imaginary part, $\Im a_j \leq 0$.
\end{enumerate}
\vspace{0.1cm}
\end{assums}

The conditions in (i) of Assumptions \ref{assums:qualitative_dyson} are motivated by the definition of the variances in \eqref{eq:def_variances}. 
In particular, since $X_\mu$ is Hermitian the variances from \eqref{eq:def_variances} satisfy $s_{ij}^\mu = s_{ji}^\mu$. 

In order to apply the results of Subsections \ref{subsec:Solution of the Dyson equation} and 
\ref{subsec:stability_mde} to \eqref{eq:Dyson}, we 
now relate it to the matrix Dyson equation (MDE) \eqref{eq:Dyson_matrix}.
It turns out that \eqref{eq:Dyson} is a special case when the MDE on $\tenspace = \C^{K\times K} \otimes \C^{N\times N}$ 
is restricted to the \emph{block diagonal matrices} 
\begin{equation} \label{eq:def_blockdiag}
\blockdiag \defeq \linspan\{ a\otimes D\colon a \in\C^{K\times K}, D \in \C^{N\times N} \text{ diagonal}\}\subset \tenspace. 
\end{equation}
We recall $E_{ll}$, $\Sf_l$ and $P_{ll}$ from \eqref{eq:def_E_ij}, \eqref{eq:sopb_l} and \eqref{eq:def_P_ij}, 
respectively, and define $\Af\in\tenspace$ and $\sopb\colon \tenspace \to \tenspace$ through
\begin{equation} \label{eq:def_sopb}
\Af \defeq \sum_{l=1}^N a_l \otimes E_{ll}, \qquad  
\sopb[\Rf] \defeq \sum_{l=1}^N \Sf_l[(P_{11} \Rf, \ldots, P_{NN} \Rf)] \otimes E_{ll}.
\end{equation}
With these definitions, the Dyson equation in vector form, \eqref{eq:Dyson}, can be rewritten in the matrix form 
\eqref{eq:Dyson_matrix} for a solution matrix $\Mf \in \tenspace$. 
In the following, we will refer to \eqref{eq:Dyson_matrix} with these choices of $\tenspace$, $\Af$ and $\sopb$ as the \emph{Dyson equation in matrix form}.

In the remainder of the paper, we will consider the Dyson equation in matrix form, \eqref{eq:Dyson_matrix}, exclusively with the choices of $\Af$ and $\sopb$ from \eqref{eq:def_sopb}.
We have the following connection between \eqref{eq:Dyson} and \eqref{eq:Dyson_matrix}. 
If $\Mf$ is a solution of \eqref{eq:Dyson_matrix} then, 
since the range of $\sopb$ is contained in $\blockdiag$
and $\Af \in \blockdiag$, we have $\Mf \in \blockdiag$, i.e, it can be written as
\begin{equation} \label{eq:structure_M}
 \Mf(z) = \sum_{j=1}^N m_j(z) \otimes E_{jj}
\end{equation}
for some unique $m_1(z), \ldots, m_N(z) \in \C^{K \times K}$. Moreover, these $m_i$ solve \eqref{eq:Dyson}. Conversely, if $\mf=(m_1, \ldots, m_N) \in(\C^{K\times K})^N$ 
solves \eqref{eq:Dyson} then $\Mf$ defined via \eqref{eq:structure_M} is a solution of \eqref{eq:Dyson_matrix}. Furthermore, if $\Mf$ satisfies \eqref{eq:structure_M} then 
$\Im \Mf$ is positive definite if and only if $\Im m_j$ is positive definite for all $j \in[N]$. 
This correspondence yields the following translation of Lemma~\ref{lmm:matrix support of V} to the setting for Kronecker random matrices, Proposition \ref{pro:exist_unique_M} below. 

For part (ii), we recall $\norm{\rf}=\max_{i=1}^N\abs{r_i}$ for $\rf=(r_1, \ldots, r_N) \in (\C^{K\times K})^N$ and that $\norm{\Sf}$ denotes the 
operator norm of $\Sf \colon (\C^{K\times K})^N\to (\C^{K\times K})^N$ induced by $\norm{\genarg}$.
We also used that $\norm{\Sf} = \norm{\sopb}$, which is easy to see  since $\Sf=\sopb$ on the block diagonal matrices $(\C^{K\times K})^N \cong \blockdiag$ and $\sopb=0$ on the orthogonal complement $\blockdiag^\perp$.
The orthogonal complement is defined with respect to the scalar product on $\tenspace$ introduced in \eqref{eq:def_normHS}.
Furthermore, we remark that the identity \eqref{eq:structure_M} implies 
\[ \normtwo{\Mf} = \norm{\mf}. \]

\begin{pro}[Existence, uniqueness of $\mf$] \label{pro:exist_unique_M}
Under Assumptions \ref{assums:qualitative_dyson} we have 
\begin{enumerate}[(i)]
\item 
There is a unique function $\mf \colon \Hb \to (\C^{K\times K})^N$ such that the components $\mf(z)=(m_1(z), \ldots, m_N(z))$ satisfy \eqref{eq:Dyson} for $z\in\Hb$ and all $j\in[N]$ 
and $\Im m_j(z)$ is positive definite for all $z\in\Hb$ and all $j \in [N]$.
Furthermore,
for each $j \in [N]$, there is a measure $v_j$ on $\R$ with values in the positive semidefinite matrices of $\C^{K\times K}$ such that $v_j(\R) = \id$ and for all $z \in \Hb$, we have 
\begin{equation} \label{eq:Mf_Stieltjes_integral}
 m_j(z) = \int_\R \frac{v_j(\di \tau)}{\tau -z }. 
\end{equation}
\item If $a_j$ is Hermitian, i.e., $a_j=a_j^*$ for all $j \in [N]$ then the union of the supports of $v_j$ is 
comparable with the union of the spectra of the $a_j$ in the following sense   
\begin{subequations} 
\begin{align}
\label{eq:support_V}
\bigcup_{j=1}^N\supp v_j & \, \subset \, \bigcup_{j=1}^N \spec a_j + [-2 \norm{\Sf}^{1/2}, 2 \norm{\Sf}^{1/2}], \\
\bigcup_{j=1}^N \spec a_j & \, \subset \, \bigcup_{j=1}^N \supp v_j + [-\norm{\Sf}^{1/2}, \norm{\Sf}^{1/2}]. \label{eq:spec_A_in_supp_rho}
\end{align}
\end{subequations}
\end{enumerate}
\end{pro}

\begin{proof}[Proof of Lemma \ref{lem:Dyson_non_hermitian}]
Using the identification $\C^{2\times 2} \otimes \C^{ L\times  L} \cong \C^{K\times K}$ for $K = 2 L$ and the definitions in \eqref{eq:mapping_non_hermitian_to_hermitian1} and 
\eqref{eq:mapping_non_hermitian_to_hermitian2}, the lemma is an immediate consequence of Proposition~\ref{pro:exist_unique_M} with $a_j = a_j^\zeta$ for $j \in [N]$
since the proof of the proposition only uses the qualitative conditions in Assumptions~\ref{assums:qualitative_dyson}.
\end{proof}

Proposition \ref{pro:exist_unique_M} asserts that there is a measure $V_\Mf$ on $\R$ with values in the positive semidefinite elements of $\blockdiag \subset \tenspace$ 
such that for $z\in\Hb$, we have
\begin{equation} \label{eq:def_V_Mf}
V_\Mf(\di \tau) \defeq \sum_{j=1}^N v_j(\di \tau)\otimes E_{jj}, \qquad \Mf(z) = \int_\R \frac{1}{\tau - z } V_\Mf(\di \tau). 
\end{equation}
Clearly, we have $V_\Mf=\Vf$ for the unique measure $\Vf$ with values in positive semidefinite matrices that satisfies \eqref{matrix Stieltjes transform rep}. And we have $\supp V_\Mf=\supmeas$ with the self-consistent density of states defined in \eqref{eq:matrix def_self_cons_dens_states}. Note that in this setup
\begin{equation} \label{eq:def_self_cons_dens_states}
 \rho(\di \tau) = \frac{1}{NK} \sum_{j=1}^N \Tr v_j(\di\tau)\,,
\end{equation}
with the $\C^{K \times K}$-matrix valued measures $v_j$ defined through \eqref{eq:Mf_Stieltjes_integral}.

In the remainder of the paper, $\mf=(m_1, \ldots, m_N)$ and $\Mf$ always denote the unique solutions of \eqref{eq:Dyson} and~\eqref{eq:Dyson_matrix}, respectively, 
connected via \eqref{eq:structure_M}. 
We now modify the concept of {comparison relation} introduced in Convection~\ref{conv:matrix_Comparison_Relation} so that inequalities are understood up to constants depending only on the model parameters from Assumption~\ref{assums:qualitative_dyson}.
\begin{convention}[Comparison relation] \label{conv:Comparison_Relation}
From here on we use the comparison relation introduced in Convection~\ref{conv:matrix_Comparison_Relation} so that the constants implicitly hidden in this notation may  depend only on 
$K$, $\ell$, $\kappa_1$ from \eqref{eq:upper_bound_variances} and $\alpha^*$ from \eqref{eq:bound_coefficients}. 
\end{convention}

\begin{lem}[Bounds on $\Sf$] \label{lem:norm_Sf}
Assumptions \ref{assums:qualitative_dyson} imply
\begin{equation}
\normsp{\Sf} \, \lesssim \, 1, \qquad \norm{\Sf} \, \lesssim \, 1. \label{eq:norm_Sf}
\end{equation}
\end{lem}

\begin{proof}
Direct estimates of $\Sf[\af]$ for $\af\in(\C^{K\times K})^N$ starting from the definition of $\Sf_i$, \eqref{eq:sopb_l}, and using the assumptions 
\eqref{eq:upper_bound_variances} and \eqref{eq:bound_coefficients} yield the bounds in \eqref{eq:norm_Sf}.
\end{proof}

Similarly to $\lopb$, we now introduce the stability operator of the Dyson equation in vector form, \eqref{eq:Dyson}. 
In fact, it is defined through
\begin{equation} \label{eq:def_stab_operator_vector}
 \Lf \colon (\C^{K\times K})^N \to (\C^{K\times K})^N, \quad \Lf (r_1, \ldots,r_N)  \defeq (r_i - m_i \Sf_i[\rf]m_i)_{i=1}^N.
\end{equation}
We remark that $\sopb$ and thus $\lopb$ leave the set of block diagonal matrices $\blockdiag$ defined 
in \eqref{eq:def_blockdiag} invariant. 
The operators $\Sf$ and $\Lf$ are the restrictions of $\sopb$ and $\lopb$ to $\blockdiag$. In particular, we have 
\begin{equation} \label{eq:comparing_norms_of_Lf_and_lopb}
 \normsp{\Lf^{-1}} \leq \normsp{\lopb^{-1}}, \qquad \normsp{\lopb^{-1}} \leq \max\{1, \normsp{\Lf^{-1}}\}, \qquad \norm{\Lf^{-1}} \leq \norm{\lopb^{-1}}, 
\end{equation}
since $\lopb$ acts as the identity map on the orthogonal complement $\blockdiag^\perp$ of the block diagonal matrices.
Here, the orthogonal complement is defined with respect to the scalar product on $\tenspace$ introduced in \eqref{eq:def_normHS}.
Moreover, $\Lf$ is invertible if and only if $\lopb$ is invertible. 
Using \eqref{eq:comparing_norms_of_Lf_and_lopb} the bounds on $\lopb$ from Lemma~\ref{lem:Lf_invertible} can be translated into bounds on $\Lf$

\section{Hermitian Kronecker matrices} \label{sec:Hermitian_Kronecker_matrices}

The analysis of a non-Hermitian random matrix usually starts with Girko's Hermitization procedure. 
It provides a technique to extract spectral information about a  non-Hermitian matrix $\Xf$  
 from a family of Hermitian matrices $(\Hf^\zeta)_{\zeta \in \C}$ defined through
\begin{equation} \label{eq:def_Hf_zeta}
 \Hf^\zeta \defeq \begin{pmatrix} 0 & 1\\ 0 & 0 \end{pmatrix} \otimes \Xf + \begin{pmatrix} 0 & 0 \\ 1 & 0 \end{pmatrix} \otimes \Xf^* - \begin{pmatrix} 0 & \zeta \\ \bar \zeta & 0 \end{pmatrix} \otimes \id,
\qquad \zeta \in \C.
\end{equation}
Applying Girko's Hermitization procedure to a Kronecker random matrix $\Xf$ as in \eqref{eq:def_Xf} generates a Hermitian Kronecker matrix $\Hf^\zeta \in \C^{2\times 2} \otimes \C^{L\times L} \otimes \C^{N\times N}$. 
However, similarly to our analysis in Section \ref{sec:Dyson_equation}, we study more general Kronecker matrices $\Hf \in \C^{K\times K}\otimes \C^{N\times N}$ as in \eqref{eq:def_Hf} below for $K, N\in \N$.
This is motivated by the identification $\C^{2\times 2} \otimes \C^{L\times L} \cong \C^{2L\times 2L}$. 

For $K, N \in \N$, let the random matrix $\Hf \in \C^{K\times K}\otimes \C^{N\times N}$ be defined through
\begin{equation} \label{eq:def_Hf}
 \Hf \defeq \sum_{\mu=1}^{\ell} \alpha_\mu \otimes X_\mu + \sum_{\nu=1}^{\ell} 
\left( \beta_\nu \otimes  Y_\nu  + \beta_\nu^* \otimes Y_\nu^* \right) + \sum_{i=1}^N a_i \otimes E_{ii}.
\end{equation}
Furthermore, we make the following assumptions. 
Let $\ell \in\N$. 
For $\mu\in [\ell]$, let $ {\alpha}_\mu\in \C^{ K\times  K}$ be a deterministic Hermitian matrix and $X_\mu = X_\mu^* \in \C^{N \times N}$ a Hermitian random matrix
with centered and independent entries (up to the Hermitian symmetry constraint). 
For $\nu \in [\ell]$, let ${\beta}_\nu \in \C^{ K\times  K}$ be a deterministic matrix and $Y_\nu$ a random matrix with centered and independent entries. 
We also assume that $X_1,\ldots, X_\ell, Y_1, \ldots, Y_\ell$ are independent.
Let $a_1, \ldots, a_N \in \C^{K\times K}$ be some deterministic matrices with negative semidefinite imaginary part.
We recall that $E_{ii}$ was defined in \eqref{eq:def_E_ij} and introduce the expectation $\Af\defeq \E\Hf=\sum_{i=1}^N a_i \otimes E_{ii}$. 

If $\Af$ is a Hermitian matrix then $\Hf$ as in \eqref{eq:def_Hf} with the above properties is a Hermitian Kronecker random matrix in the sense of Definition \ref{def:kronecker_matrix}. 
As in the setup from \eqref{eq:def_Xf}, the matrices $\alpha_1, \ldots \alpha_{\ell},\beta_1, \ldots, {\beta}_{\ell}$ are called \emph{structure matrices}.

Since the imaginary parts of $a_1, \ldots, a_N$ are negative semidefinite, the same holds true for the imaginary part of $\Af$ and $\Hf$.
Hence, the matrix $\Hf-z\id$ is invertible for all $z \in \Hb$.
For $z\in\Hb$, we therefore introduce the resolvent $\Gf(z)$ of $\Hf$ and its ``matrix elements'' $G_{ij}(z)\defeq P_{ij} \Gf \in\C^{K\times K}$ for $i,j\in[N]$ defined through
\[  \Gf(z)\defeq (\Hf - z \id )^{-1}, \qquad \qquad \Gf(z) = \sum_{i,j=1}^N G_{ij}(z) \otimes E_{ij}. \]
We recall that $P_{ij}$ has been defined in \eqref{eq:def_P_ij}.
Our goal is to show that $G_{ij}$ is small for $i\neq j$ and $G_{ii}$ is well approximated by the deterministic matrix $m_i(z) \in \C^{K\times K}$ in the regime where $K \in \N$ is fixed 
and $N \in \N$ is large.

Apart from the above listed qualitative assumptions, we will need the following quantitative assumptions. 
To formulate them we use the same notation as before, i.e., the entries of $X_\mu$ and $Y_\nu$ are denoted by $X_\mu=(x_{ij}^\mu)_{i,j=1}^N$ and $Y_\nu=(y_{ij}^\nu)_{i,j=1}^N$ 
and their variances by $s_{ij}^\mu \defeq \E \abs{x_{ij}^\mu}^2$ and $t_{ij}^\nu \defeq \E\abs{y_{ij}^\nu}^2$ (cf. \eqref{eq:def_variances}).

\begin{assums} \label{assums:Hermitian_local_law}
We assume that all variances $s_{ij}^\mu$ and $t_{ij}^\mu$ satisfy \eqref{eq:upper_bound_variances} and the entries $x_{ij}^\mu$ and $y_{ij}^\nu$ of the random matrices fulfill 
the moment bounds \eqref{eq:moments_bounds}. 
Furthermore, the structure matrices satisfy \eqref{eq:bound_coefficients}.
\end{assums}

In this section, the model parameters are defined to be $K$, $\ell$, $\kappa_1$ from \eqref{eq:upper_bound_variances}, the sequence $(\varphi_p)_{p \in \N}$ from \eqref{eq:moments_bounds} 
and $\alpha^*$ from \eqref{eq:bound_coefficients}, so the relation $\lesssim$ indicates an inequality up to a multiplicative constant depending on these model parameters. 
Moreover, for the real and imaginary part of the spectral parameter $z$ we will write $E=\Re z$ and $\eta = \Im z$, respectively.

\subsection{Error term in the perturbed Dyson equation}

We introduce the notion of \emph{stochastic domination}, a high probability bound up to $N^\eps$ factors. 

\begin{defi}[Stochastic domination] \label{def:stochastic_domination}
If $\Phi = (\Phi^{(N)})_{N}$ and $\Psi = (\Psi^{(N)})_{N}$ 
are two sequences of nonnegative random variables, then we say that $\Phi$ is \textbf{stochastically dominated} by $\Psi$, $\Phi \prec \Psi$, if for all $\eps >0$ and $D>0$ 
there is a constant $C(\eps,D)$ such that  
\begin{equation} \label{eq:def_stoch_domination}
 \P \left( \Phi^{(N)} \geq N^\eps \Psi^{(N)} \right) \leq \frac{C(\eps,D)}{N^D} 
\end{equation}
for all $N \in \N$ and the function $(\eps, D) \mapsto C(\eps,D)$ depends only on the model parameters. 
If $\Phi$ or $\Psi$ depend on some additional parameter $\delta$ and the function $(\eps,D) \mapsto C(\eps,D)$ additionally depends on $\delta$
then we write~$\Phi \prec_\delta\Psi$. 
\end{defi}

We set $h_{ij} \defeq P_{ij} \Hf \in \C^{K \times K}$.
Using $P_{lm}\Af = a_l \delta_{lm}$,  $\E\,x_{ik}^\mu = 0$, $\E\,y_{ik}^\nu=0$, \eqref{eq:upper_bound_variances}, \eqref{eq:bound_coefficients} and 
\eqref{eq:moments_bounds} we trivially obtain 
\begin{equation} \label{eq:trivial_control_x_ij_prec}
\abs{P_{ik} \left( \Hf-\Af\right)} = \abs{h_{ik}-a_i\delta_{ik}} \prec N^{-1/2}. 
\end{equation}
For $B \subset [N]$ we set 
\[ \Hf^B \defeq \sum_{i, j = 1}^N h_{ij}^B \otimes E_{ij}, \qquad h_{ij}^B \defeq h_{ij} \char(i,j \notin B), \]
and denote the resolvent of $\Hf^B$ by
$ \Gf^B(z) \defeq \left( \Hf^B - z \id \right)^{-1}$ for $z \in \Hb$.
Since $\Im \Hf^B = \Im \Af^B \leq 0$ for $B \subset [N]$, the matrix $(\Hf^B - z\id)$ is invertible 
for all $z \in\Hb$ and 
\begin{equation} \label{eq:trivial_bound_resolvent}
 \normtwo{\Gf^B(z)} \leq \frac{1}{\Im z}.
\end{equation}

In the following, we will use the convention 
\[ \sum_{k \in A}^B \defeq \sum_{k \in A \setminus B} \]
for $A, B \subset [N]$ and $B \subset A$. 
If $A = [N]$ then we simply write $\sum_k^B$. 

For $i \in[N]$, starting from the Schur complement formula, 
\begin{equation} \label{eq:Schur_complement_formula} 
-\frac{1}{G_{ii}} = z - h_{ii} + \sum_{k,l}^{\{i\}} h_{ik}G_{kl}^{\{i\}}h_{li}, 
\end{equation}
and using the definition of $\Sf_i$ in \eqref{eq:sopb_l}, we obtain the perturbed Dyson equation
\begin{equation} \label{eq:dyson_perturbed}
-\frac{1}{g_{i}}\,=\, z-a_{i}+ \Sf_i[\gf]+d_{i}.
\end{equation}
Here, we introduced 
\begin{equation} \label{eq:def_resolvent_block_diagonal}
g_{i} \defeq G_{ii}, \qquad \gf\defeq(g_1,\ldots, g_N)\in (\C^{K \times K})^N
\end{equation}
and the error term $d_i\in\C^{K\times K}$. 
We remark that \eqref{eq:dyson_perturbed} is a perturbed version of the Dyson equation in vector form, \eqref{eq:Dyson}, and recall that $\mf$ denotes its 
unique solution (cf. Proposition \ref{pro:exist_unique_M}). 
To represent the error term $d_i$ in \eqref{eq:dyson_perturbed}, we use $ h_{ik} = a_i \delta_{ik} + \sum_{\mu} x_{ik}^\mu \alpha_\mu + \sum_\nu \left(y_{ik}^\nu \beta_\nu + \overline{y_{ki}^\nu} \beta_\nu^* \right)$ 
and write 
$d_{i}\defeq d_i^{(1)} + \ldots + d_i^{(8)}$, where 
\begin{subequations} \label{eq:error_terms}
\begin{align}
d^{(1)}_{i} & \defeq \, - h_{ii}+a_{i}, \label{eq:def_d1}\\
d^{(2)}_{i} & \defeq \, \sum_k^{\{i\}} \Big( \sum_\mu \alpha_\mu G_{kk}^{\{i\}} \alpha_\mu \left( \abs{x_{ik}^\mu}^2 -s_{ik}^\mu \right) + 
 \sum_\nu \left( (\abs{y_{ik}^\nu}^2 - t_{ik}^\nu) \beta_\nu G_{kk}^{\{i\}} \beta_\nu^* + (\abs{y_{ki}^\nu}^2 - t_{ki}^\nu) \beta_\nu^*G_{kk}^{\{i\}}\beta_\nu \right)  \Big) , \label{eq:def_d2} \\
d^{(3)}_{i} & \defeq \, \sum_\nu \sum_k^{\{i\}} \left( y_{ik}^\nu \beta_\nu G_{kk}^{\{i\}} \beta_\nu y_{ki}^\nu+ \overline{y_{ki}^\nu} \beta_\nu^* G_{kk}^{\{i\}} \beta_\nu^* \overline{y_{ik}^\nu}\right) \label{eq:def_d3} \\
d^{(4)}_{i} & \defeq \, 
 \Big( \sum_{\mu = \mu'} \sum_{k \neq l}^{\{i\}} + \sum_{\mu \neq \mu'} \sum_{k,l}^{\{i\}} \Big) \alpha_\mu x_{ik}^\mu G_{kl}^{\{i\}} x_{li}^{\mu'} \alpha_{\mu'} ,\label{eq:def_d4} \\
d^{(5)}_{i} & \defeq \,  \Big(\sum_{\nu = \nu'} \sum_{k\neq l }^{\{i\}} + \sum_{\nu \neq \nu'} \sum_{k,l}^{\{i\}} \Big) 
\left( y_{ik}^\nu \beta_\nu + \overline{y_{ki}^\nu} \beta_\nu^* \right) G_{kl}^{\{i\}} \left( y_{li}^{\nu'} \beta_{\nu'} + \overline{y_{il}^{\nu'}} \beta_{\nu'}^* \right),\label{eq:def_d5} \\
d^{(6)}_{i} & \defeq \, \sum_{k,l}^{\{i\}}  \sum_\mu \sum_\nu \left( \alpha_\mu x_{ik}^\mu G_{kl}^{\{i\}} \left( y_{li}^\nu \beta_\nu + \overline{y_{il}^\nu} \beta_\nu^* \right) 
+ \left(y_{ik}^\nu \beta_\nu + \overline{y_{ki}^\nu} \beta_\nu^* \right) G_{kl}^{\{i\}} x_{li}^\mu \alpha_\mu \right), \label{eq:def_d6}  \\
d^{(7)}_{i} & \defeq \, \sum_k^{\{i\}} \Big( \sum_\mu \alpha_\mu s_{ik}^\mu\left( G_{kk}^{\{i\}} - G_{kk} \right)\alpha_\mu 
+  \sum_\nu \left( t_{ik}^\nu \beta_\nu \left(G_{kk}^{\{i\}}- G_{kk}\right) \beta_\nu^* + t_{ki}^\nu \beta_\nu^*\left(G_{kk}^{\{i\}} - G_{kk}\right)\beta_\nu \right) \Big),\label{eq:def_d7} \\
d^{(8)}_{i} & \defeq \, -\Big( \sum_\mu s_{ii}^\mu \alpha_\mu G_{ii}\alpha_\mu  + 
\sum_\nu t_{ii}^\nu\left(\beta_\nu G_{ii}\beta_\nu^* +  \beta_\nu^* G_{ii} \beta_\nu \right)\Big)  . \label{eq:def_d8}
\end{align}
\end{subequations}

In the remainder of this section, we consider $E=\Re z$ to be fixed and view quantities like $\mf$ and $\Gf$ only 
as a function of $\eta=\Im z$. 
In the following lemma, we will use the following random control parameters to bound the error terms introduced in \eqref{eq:error_terms}:
\begin{equation} \label{eq:def_Lambda_o_Lambda_HS}
\begin{split}
 \Lambda_{\rm{hs}}(\eta) &\defeq \frac{1}{N} \Big[ \Tr \Gf(E+\ii\eta)^*\Gf(E+\ii\eta) \Big]^{1/2}\,
\\ 
 \Lambda_{\rm{w}}(\eta) &\defeq \frac{1}{\sqrt{2N}}\max_{i=1}^N \Big[ \Tr P_{ii}[\Gf(E+\ii\eta)^*\Gf(E+\ii\eta)+\Gf(E+\ii\eta)\Gf(E+\ii\eta)^*]\Big]^{1/2},
 \\
 \Lambda(\eta)&\defeq \max_{i,j=1}^N \2\abs{G_{ij}(E+\ii\eta)- m_i(E+\ii\eta)\2\delta_{ij}}\,.
\end{split}
\end{equation}

We remark that due to our conventions, we have
\[ \norm{\mf} = \max_{i=1}^N \abs{m_i}, \qquad \norm{\mf^{-1}} = \max_{i=1}^N \abs{m_i^{-1}}. \]

\begin{lem} \label{lem:initial_error_estimates}
\begin{enumerate}[(i)]
\item Uniformly for $\eta \geq 1$ and $i\neq j$, we have 
\begin{subequations}
\begin{align}
\abs{d_i} &\, \prec  \, 1, \label{eq:deterministic_bound_error} \\
\abs{G_{ij}} & \, \prec \, \eta^{-2}.  \label{eq:deterministic_offdiagonal}
\end{align}
\end{subequations}
\item 
Uniformly for $\eta >0$, we have
\begin{subequations} \label{eq:initial_error_estimates_bound_G}
\begin{align}
\big(\abs{d^{(1)}_{i}} + \ldots + \abs{d^{(6)}_{i}} \big)\,\chi
& \prec \, \frac{1}{\sqrt{N}}+\Lambda_{\rm{hs}}+ \norm{\mf^{-1}}\Lambda_{\rm{w}}^2, \label{eq:initial_error_estimates_bound_G1}\\ 
\big(\abs{d^{(7)}_{i}} + \abs{d^{(8)}_{i}}\big)\,\chi & \prec \, \norm{\mf^{-1}}\Lambda_{\rm{w}}^2 + \frac{1}{N}\abs{G_{ii}}, \label{eq:initial_error_estimates_bound_G2}
\end{align}
\end{subequations}
where $\chi$ is the characteristic function $\chi\defeq \char(\Lambda \le (4\norm{\mf^{-1}})^{-1})$. 

Moreover, uniformly for $\eta>0$ and $i\neq j$, we have
\begin{equation} \label{eq:initial_bound_offdiag}
\abs{G_{ij}}\,\chi\,\prec\, \norm{\mf} \Lambda_{\rm{w}}.
\end{equation}
\end{enumerate}
\end{lem}

In the proof of Lemma \ref{lem:initial_error_estimates}, we use the following relation between the entries of $\Gf^T$ and $\Gf^{T\cup \{k\}}$ 
\begin{equation} \label{eq:resolvent_identity}
 G_{ij}^T = G_{ij}^{T\cup \{k\}} + G_{ik}^T \frac{1}{G_{kk}^T} G_{kj}^T 
\end{equation}
for $T \subset [N]$, $k \notin T$ and $i,j \notin T\cup \{k\}$. This is an identity of $K\times K$ matrices and $1/G_{kk}^T$ is understood as the inverse matrix of $G_{kk}^T$.
The proof of \eqref{eq:resolvent_identity} follows from the Schur complement formula. 

\begin{proof}
We will prove the bounds in \eqref{eq:initial_error_estimates_bound_G} in parallel with the estimate
\begin{equation} \label{eq:bound_d_no_char_function} 
\abs{d^{(1)}_i} +\ldots + \abs{d^{(8)}_i} \, \prec  \, \frac{1}{\sqrt N} + \frac{1}{N} \Big( \sum_{k,l}^{\{i\}} \abs{G_{kl}^{\{i\}}}^2  \Big)^{1/2}
+ \frac{1}{N} \sum_{k}^{\{i\}}\abs{G^{\{i\}}_{kk}} + \frac{1}{N}\sum_k \abs{G_{kk}} 
\end{equation}
that we will use to show \eqref{eq:deterministic_bound_error}.

The trivial estimate \eqref{eq:trivial_control_x_ij_prec} implies that $\abs{d^{(1)}_{i}}\prec 1/\sqrt{N}$. 

In the remaining part of the proof, we will often apply the large deviation bounds with scalar valued random variables from Theorem C.1 in \cite{EJP2473}. 
In our case, they will be applied to sums or quadratic forms of independent random variables, whose coefficients are $K\times K$ matrices; 
this generalization clearly follows from the scalar case \cite{EJP2473} if applied to each entry separately.

We first show the following estimate
\begin{equation} \label{eq:aux_est_d2_d3}
\abs{d^{(2)}_i} +\abs{d^{(3)}_i} \,\prec\, \frac{1}{\sqrt{N}}\pbb{\frac{1}{N}\sum_k^{\{i\}}\abs{G_{kk}^{\{i\}}}^2}^{1/2}.
\end{equation}
From the linear large deviation bound (C.2) in \cite{EJP2473}, we conclude that the first term in \eqref{eq:def_d2} is bounded by 
\[\sum_\mu \abs{\alpha_\mu}\absB{\sum_k^{\{i\}} G_{kk}^{\{i\}} (\abs{x_{ik}^\mu}^2 -s_{ik}^\mu)} \abs{\alpha_\mu} \prec 
\frac{1}{N} \Big(\sum_k^{\{i\}} \abs{G_{kk}^{\{i\}}}^2\Big)^{1/2}. \]   
The second and third term in \eqref{eq:def_d2} are estimated similarly with the help of (C.2) in \cite{EJP2473} which yields \eqref{eq:aux_est_d2_d3} for~$\abs{d^{(2)}_{i}}$.
We apply the linear large deviation bound (C.2) in \cite{EJP2473} and bound the first term in \eqref{eq:def_d3} as follows:
\[ \absB{\sum_\nu \Big(\sum_k^{\{i\}} y_{ik}^\nu y_{ki}^\nu \beta_\nu G_{kk}^{\{i\}} \beta_\nu \Big)} \prec \frac{1}{N}  \bigg(\sum_k^{\{i\}} \abs{G_{kk}^{\{i\}}} ^2 \bigg)^{1/2}. \] 
The bound on the second term in \eqref{eq:def_d3} is obtained in the same way. Consequently, we have proved \eqref{eq:aux_est_d2_d3}.

Using the quadratic large deviation bounds (C.4) and (C.3) in \cite{EJP2473}, we obtain 
\begin{equation} \label{eq:aux_est_d4_d5_d6}
\abs{d^{(4)}_i} +\abs{d^{(5)}_i}+\abs{d^{(6)}_i} \,\prec\, \pbb{\frac{1}{N^2}\sum_{k,l}^{\{i\}}\abs{G_{kl}^{\{i\}}}^2}^{1/2}.
\end{equation}
Moreover, \eqref{eq:aux_est_d2_d3} and \eqref{eq:aux_est_d4_d5_d6} also imply that $\abs{d^{(2)}_i}+ \ldots +\abs{d^{(6)}_i}$ are bounded by the second term on the right-hand side of \eqref{eq:bound_d_no_char_function}.

Using \eqref{eq:resolvent_identity}, \eqref{eq:upper_bound_variances} and  \eqref{eq:bound_coefficients}, we conclude 
\begin{equation} \label{eq:aux_est_d7}
\abs{d^{(7)}_i}\,\lesssim\, \min\Big\{ \frac{1}{N}\sum_{k}^{\{i\}}\abs{G_{ki}}\absbb{\frac{1}{G_{ii}}}\abs{G_{ik}}, \frac{1}{N}\sum_{k}^{\{i\}}( \abs{G^{\{i\}}_{kk}} + \abs{G_{kk}})  \Big\}.
\end{equation}
The assumptions \eqref{eq:upper_bound_variances} and \eqref{eq:bound_coefficients} imply 
\begin{equation} \label{eq:aux_est_d8}
\abs{d_{i}^{(8)}} \lesssim \abs{G_{ii}}/N. 
\end{equation}
This concludes the proof of \eqref{eq:bound_d_no_char_function}. Applying \eqref{eq:trivial_bound_resolvent} to \eqref{eq:bound_d_no_char_function}, we obtain \eqref{eq:deterministic_bound_error}.

For all $k,l \notin \{i\}$, we now show that 
\begin{equation} \label{eq:aux_estimate_minors}
 \absB{G_{kl}^{\{i\}}}\chi \leq \abs{G_{kl}} + \frac{4}{3} \norm{\mf^{-1}} \abs{G_{ki}} \abs{G_{il}}. 
\end{equation}
This immediately yields \eqref{eq:initial_error_estimates_bound_G1} using \eqref{eq:aux_est_d2_d3} and \eqref{eq:aux_est_d4_d5_d6}.
For the proof of \eqref{eq:aux_estimate_minors}, we conclude from \eqref{eq:resolvent_identity} by dividing and multiplying the second term by $m_i$ that
\begin{equation}\label{eq:one_removed}
G_{kl}^{\{i\}} = G_{kl} - G_{ki}\frac{1}{G_{ii}}m_i \frac{1}{m_i} G_{il}.
\end{equation}
From the definition of $\chi$ in Lemma~\ref{lem:initial_error_estimates}, we see that
\begin{equation} \label{bound on MinvG}
\absbb{\frac{1}{m_i}G_{ij}-\delta_{i j}}\,\chi \le  \frac{1}{4}\,, \qquad \absbb{\frac{1}{G_{ii}}m_i}\,\chi \,\le\, \frac{4}{3}\,,
\end{equation}
which proves \eqref{eq:aux_estimate_minors} and hence \eqref{eq:initial_error_estimates_bound_G1}.

Since \eqref{eq:initial_error_estimates_bound_G2} is established for $\abs{d_i^{(8)}}$ (cf. \eqref{eq:aux_est_d8}), 
it suffices to use the second bound in \eqref{bound on MinvG} to finish the proof of \eqref{eq:initial_error_estimates_bound_G2}
by estimating $\abs{d_i^{(7)}}$ via the first term in \eqref{eq:aux_est_d7}.

We now show \eqref{eq:initial_bound_offdiag} and \eqref{eq:deterministic_offdiagonal}. The identity 
\[
G_{i j}\,=\, -\sum_k^{\{j\}}G_{ik}^{\{j\}} h_{k j}G_{jj}
\]
and the linear large deviation bound (C.2) in \cite{EJP2473} imply 
\begin{equation} \label{eq:aux_estimate_offdiag}
\abs{G_{ij}}\,\prec\,\pbb{ \frac{1}{N}\sum_k^{\{j\}}\abs{G_{ik}^{\{j\}}}^2}^{1/2}\abs{G_{jj}}. 
\end{equation}
Using \eqref{eq:trivial_bound_resolvent} to estimate $\abs{G_{ik}^{\{j\}}}$ and $\abs{G_{jj}}$, we obtain \eqref{eq:deterministic_offdiagonal}. 
Applying the estimate \eqref{eq:aux_estimate_minors} and the definition of $\chi$ in \eqref{eq:aux_estimate_offdiag} yield
$\abs{G_{ij}}\chi \prec \abs{G_{jj}}\chi \Lambda_\rm{w}$.
Hence, the second bound in \eqref{bound on MinvG} implies \eqref{eq:initial_bound_offdiag} and conclude the proof of Lemma~\ref{lem:initial_error_estimates}.
\end{proof}

For the following computations, we recall the definition of the product and the imaginary part on $(\C^{K\times K})^N$ from \eqref{eq:def_function_vector_matrices} 
and \eqref{eq:def_product_vector_matrices}, respectively. 

The proof of the following Lemma \ref{lem:Gap_Lemma} is based on inverting the stability operator in the difference equation 
describing $\gf-\mf$ in terms of $\df$. We derive this equation first. 
Subtracting \eqref{eq:Dyson} from \eqref{eq:dyson_perturbed} and multiplying the result from the left by $m_i$ and 
from the right by $g_i$ yield
\[ g_i - m_i = m_i \Sf_i[\gf-\mf] m_i + m_i d_i g_i + m_i \Sf_i[\gf-\mf](g_i - m_i) \]
for $i \in [N]$. 
Introducing $\df =(d_1, \ldots, d_N) \in (\C^{K\times K})^N$ as well as recalling $\Sf[\rf] = (\Sf_i[\rf])_{i=1}^N$, the definition of $\Sf_i$ from \eqref{eq:sopb_l} 
and $\Lf[\rf] = \rf - \mf \Sf[\rf]\mf$ from \eqref{eq:def_stab_operator_vector}, we can write
\begin{equation} \label{eq:stab_vector}
 \Lf(\gf-\mf) = \mf \df \gf + \mf \Sf[\gf-\mf] (\gf-\mf). 
\end{equation}
Since $\Lf$ is invertible for $z\in\Hb$ by Lemma \ref{lem:Lf_invertible} (\ref{enu_item:Lf_invertible}) and \eqref{eq:comparing_norms_of_Lf_and_lopb}, 
applying the inverse of $\Lf$ on both sides of \eqref{eq:stab_vector} and estimating the norm yields
\begin{equation} \label{eq:norm_stab_equation_vector}
\norm{\gf-\mf}\,\le\, \norm{\Lf^{-1}}\norm{\mf}(\norm{\df} \norm{\gf}+ \norm{\Sf} \norm{\gf-\mf}^2)
\end{equation}

We recall the definition of $\rho$ from \eqref{eq:def_rho}.

\begin{lem} \label{lem:Gap_Lemma}
\begin{enumerate}[(i)]
\item Uniformly for $\eta \geq \max\{1, \abs{E}, \normtwo{\Af}\}$, we have  
\begin{equation} \label{eq:behaviour_lambda_large_eta}
\Lambda \prec \eta^{-2}.
\end{equation}
\item Uniformly for $\eta  > 0$, we have
\begin{equation} \label{Lambda gap bound}
\norm{\gf-\mf} \, \chi\p{\Lambda \le \vartheta} \,\prec\, \norm{\Lf^{-1}}\norm{\mf}^2\left(\frac{1}{\sqrt{N}}  +  \Lambda_{\rm{hs}}+\norm{\mf^{-1}}\Lambda_{\rm{w}}^2 \right),
\end{equation}
where 
\bels{def of vartheta}{
\vartheta\,\defeq\, \frac{1}{4(
\norm{\Lf^{-1}}\norm{\mf} \norm{\Sf}+ \norm{\mf^{-1}})}\,.
}
\item Let $a_1, \ldots, a_N$ be Hermitian. We define
\begin{align*}
\psi\,&\defeq\,\norm{\Lf^{-1}}\norm{\mf}^2\norm{\mf^{-1}}\frac{1}{N \eta}\,,
\\
\varphi\,&\defeq\, \norm{\Lf^{-1}}\norm{\mf}^2 
\left(\frac{1}{\sqrt{N}}  +   \sqrt{\frac{\rho}{N \eta}} + \norm{\Lf^{-1}}\norm{\mf}^2\frac{1}{N \eta} 
+  \frac{\norm{\mf^{-1}}}{N\eta} \norm{\im \mf} \right)   
+\norm{\mf} \left( \sqrt{\frac{\norm{\Im \mf}}{N \eta}}  + \frac{\norm{\mf}}{N\eta} \right)\,.
\end{align*}
Then for all $\delta>0$ and uniformly for all $\eta>0$ such that $\psi(\eta) \le N^{-\delta}$ we have 
\begin{equation} \label{eq:gap_lemma_bound_hermitian}
\Lambda \, \chi\p{\Lambda \le \vartheta} \,\prec_\delta\,\varphi\,.
\end{equation}
\end{enumerate}
\end{lem}

Note that the proof of (iii) of Lemma \ref{lem:Gap_Lemma} requires  
$\Hf$ to be Hermitian because of the use of the Ward identity, $\Gf(\eta)^*\Gf(\eta) =\eta^{-1} \Im\Gf(\eta)$. 
The Ward identity implies $P_{ii} \Gf^*\Gf = P_{ii} \Gf\Gf^* = \Im G_{ii}/\eta$ and hence,  
\begin{equation} \label{eq:ward_identity}
\Lambda_{\rm{hs}}\,=\, \sqrt{\frac{\avg{\im \bs{G}}}{N \eta}}, 
\qquad \Lambda_{\rm{w}}\,=\,\max_i\sqrt{\frac{ \im \Tr G_{ii} }{N \eta}}.
\end{equation}

\begin{proof}
We start with the proof of \eqref{eq:behaviour_lambda_large_eta}. 
We remark that $\norm{\gf}+\norm{\mf} \leq 2/\eta$ by \eqref{eq:trivial_bound_resolvent} and \eqref{eq:Mf_upper_bound_eta}. 
Therefore, for $\eta \geq \max\{1,\abs{E},\normtwo{\Af}\}$, we conclude from \eqref{eq:norm_stab_equation_vector} that 
\[ \norm{\gf-\mf} \lesssim \frac{1}{\eta^2} \norm{\df} + \frac{1}{\eta^3}. \]
Here, we also used \eqref{eq:norm_Lf_large_eta}, \eqref{eq:comparing_norms_of_Lf_and_lopb} and \eqref{eq:norm_Sf}. 
Since $\norm{\df} \prec 1$ by \eqref{eq:deterministic_bound_error}, we get 
$\norm{\gf-\mf} \prec \eta^{-2}$ in this $\eta$-regime.
Hence, combined with the bound \eqref{eq:deterministic_offdiagonal} for the offdiagonal terms, 
we obtain \eqref{eq:behaviour_lambda_large_eta}. 

For the proof of (ii), we also start from \eqref{eq:norm_stab_equation_vector}.
Since $2\norm{\Lf^{-1}}\norm{\mf} \norm{\Sf}\vartheta\le1$ by definition of $\vartheta$ (cf. \eqref{def of vartheta}) and $\norm{\gf}\char(\Lambda \leq \vartheta) \leq \norm{\mf} \norm{\mf^{-1}\gf}\char(\Lambda \leq \vartheta)
 \leq 4\norm{\mf}/3$ 
by the second bound in \eqref{bound on MinvG}, we conclude that
\bels{Lambda stability bound}{
\norm{\gf-\mf}\,\chi\p{\Lambda \le \vartheta}\,\le\,8 \norm{\Lf^{-1}}\norm{\mf}\norm{\df}\norm{\mf}/3 \,.
}
Applying \eqref{eq:initial_error_estimates_bound_G} to the right hand side and using $\abs{G_{ii}} \leq \sqrt{N} \Lambda_\rm{hs}$, we obtain \eqref{Lambda gap bound}.

For the proof of (iii), let now $\Hf$ be Hermitian. Therefore, \eqref{eq:ward_identity} is applicable and yields
\[
\Lambda_{\rm{hs}}\,=\, \sqrt{\frac{\avg{\im \bs{G}}}{N \eta}}\lesssim \sqrt{\frac{\rho}{N \eta}} + \frac{1}{\eps}\frac{1}{N \eta} + \eps \norm{\gf-\mf}  \,,
\qquad \Lambda_{\rm{w}}^2\,=\,\bigg(\max_{i=1}^N\sqrt{\frac{ \im \Tr G_{ii} }{N \eta}}\,\, \bigg)^2 \leq \frac{\norm{\im \mf}}{N \eta} +\frac{\norm{\gf-\mf}}{N \eta}.
\]
Here, we used $\avg{\Im \Gf} \leq \avg{\Im \Mf} + \norm{\gf-\mf}$, $\avg{\Im \Mf} = \pi \rho$ and Young's inequality as well as introduced an arbitrary $\eps>0$ in the first estimate. 
We plug these estimates into the right-hand side of \eqref{Lambda gap bound} and choose $\eps\,\defeq \, N^{-\gamma}/(\norm{\Lf^{-1}}\norm{\mf}^2)$ for arbitrary $\gamma>0$.
Thus, we can absorb $\norm{\gf-\mf}$ in the estimate on $\Lambda_\rm{hs}$ into the left-hand side of \eqref{Lambda gap bound}. Similarly, using $\psi(\eta) \leq N^{-\delta}$ we absorb $\norm{\gf-\mf}$ in the estimate
on $\Lambda_\rm{w}$ into the left-hand side of \eqref{Lambda gap bound}. 
This yields \eqref{eq:gap_lemma_bound_hermitian} for the contribution of the diagonal entries to $\Lambda$.

For the offdiagonal entries, we use the second relation in \eqref{eq:ward_identity} and get as before 
\[
 \Lambda_\rm{w} = \max_{i=1}^N\sqrt{\frac{ \im \Tr G_{ii} }{N \eta}} \leq \sqrt{\frac{\norm{\Im \mf}}{N\eta}} + \frac{1}{\eps} \frac{1}{N\eta} + \eps \Lambda.
\]
Using this estimate in \eqref{eq:initial_bound_offdiag} and choosing $\eps \defeq N^{-\gamma}/\norm{\mf}$ to absorb $\Lambda$ into the left-hand side, 
we obtain \eqref{eq:gap_lemma_bound_hermitian} for diagonal and offdiagonal entries of $\Gf$. This concludes the proof of Lemma \ref{lem:Gap_Lemma}.
\end{proof}

\begin{lem}[Averaged local law] \label{lem:avg_tech_local_law}
Suppose for some deterministic control parameter $0<\Phi\le N^{-\eps}$ a local law holds in the form
\begin{equation} \label{eq:assumption_averaged_local_law}
\Lambda \,\prec\,\frac{\Phi}{\norm{\mf^{-1}}}\,.
\end{equation}
Then for any deterministic $c_1, \ldots, c_N\in\C^{K\times K}$ with $\max_i\abs{c_i} \le 1 $ we have
\begin{equation} \label{eq:local_law_average_general}
\absB{\frac{1}{N} \sum_{i=1}^N c_i^*(G_{ii} -m_i)} \,\prec\, \norm{(\Lf^{-1})^*}\norm{\mf}\pbb{ \frac{\Phi^2}{\norm{\mf^{-1}}^2} +
\max\cbb{\frac{1}{\sqrt{N}},\Phi}\Phi+
\frac{\norm{\mf}^2}{N} + \Lambda^2_\rm{w} \norm{\mf}\norm{\mf^{-1}}}.
\end{equation}
\end{lem}

In \eqref{eq:local_law_average_general}, the adjoint of $\Lf^{-1}$ is understood with respect to the scalar product 
$\Tr(\xf\cdot \yf)$, where we defined the dot-product $\xf \cdot \yf$ 
for $\xf =(x_1, \ldots, x_N)$, $\yf = (y_1, \ldots, y_N) \in (\C^{K\times K})^N$ 
via
\begin{equation} \label{eq:def_matrix_valued_scalar_product}
 \xf\cdot \yf \defeq \frac{1}{N}\sum_{i=1}^N x_i^* y_i  \in \C^{K\times K}.
\end{equation}
It is easy to see that $\xf \cdot \Lf^{-1} \yf = ((\Lf^{-1})^* \xf) \cdot \yf$. 

\begin{proof}
We set $\cf \defeq (c_1, \ldots, c_N)$ and recall $\gf = (G_{11}, \ldots, G_{NN}) \in (\C^{K\times K})^N$. Using \eqref{eq:stab_vector}, we compute 
\begin{equation}\label{eq:loc_law_average_aux2}
 \frac{1}{N}\sum_{i=1}^N c_i^* (G_{ii}-m_i)  = \cf\cdot (\gf-\mf) = (\mf^*(\Lf^{-1})^*[\cf])\cdot (\df\gf + \Sf[\gf-\mf](\gf-\mf)). 
\end{equation}
We rewrite the term $\df\gf$ next. Indeed, a straightforward computation starting from the Schur complement formula \eqref{eq:Schur_complement_formula} 
shows that 
\begin{equation}\label{eq:loc_law_average_aux3}
 d_{i} G_{ii}= \Big(Q_i \frac{1}{G_{ii}} \Big)G_{ii} + (d^{(7)}_{i} +  d^{(8)}_{i})G_{ii}  = 
\Big(Q_i \frac{1}{G_{ii}} \Big)m_i+ \Big(Q_i \frac{1}{G_{ii}} \Big)\left( G_{ii}-m_i\right) + (d^{(7)}_{i} +  d^{(8)}_{i})G_{ii},
\end{equation}
where we defined $Q_i Z \defeq Z - \E_i Z$ and the conditional expectation 
\[ \E_i Z \defeq \E [ Z | \Hf^{\{i\}}]=\E\big[ Z \big| \{ x^\mu_{kl}, y^\nu_{kl} \colon k, l \in [N]\setminus \{i\}, \mu,\nu \in [\ell]\} \big] \] 
for any random variable $Z$.

The advantage of the representation \eqref{eq:loc_law_average_aux3} is that we can apply the following proposition to the first term on the right-hand side.
It shows that when $Q_i(1/G_{ii})$ is averaged in $i$, there are certain cancellations taking place such that the average has a smaller order than $Q_i (1/G_{ii}) = O(\Lambda)$. 
The first statement of this type was proved for generalized Wigner matrices in \cite{EYYBern}. The complete proof in our setup will be presented in Section \ref{sec:FA}. 

\begin{pro}[Fluctuation Averaging] \label{pro:fluctuation_averaging}
Let $\Phi$ be a deterministic control parameter such that $  0< \Phi \le N^{-\eps}$. 
If 
\bels{FA assumption}{
\max_{i , j}\absbb{\frac{1}{m_i}G_{ij}-\delta_{ij}} \,\prec\, \Phi\, ,
}
then for any deterministic $c_1, \ldots, c_N \in \C^{K\times K}$ satisfying $\max_i \abs{c_i}\le 1$ we have 
\bels{FA bound}{
\absbb{\frac{1}{N}\sum_{i=1}^N c_i\2Q_i \frac{1}{G_{ii}}m_i}\,\prec\, \max\cbb{\frac{1}{\sqrt{N}},\Phi}\Phi\,.
}
\end{pro}
Note that the assumption \eqref{eq:assumption_averaged_local_law} directly implies \eqref{FA assumption}. Moreover, \eqref{FA assumption} yields 
\[
\absbb{\Big(Q_i \frac{1}{G_{ii}} \Big)\left( G_{ii}-m_i\right)}\,\leq \,  
\absbb{Q_i\Big( \frac{1}{G_{ii}} m_i-\id\Big)}\norm{\mf^{-1}}\Lambda\,\prec\, \Phi^2.
\]
Thus, we obtain from \eqref{eq:loc_law_average_aux2} and \eqref{eq:loc_law_average_aux3} 
the relation
\begin{equation} \label{eq:loc_law_average_aux4}
\abs{\cf\cdot(\gf-\mf)} \prec \norm{(\Lf^{-1})^*} \norm{\mf} \left(\frac{1}{N} \absB{\sum_{i=1}^N \wt{c}_i Q_i\frac{1}{G_{ii}} m_i} + \Phi^2 + 
\max_{i=1}^N (\abs{d_i^{(7)}} + \abs{d_i^{(8)}})\abs{G_{ii}} + \norm{\Sf} \Lambda^2 \right), 
\end{equation}
where $\wt{\cf}=(\wt{c}_1, \ldots, \wt{c}_N) \in(\C^{K\times K})^N$ is a multiple of $\mf^*(\Lf^{-1})^*[\cf]$ and $\norm{\wt{\cf}} \leq 1$.
From this estimate, we now conclude \eqref{eq:local_law_average_general}. 
Since \eqref{FA assumption} is satisfied by \eqref{eq:assumption_averaged_local_law} 
the bound \eqref{FA bound} implies that the first term on the right-hand side of \eqref{eq:loc_law_average_aux4} 
is controlled by the right-hand side of \eqref{eq:local_law_average_general}. For the third term, we use 
\eqref{eq:initial_error_estimates_bound_G2} and $\abs{G_{ii}} \leq \norm{\mf} + \Phi/\norm{\mf^{-1}}$ 
as well as $\Phi \leq 1 \leq \norm{\mf}\norm{\mf^{-1}}$. 
Hence, \eqref{eq:norm_Sf} concludes the proof of~\eqref{eq:local_law_average_general} and Lemma \ref{lem:avg_tech_local_law}.
\end{proof}

\subsection{No eigenvalues away from self-consistent spectrum} \label{subsec:local_law}

We now state and prove our result for Hermitian Kronecker matrices $\Hf$, Theorem~\ref{thm:no_eigenvalues} below.
The theorem has two parts. 
For simplicity, we state the first part under the condition that $\Af=\sum_i a_i \otimes E_{ii}$ is bounded.  
We relax this condition in the second part 
 for the purpose of our main result, Theorem~\ref{thm:no_eigenvalues_non_hermitian}. 
In this application, $\Af =\Af^\zeta=\sum_i a_i^{\zeta} \otimes E_{ii}$, where $a_i^\zeta$ are given in \eqref{eq:mapping_non_hermitian_to_hermitian2},
and we need to deal with unbounded $\zeta$ as well.

We recall that $\mf=(m_1, \ldots, m_N)$ is the unique solution of \eqref{eq:Dyson} with positive imaginary part.
Moreover, the function $\rho\colon \Hb \to \R_+$ was defined in \eqref{eq:def_rho}, the set $\supmeas$ 
in Definition \ref{def:self-consistent_density_of_states} and $\disvz \defeq \dist(z,\supmeas)$.  
We denote $E \defeq \Re z$ and $\eta \defeq \Im z$.
For a matrix $\Bf$, we write $\smin(\Bf)$ to denote its smallest singular value.

\begin{thm}[No eigenvalues away from $\supmeas$] \label{thm:no_eigenvalues}
Fix $K\in \N$. 
Let $\Af=\sum_{i=1}^N a_i \otimes E_{ii}$ be a Hermitian matrix 
and $\Hf$ be a Hermitian Kronecker random matrix as in \eqref{eq:def_Hf} 
such that \eqref{eq:upper_bound_variances},  \eqref{eq:moments_bounds} and \eqref{eq:bound_coefficients} are satisfied.
\begin{enumerate}[(i)]
\item Assume that $\Af$ is bounded, i.e.,  $\normtwo{\Af}\le\kappa_4$. Then
 there is a universal constant $\delta >0$ such that
 for each $D>0$, there is a constant $C_{D} >0$ such that 
\begin{equation} \label{eq:no_eigenvalues_hermitian}
 \P \Big( \spec(\Hf)\subset \{ \tau \in\R \colon \dist(\tau, \supmeas) \leq N^{-\delta} \} \Big) \geq 1 - \frac{C_{D}}{N^D}. 
\end{equation}
\item Assume now only the weaker bound 
\begin{equation}
\normtwo{\Af} = \max_{i=1}^N \abs{a_i} \leq N^{\kappa_7} \label{eq:A_two_bounded} 
\end{equation}
Let $\Hbouttwo$ be defined through
\begin{equation}\label{out}
\Hbouttwo   \defeq  \bigg\{ w \in\Hb\colon \dist(w,\spec\Af) \geq 2 \norm{\Sf}^{1/2} +1,~ \frac{\norm{\Af-w\id}_2}{\smin(\Af-w\id)} \leq \kappa_9 \bigg\}.
\end{equation}
Then for each $D >0$, there is a constant $C_D >0$ such that 
\begin{equation} \label{eq:no_eigenvalues_hermitian1}
 \P \Big( \spec(\Hf)\cap \Hbouttwo =\emptyset \Big) \geq 1 - \frac{C_{D}}{N^D}. 
\end{equation}
\end{enumerate}
The constants $C_{D}$ in \eqref{eq:no_eigenvalues_hermitian} and \eqref{eq:no_eigenvalues_hermitian1} 
only depend on $K$, $\kappa_1$, $(\varphi_p)_{p\geq 3}$, $\alpha_*$, $\kappa_4$, $\kappa_7$  and $\kappa_9$
in addition to $D$.
\end{thm}

We will prove Theorem~\ref{thm:no_eigenvalues} as a consequence of the following Lemma~\ref{lem:local_law_aux_bound}.
This lemma is a type of local law. 
Its general comprehensive version, Lemma~\ref{lem:local_law} below,
is a standard application of Lemma~\ref{lem:Gap_Lemma}, Lemma~\ref{lem:avg_tech_local_law} and 
Proposition~\ref{pro:fluctuation_averaging}.
For the convenience of the reader, we will give 
an outline of the proof in Appendix \ref{app:proof_local_law}.

We also consider $\kappa_7, \kappa_8, \kappa_9$ 
from \eqref{eq:A_two_bounded} and \eqref{eq:def_Hbout} below, respectively, as model parameters.

\begin{lem} \label{lem:local_law_aux_bound}
Fix $K\in \N$. Let $\kappa_7>0$ and 
$\Af=\sum_{i=1}^N a_i \otimes E_{ii}$ be a Hermitian matrix such that \eqref{eq:A_two_bounded} holds true.  
Let $\Hf$ be a Hermitian Kronecker random matrix as in \eqref{eq:def_Hf} 
such that \eqref{eq:upper_bound_variances},  \eqref{eq:moments_bounds} and \eqref{eq:bound_coefficients} are satisfied.
We define 
\begin{subequations} \label{eq:def_Hbout}
\begin{align} 
\Hboutone & \, \defeq \,\Big \{w \in\Hb \colon \dist(w, \spec \Af) \leq 2 \norm{\Sf}^{1/2} +1,~ \norm{\Af}_2 \leq \kappa_8 \Big\},
\label{eq:def_Hboutone}\\
\Hbouttwo & \, \defeq \, \bigg\{ w \in\Hb\colon \dist(w,\spec\Af) \geq 2 \norm{\Sf}^{1/2} +1,~ \frac{\norm{\Af-w\id}_2}{\smin(\Af-w\id)} \leq \kappa_9 \bigg\}.
\label{eq:def_Hbouttwo} 
\end{align}
\end{subequations}
Then there are $p\in\N$ and $P\in\N$ independent of $N$ and the model parameters such that 
\begin{equation}  \label{eq:average_local_law_aux2}
 \absB{\frac{1}{N}\sum_{i=1}^N \Tr\Im (G_{ii}(z) - m_i(z))} \prec \max\Big\{ 1, \frac{1}{\disvz[P]} \Big\} \Big(\frac{1}{N } + \frac{1}{(N\eta)^2} \Big) 
\end{equation}
for any $z=E+\ii\eta\in\Hboutone\cup\Hbouttwo$ such that $\abs{E} \leq N^{\kappa_7+1}$ and $\eta\geq N^{-1+\gamma}(1+\disvz[-p])$. 
\end{lem}

 We remark that since $\Af$ is Hermitian, if $\| \Af\|_2$ is bounded, then the second condition in \eqref{eq:def_Hbouttwo} is
automatically satisfied (perhaps with a larger $\kappa_9$), given the first one. So for $\| \Af\|_2\le\kappa_8$, alternatively, we could
have  defined the sets
\begin{equation}\label{alt}
\begin{aligned} 
\Hboutone   & \defeq \Big \{w \in\Hb \colon \dist(w, \spec \Af) \leq 2 \norm{\Sf}^{1/2} +1 \Big\},\\
\Hbouttwo   & \defeq  \bigg\{ w \in\Hb\colon \dist(w,\spec\Af) \geq 2 \norm{\Sf}^{1/2} +1 \bigg\}.
\end{aligned} 
\end{equation}
If $\| \Af\|_2$ does not have an $N$-independent bound, then we could have  defined
$\Hboutone   \defeq \emptyset$ and $\Hbouttwo$ as in \eqref{out}. 
The estimate \eqref{eq:average_local_law_aux2} holds as stated with these alternative definitions of $\Hboutone$ and 
$\Hbouttwo$. 

\begin{defi}(Overwhelming probability) We say that an event $A^{(N)}$ happens \emph{asymptotically with overwhelming probability}, a.w.o.p.,  
if for each $D>0$ there is $C_D>0$ such that for all $N \in \N$, we have
\[ \P \big(A^{(N)} \big) \geq 1- \frac{C_D}{N^D}. \]
\end{defi}

\begin{proof}[Proof of Theorem~\ref{thm:no_eigenvalues}]  
From \eqref{eq:trivial_control_x_ij_prec}, we conclude the crude bound 
\begin{equation} \label{eq:no_eigenvalues_far_outside}
 \max_{\lambda \in \spec \Hf} \abs{\lambda}^2 \leq \Tr(\Hf^2) = \sum_{i,j=1}^N \abs{h_{ij}}^2 \prec (1 +\normtwo{\Af}^2) N.  
\end{equation}
Therefore, there are a.w.o.p. no eigenvalues of $\Hf$ outside of $[-a,a]$ with $a \defeq (1+\normtwo{\Af})\sqrt{N}$. 

We introduce the set $A_\delta \defeq \{ \omega \in\R\colon \dist(\omega, \supmeas) \geq N^{-\delta}\}$ for $\delta>0$. 
The previous argument proves that there are no eigenvalues in $A_\delta\setminus[-a,a]$ for any $\delta >0$. For the opposite regime, 
i.e. to show  that $A_\delta \cap [-a,a]$ does not contain any eigenvalue of $\Hf$ a.w.o.p.
with some small $\delta>0$, we use the following standard lemma and will include a proof 
for the reader's convenience at the end of this section.

\begin{lem} \label{lem:initial_estimate_no_eigenvalues_outside}
Let $\Hf$ be an arbitrary Hermitian random matrix and $\Gf(z) \defeq (\Hf-z\id)^{-1}$ its resolvent at $z\in\Hb$.
Let $\Phi \colon \Hb \to \R_+$ be a deterministic (possibly $N$-dependent) control parameter such that 
\begin{equation} \label{eq:no_eigenvalues_outside_condition}
 {\frac{1}{N} \Im \Tr \Gf(\tau +\ii \eta_0)} \prec \Phi(\tau +\ii \eta_0) 
\end{equation}
for some   $\tau\in \R$ and $\eta_0  > 0$.   
\begin{enumerate}[(i)]
\item If $ (N\eta_0)^{-1} \geq { N^\eps}\Phi(\tau+\ii\eta_0)$ for some $\eps>0$  then 
$\spec(\Hf) \cap [\tau-\eta_0, \tau + \eta_0] = \varnothing$ a.w.o.p. 
\item 
Let $\mathcal{E} \defeq \{ \tau \in [-N^C,N^C] \colon (N\eta_0)^{-1} \geq { N^\eps}\Phi(\tau +\ii\eta_0) \}$ for some $C>0$ and $\eps>0$. Furthermore,
 suppose that   $\eta_0\ge N^{-c}$ for some $c>0$ and \eqref{eq:no_eigenvalues_outside_condition} holds uniformly for all  $\tau\in\mathcal{E}$. 
Then $\spec(\Hf) \cap \mathcal{E} = \varnothing$ a.w.o.p. 
\end{enumerate}
\end{lem} 
We now finish the proof of Theorem~\ref{thm:no_eigenvalues}. In fact, by \eqref{eq:A_two_bounded} we have $a \lesssim N^{\kappa_7+1/2}$,
thus we work in the regime $|E|\le N^{\kappa_7+1}$.  We choose 
$$
\Phi(z) \defeq \rho(z)+ \max\{1, \disvz[-P]\}\Big(\frac{1}{N} + \frac{1}{(N\im z)^2}\Big) \quad \text {and} \quad \eta_0\defeq N^{-2/3}.
$$
For small enough $\delta$ and $\gamma$, we can assume that $\eta_0 \geq N^{-1+\gamma}(1+ \dist(\tau+\ii\eta_0,\supmeas)^{-p})$ for $\dist(\tau, \supmeas) \geq N^{-\delta}$. 
 Consider first the case when $\| \Af\|_2\le \kappa_4$, then $\Hboutone$ and $\Hbouttwo$ are complements of each other, 
see the remark at~\eqref{alt},
and then
 \eqref{eq:no_eigenvalues_outside_condition} is satisfied by \eqref{eq:average_local_law_aux2} 
for any $\tau$ with $|\tau|\le N^{\kappa_7+1}$. 
Moreover, owing to \eqref{eq:bound_rho_dist_support}, we have
\[ \Phi(E +\ii\eta_0) \lesssim \frac{ N^{2\delta}}{N^{2/3}} + {N^{P\delta}}\Big(\frac{1}{N} + \frac{1}{N^{2/3}}\Big) \]
for all $E \in A_\delta\cap[-a,a]$.
Therefore, by possibly reducing $\delta>0$ and introducing a sufficiently small $\eps>0$, we can assume $ N^\eps \Phi(E +\ii\eta_0) \leq N^{-1/3}=(N\eta_0)^{-1}$. 
Thus, from Lemma \ref{lem:initial_estimate_no_eigenvalues_outside} we infer that $\Hf$ does not have any eigenvalues in $A_\delta\cap [-a,a]$ a.w.o.p. 
Combined with the argument preceding Lemma \ref{lem:initial_estimate_no_eigenvalues_outside}, which excludes a.w.o.p.
eigenvalues of $\Hf$ in $A_\delta\setminus[-a,a]$, this proves~\eqref{eq:no_eigenvalues_hermitian}  if $\| \Af\|_2\le \kappa_4$.
Under the weaker assumption $\| \Af\|_2\le N^{\kappa_7}$ the same argument works but only for $E\in \Hbouttwo$
since \eqref{eq:average_local_law_aux2} was proven only in this regime.
\end{proof}

\begin{proof}[Proof of Lemma \ref{lem:initial_estimate_no_eigenvalues_outside}]
For the proof of part (i), we compute 
\[ \frac{1}{N} \Im \Tr \Gf(\tau +\ii\eta) = \frac{1}{N} \sum_i \frac{\eta}{(\lambda_i - \tau)^2 +\eta^2 }. \]
Estimating the maximum from above by the sum, we obtain from the previous identity and the assumption that
\begin{equation} \label{eq:no_eigenvalues_outside_aux1}
 \frac{1}{N} \max_i \frac{\eta_0}{(\lambda_i - \tau)^2 +\eta_0^2 } \prec \Phi \leq \frac{N^{-\eps}}{N\eta_0}. 
\end{equation}
We conclude that $\min_i \abs{\lambda_i - \tau} \geq \eta_0$ a.w.o.p. 
and hence (i) follows. 

The part (ii) is an immediate consequence of (i) and a union bound argument using the Lipschitz-continuity in $\tau$ on $\mathcal{E}$ of the left-hand side of \eqref{eq:no_eigenvalues_outside_aux1} with
Lipschitz-constant bounded by $N^{3(C+c)}$ and the boundedness of $\mathcal{E}$, i.e., $\mathcal{E} \subset [-N^C, N^C]$. 
\end{proof}

\section{Fluctuation Averaging: Proof of Proposition \ref{pro:fluctuation_averaging}} \label{sec:FA}

In this section, we prove the Fluctuation Averaging which was stated as Proposition \ref{pro:fluctuation_averaging} in the previous section. 

\begin{proof}[Proof of Proposition \ref{pro:fluctuation_averaging}]

We fix an even $ p \in \N$ and use the abbreviation
\[
Z_i\,\defeq\, c_i\2Q_i \frac{1}{G_{ii}}m_i\,.
\]
We will estimate the $p$-th moment of $\frac{1}{N}\sum_i Z_i$. 
For a $p$-tuple $\bs{i}=(i_1, \dots, i_{p}) \in \{1, \dots,N\}^p$ we call a label $i_l$ a \emph{lone label} if it appears only once in $\bs{i}$. We denote by $J_L$ all tuples $\bs{i} \in \{1, \dots,N\}^{p}$ with exactly $L$ lone labels. Then we have
\bels{FA start}{
\E\,\absbb{\frac{1}{N}\sum_{i=1}^N \2Z_i}^p\,\le\,\frac{1}{N^p} \sum_{L=0}^p \sum_{\bs{i} \in J_L} \abs{\E\2 Z_{i_1}\dots Z_{i_{p/2}} \ol{Z_{i_{p/2+1}}\dots Z_{i_p}}}\,.
}
For $\bs{i} \in J_L$ we estimate
\bels{FA estimate}{
\abs{\E\2 Z_{i_1}\dots Z_{i_{p/2}} \ol{Z_{i_{p/2+1}}\dots Z_{i_p}}}\,\prec\, \Phi^{p+L}.
}
Before verifying \eqref{FA estimate} we show this bound is sufficient to finish the proof. Indeed, using $\abs{J_L}\le C(p)N^{(L+p)/2}$ and \eqref{FA estimate} in \eqref{FA start} yields
\[
\E\,\absbb{\frac{1}{N}\sum_{i=1}^N \2Z_i}^p\,\prec\,\sum_{L=0}^p N^{-(p-L)/2}\Phi^{p+L} \,\prec\,\pbb{ \max\cbb{\frac{1}{\sqrt{N}},\Phi}\Phi}^{p}.
\]
This implies \eqref{FA bound}. 

The rest of the proof is dedicated to showing \eqref{FA estimate}. Since the complex conjugates do not play any role in the following arguments, we  omit them in our notation.  Furthermore, by symmetry we may assume that $\{i_1, \dots, i_L\}$ are the lone labels in $\bs{i}$.

We we fix $\ell \in \{0, \dots,L\}$ and $l \in \{1, \dots,p\}$. For any $K \in \N_0$ we call a pair 
\[
(\bs{t},\bs{T})\qquad \text{with} \qquad \bs{t}\,=\,(t_1, \dots, t_{K-1}) \,,\quad \bs{T}\,=\,(T_{0}, T_{01}, T_1, T_{12},\dots,T_{K-1}, T_{K-1K}, T_{K} )\,,
\]
an $l$-\emph{factor} (\emph{at level} $\ell$) if for all $k \in \{1, \dots,K-1\}$ and all $k' \in \{1, \dots,K-2\}$ the entries of the pair satisfy
\bels{l-factor condition}{
&t_{{k}} \in \{i_1, \dots, i_\ell\},\qquad  T_{k}, T_{k' k'+1} \subseteq \{i_1, \dots, i_\ell\}\,,
\\ 
&t_{k'} \ne t_{k'+1}\,,\quad  t_{k} \not \in T_k\,,\quad t_{k'}, t_{k'+1} \not \in T_{k'k'+1}\,,\quad t_1 \neq i_l\,,,\quad t_{K-1} \neq i_l \,,\quad i_l \not \in T_0 \cup T_{K+1}\,.
}
Then we associate to such a pair the expression
\bels{l-factor expression}{
Z_{\bs{t},\bs{T}}\,\defeq\, c_{i_l}\2Q_{i_l}\sbb{ \frac{1}{G_{i_li_l}^{T_0}}G_{i_lt_1}^{T_{01}} \frac{1}{G_{t_1t_1}^{T_1}}G_{t_1t_2}^{T_{12}} \frac{1}{G_{t_2t_2}^{T_2}}\dots \frac{1}{G_{t_{K-1}t_{K-1}}^{T_{K-1}}}G_{t_{K-1}i_l}^{T_{K-1 K}} \frac{1}{G_{i_li_l}^{T_{K}}} }m_{i_l}\,.
}
In particular, for $K=0$ we have
\[
Z_{\emptyset,(T_0)}\,\defeq\, c_{i_l}\2Q_{i_l} \frac{1}{G_{i_li_l}^{T_0}}m_{i_l}\,,\qquad Z_{\emptyset,(\emptyset)}\,\defeq\, Z_{i_l}\,.
\]
We also call
\[
d(\bs{t},\bs{T})\,\defeq\, K\,,
\]
the \emph{degree} of the $l$-factor $(\bs{t},\bs{T})$.

By induction on $\ell$ we now prove the identity
\bels{Expanded formula}{
\E\2 Z_{i_1}\dots Z_{i_p}\,= \sum_{(\ul{\bs{t}},\ul{\bs{T}})\in \cal{I}_\ell}(\pm) \2\E\2 Z_{\bs{t}_1,\bs{T}_1}\dots Z_{\bs{t}_p,\bs{T}_p}\,,
}
where the sign $(\pm)$ indicates that each summand may have a coefficient $+1$ or $-1$ and the sum is over a set $\cal{I}_\ell$ that contains pair of $p$-tuples $\ul{t}=(\bs{t}_1, \dots, \bs{t}_p)$ and $\ul{T}=(\bs{T}_1, \dots, \bs{T}_p)$ such that $(\bs{t}_l,\bs{T}_l)$ for all $l=1, \dots,p$ is an $l$-factor at level $\ell$. Furthermore, for all $\ell\in \{0, \dots,L\}$ the size of $\cal{I}_\ell$ and the maximal degree of the $l$-factors $(\bs{t}_l,\bs{T}_l)$ are bounded by a constant depending only on $p$ and 
\bels{total degree lower bound}{
\sum_{i=1}^p\max\{1,d(\bs{t}_l,\bs{T}_l)\} \,\ge\, p+\ell\,,\qquad (\ul{\bs{t}},\ul{\bs{T}})\in \cal{I}_\ell\,. 
}

The bound \eqref{FA estimate} follows from \eqref{Expanded formula} and \eqref{total degree lower bound} for $\ell=L$ because 
\bels{bound on l-factor}{
\abs{Z_{\bs{t},\bs{T}}}\,\prec\, \Phi^{\max\{1,d(\bs{t},\bs{T})\}},
}
for any $l$-factor $(\bs{t},\bs{T})$. We postpone the proof of \eqref{bound on l-factor} to the very end of the proof of Proposition \ref{pro:fluctuation_averaging}.

The start of the induction for the proof of \eqref{Expanded formula} is trivial since for $\ell=0$  we can chose the set $\cal{I}_\ell$ to contain only one element with $(\bs{t}_l,\bs{T}_l)=(\emptyset, (\emptyset))$ for all $l=1, \dots,p$. For the induction step, suppose that \eqref{Expanded formula} and \eqref{total degree lower bound} have been proven for some $\ell \in \{1, \dots, L-1\}$. Then we expand all $l$-factors $(\bs{t}_l,\bs{T}_l)$ with $l \ne \ell+1$ within each summand on the right hand side of \eqref{Expanded formula} in the lone index $i_{\ell+1}$ by using the formulas 
\begin{subequations}
\label{expansion step}
\begin{align}
\label{off-diagonal expansion step}
 G^T_{ij}\,&=\,G_{ij}^{T\cup\{k\}}+G_{ik}^T\frac{1}{G_{kk}^T}G_{kj}^T\,,\qquad &i,j  \notin  \{k\} \cup T\,,
 \\
 \label{diagonal expansion step}
\frac{1}{G_{ii}^T}\,&=\,  \frac{1}{G^{T\cup\{k\}}_{ii}}-\frac{1}{G_{ii}^T}G_{ik}^T\frac{1}{G_{kk}^T}G_{ki}^T \frac{1}{G^{T \cup\{k\}}_{ii}}\,,\qquad &i  \notin \{k\}\cup T\,,
\end{align}
\end{subequations}
for $k=i_{\ell+1}$. More precisely, for all $l \neq \ell+1$ we use \eqref{expansion step} on each factor on the right hand side of \eqref{l-factor expression} with $(\bs{t},\bs{T})=(\bs{t}_l,\bs{T}_l)$; \eqref{off-diagonal expansion step} for the off-diagonal and \eqref{diagonal expansion step} for the inverse diagonal resolvent entries. Multiplying out the resulting factors, we write $\E\2 Z_{\bs{t}_1,\bs{T}_1}\dots Z_{\bs{t}_p,\bs{T}_p}$ as a sum of 
\[
2^{\sum_{l \neq \ell+1}2d(\bs{t}_l,\bs{T}_l)+1}
\]
summands of the form 
\bels{summand after expanding}{
\E\2 Z_{\wt{\bs{t}}_1,\wt{\bs{T}}_1}\dots Z_{\wt{\bs{t}}_p,\wt{\bs{T}}_p}\,,
}
where for all $l=1, \dots, p$ the pair $(\wt{\bs{t}}_l,\wt{\bs{T}}_l)$ is an $l$-factor at level $\ell+1$. Note that we did not expand the $\ell+1$-factor $Z_{\bs{t}_{\ell+1}, \bs{T}_{\ell+1}}$. In particular, the only nontrivial conditions for $(\wt{\bs{t}}_l,\wt{\bs{T}}_l)$ to be an $l$-factor at level $\ell+1$ (cf. \eqref{l-factor condition}), namely $t_{k} \ne t_{k+1}$, $t_{1} \ne i_{\ell+1}$ and $t_{K-1} \ne i_{\ell+1}$, are satisfied because $i_{\ell+1}$ does not appear as a lower index on the right hand side of \eqref{l-factor expression} when on the left hand side $(\bs{t},\bs{T})=(\bs{t}_l,\bs{T}_l)$. 

 Moreover all but one of the summands \eqref{summand after expanding} satisfy 
\[
\sum_{i=1}^pd(\wt{\bs{t}}_l,\wt{\bs{T}}_l) \,\ge\, p+\ell+1\,,
\]
because the choice of the second summand in both \eqref{off-diagonal expansion step} and \eqref{diagonal expansion step} increases the number of off-diagonal resolvent elements in the $l$-factor that is expanded. 
The only exception is the summand \eqref{summand after expanding} for which in the expansion in all factors always the first summand of \eqref{off-diagonal expansion step} and \eqref{diagonal expansion step} is chosen. However, in this case all $Z_{\wt{\bs{t}}_l,\wt{\bs{T}}_l}$ with $l\ne \ell+1$ are independent of $i_{\ell+1}$ because this lone index has been completely removed from all factors. We conclude that this particular summand vanishes identically. Thus \eqref{total degree lower bound} holds with $\ell$ replaced by $\ell+1$ and the induction step is proven.  

It remains to verify  \eqref{bound on l-factor}. For $d(\bs{t},\bs{T})=0$ we use that
\bels{K=0 estimate on Z}{
 \absbb{Q_{i_l}\frac{1}{G_{i_li_l}}m_{i_l}}\,\le\, \absbb{\frac{1}{G_{i_li_l}}m_{i_l}-\id}\,\prec\, \Phi
\,,\qquad
\absbb{\frac{1}{G_{i_li_l}^{T}}m_{i_l}- \frac{1}{G_{i_li_l}}m_{i_l}} \,\prec\, \Phi^2\,.
}
The first bound in \eqref{K=0 estimate on Z} simply uses the assumption \eqref{FA assumption} while the second bound uses the expansion formulas \eqref{expansion step} and \eqref{FA assumption}. For $K=d(\bs{t},\bs{T})>0$ we realize that $K$ encodes the number of off-diagonal resolvent entries $G^T_{ij}$ in \eqref{l-factor expression}.  In the factors of  \eqref{l-factor expression} we insert the entries of $\bs{M}$ so that \eqref{FA assumption} becomes usable, i.e. we use
\[
\frac{1}{G_{t_{k}t_{k}}^{T_{k}}}G_{t_{k}t_{k+1}}^{T_{k k+1}}\,=\,\frac{1}{G_{t_{k}t_{k}}^{T_{k}}}m_{t_{k}}\frac{1}{m_{t_{k}}}G_{t_{k}t_{k+1}}^{T_{k k+1}}\,.
\]
Then similarly to \eqref{K=0 estimate on Z} we use
\[
\absbb{\frac{1}{m_{t_{k}}}G_{t_{k}t_{k+1}}}\,\prec\, \Phi\,,\qquad 
\absbb{\frac{1}{m_{t_{k}}}G_{t_{k}t_{k+1}}^{T_{k k+1}}-\frac{1}{m_{t_{k}}}G_{t_{k}t_{k+1}}}\,\prec\, \Phi^2\,,
\] 
where again the first bound follows from \eqref{FA assumption} and the second bound from \eqref{expansion step} and \eqref{FA assumption}.
\end{proof}

\section{Non-Hermitian Kronecker matrices and proof of Theorem~\ref{thm:no_eigenvalues_non_hermitian}}

Since $\spec(\Xf) \subset \spec_\eps(\Xf)$ (cf. \eqref{eq:def_pseudospectrum}) 
for all $\eps>0$, Theorem \ref{thm:no_eigenvalues_non_hermitian} 
clearly follows from the following lemma.
 
\begin{lem}[Pseudospectrum of $\Xf$ contained in self-consistent pseudospectrum]
\label{lem:pseudospectrum_subset_self_consistent}
Under the assumptions of Theorem~\ref{thm:no_eigenvalues_non_hermitian}, we have that
for each $\eps\in(0,1]$, $\Delta>0$ and $D>0$, there is a constant $C_{\eps,\Delta,D} >0$ such that 
\begin{equation} \label{eq:pseudospectrum_subset_self_consistent}
\P \big( \spec_{\eps}(\Xf) \subset \supnon_{\eps+\Delta} \big) \geq 1- \frac{C_{\eps,\Delta,D}}{N^D}.
\end{equation}
\end{lem}

\begin{proof}
Let $\Hf^\zeta$ be defined as in \eqref{eq:def_Hf_zeta}. Note that $\zeta \in \spec_{\eps}(\Xf)$
if and only if $\dist(0,\spec(\Hf^\zeta)) \leq \eps$. 
 We set 
 \bels{def of wt A}{
 \wt{\Af} \defeq \sum_{i=1}^N \wt{a}_i \otimes E_{ii}.
 }
We first establish that $\spec_\eps(\Xf)$ is contained in $D(0,N)\defeq \{ w \in \C \colon \abs{w} \leq N\}$ a.w.o.p. 
Similarly, as in \eqref{eq:no_eigenvalues_far_outside}, using an analogue of \eqref{eq:trivial_control_x_ij_prec} 
for $\Xf$ instead of $\Hf$, we get
\[ \max_{\zeta\in\spec\Xf} \abs{\zeta}^2 \leq \Tr(\Xf^*\Xf) = \sum_{i,j=1}^N \Tr\big((P_{ij}\Xf)^* (P_{ij}\Xf)\big)\lesssim 
\sum_{i,j=1}^N \abs{P_{ij}\Xf}^2 \prec (1+ \normtwo{\wt{\Af}}^2)N. \]
Thus, all eigenvalues of $\Xf$ have a.w.o.p. moduli smaller than $(1+\normtwo{\wt{\Af}})\sqrt N\leq N$. The above characterization of $\spec_\eps(\Xf)$ and $\eps \leq 1$ yield $\spec_\eps(\Xf) \subset D(0,N)$ a.w.o.p.

We now fix an $\eps\in (0,1]$ and for the remainder of the proof the comparison relation $\lesssim$ is allowed to depend on $\eps$ 
without indicating that in the notation.  
In order to show that the complement of $\spec_\eps(\Xf)$ contains $\supnon_{\eps+\Delta}^c \cap D(0,N)$ a.w.o.p. we will apply 
 Theorem~\ref{thm:no_eigenvalues} 
to~$\Hf^\zeta$ for $\zeta \in\supnon_{\eps+\Delta}^c \cap D(0,N)$. In particular, here we have 
\[
\Af\,=\, \Af^\zeta\,\defeq\, \sum_{i} a_i^\zeta \otimes E_{ii}\,,
\]
where $a_i^\zeta$ is defined as in \eqref{eq:mapping_non_hermitian_to_hermitian2}. 

Now, we conclude that
$\spec(\Hf^\zeta)\cap [-\eps-\Delta/2,\eps+\Delta/2] = \varnothing$ a.w.o.p. for each $\zeta \in {\supnon}_{\eps+\Delta}^c\cap D(0,N)$.
If $\zeta$ is bounded, hence $\Af^\zeta$ is bounded, we can use 
\eqref{eq:no_eigenvalues_hermitian} and we need to show that 
$ [-\eps-\Delta/2,\eps+\Delta/2] \subset \{ \tau \in\R \colon \dist(\tau, \supmeas^\zeta) \ge N^{-\delta} \}$
but this is straightforward since $\zeta\in  {\supnon}_{\eps+\Delta}^c$ implies $\dist(0, \supmeas^\zeta)\ge \eps+\Delta$
by its definition.

For large $\zeta$ we use part (ii) of Theorem~\ref{thm:no_eigenvalues} and we need to show 
that  $[-\eps-\Delta/2,\eps+\Delta/2]+\ii \eta \subset \Hbouttwo$ for any small $\eta$. 
Take $z\in \Hb$ with $|z|\le \eps+\Delta/2$. 
If $|\zeta| \ge \| \wt{\Af}\| + 2\norm{\Sf}^{1/2} +2$, then $\dist(z,\spec(\Af^\zeta))\ge 2\norm{\Sf}^{1/2} +1$,
so the first condition in the definition \eqref{eq:def_Hbouttwo} of $\Hbouttwo$ is satisfied.
The second condition is straightforward since for large $\zeta$ and small $z$, both $\norm{\Af^\zeta-z\id}_2$ and 
$\smin(\Af^\zeta-z\id)$ are comparable with $|\zeta|$.

Hence, Theorem~\ref{thm:no_eigenvalues} is applicable and we conclude that 
$\spec(\Hf^\zeta)\cap [-\eps-\Delta/2,\eps+\Delta/2] = \varnothing$ a.w.o.p. for all $\zeta \in {\supnon}_{\eps+\Delta}^c$.
If $\lambda_1(\zeta)\leq\ldots\leq \lambda_{2LN}(\zeta)$ denote the ordered eigenvalues of $\Hf^\zeta$ then $\lambda_i(\zeta)$ is Lipschitz-continuous in $\zeta$
 by the Hoffman-Wielandt inequality. 
Therefore, introducing a grid in $\zeta$ and applying a union bound argument yield
\begin{equation*}
\sup_{\zeta\in{\supnon}_{\eps+\Delta}^c\cap D(0,N)} \dist(0,\spec(\Hf^\zeta)) \leq \eps  \quad \text{ a.w.o.p.}
\end{equation*}
Since $\zeta \in\spec_\eps(\Xf)$ if and only if $\dist(0,\spec(\Hf^\zeta))\leq \eps$ we obtain $\spec_\eps(\Xf)\cap {\supnon}_{\eps+\Delta}^c \cap D(0,N) = \varnothing$ a.w.o.p. 
As we proved $\spec_\eps(\Xf)\cap D(0,N)^c = \varnothing$ a.w.o.p. before this concludes the proof of 
Lemma~\ref{lem:pseudospectrum_subset_self_consistent}.
\end{proof}

\appendix
\section{An alternative definition of the self-consistent $\eps$-pseudospectrum} \label{app:alt_def_pseudo}

Instead of the self-consistent $\eps$-pseudospectrum ${\supnon}_\eps$ introduced in \eqref{eq:def_support_non_hermitian} one may work with the deterministic set $\wt{\supnon}_\eps$ from \eqref{def of wt D eps} when formulating our main result, Theorem~\ref{thm:no_eigenvalues_non_hermitian}. The advantage of the set $\wt{\supnon}_\eps$ is that it only requires solving the Hermitized Dyson equation \eqref{eq:Dyson_non_Hermitian} for spectral parameters $z$ along the imaginary axis. The following lemma shows that ${\supnon}_\eps$ and $\wt{\supnon}_{\eps}$ are comparable in the sense that  for any $\eps$ we have ${\supnon}_{\eps_1} \subseteq \wt{\supnon}_{\eps} \subseteq {\supnon}_{\eps_2}$ for certain $\eps_1,\eps_2$.

\begin{lemma} Let $\mf$ be the solution to the Hermitized Dyson equation \eqref{eq:Dyson_non_Hermitian} and suppose Assumptions~\ref{assums:non_hermitian} are satisfied. 
There is a positive constant $c$, depending only on model parameters, such that for any $\eps \in (0,1)$ we have the inclusions
\[
\wt{\supnon}_{\eps}\,\subseteq \, {\supnon}_{\sqrt{\eps}}\,,\qquad {\supnon}_{c\2\eps^{27}}\,\subseteq \, \wt{\supnon}_{\eps}\,,
\]
where $\supnon_\eps$ is the self-consistent $\eps$-pseudospectrum from \eqref{eq:def_support_non_hermitian} and $\wt{\supnon}_\eps$ is defined in \eqref{def of wt D eps}.

\end{lemma}
\begin{proof}
The inclusion $\wt{\supnon}_{\eps}\subseteq  {\supnon}_{\sqrt{\eps}}$ is trivial because $m_j^{\zeta}$ is the Stieltjes transform of $v_j^{\zeta}$. So we concentrate on the inclusion ${\supnon}_{c\eps^{27}} \subseteq \wt{\supnon}_{\eps}$.
 We fix $\zeta \in {\C \setminus} \wt{\supnon}_\eps$ and suppress it from our notation in the following, i.e. $\mf=\mf^{\zeta}, 
v_j=v_j^{\zeta}$, etc.
Recall that by assumption we have (cf. \eqref{def of wt A})
\[
\norm{\wt \Af}\,\lesssim\, 1\,.
\]
Since any large enough $\zeta$ is contained in both sets $\C\setminus \wt{\supnon}_\eps$ and $\C \setminus \supnon_\eps$ 
 by \eqref{eq:support_V} and the upper bound in \eqref{eq:Im_Mf_upper_bound_away_support},  
we may assume that $\abs{\zeta}\lesssim 1$. 
 We use the representation of $m_i$ as the Stieltjes transform of $v_i$ and that $v_i$ has bounded support to see
\[
\abs{\scalar{x}{m_i(z)y}}\,\le\, \frac{1}{2}\int_{\R} \frac{\scalar{x}{v_i(\dd \tau)x}+\scalar{y}{v_i(\dd \tau)y}}{\abs{\tau-z}}\,\lesssim\, \frac{1}{\eta}\pb{ \scalar{x}{\im m_i(z)x}+\scalar{y}{\im m_i(z)y}}\,,
\]
for any $x,y \in \C^K$, where $K=2L$. 
In particular
\bels{bound of m in terms of im m}{
\abs{m_i(z)}\,\lesssim\, \frac{\abs{\im m_i(z)}}{\eta}\,.
}
Fix an $\eta \in (0,1)$ for which the inequality
\bels{eta inequality for wt D eps}{
\frac{1}{\eta}\norm{\im \mf(\ii \eta)} \le \frac{2}{\eps}
}
holds true. Since $\zeta \in \C\setminus \wt{\supnon}_\eps$ such an $\eta$ can be chosen arbitrarily small. Then we have
\bels{eps bounds}{
\norm{\mf(\ii\eta)}\,\lesssim\, \frac{1}{\eps}\,,\qquad \norm{\mf(\ii\eta)^{-1}}\,\lesssim\,\frac{1}{\eps}\,,\qquad \eta \,\lesssim\,  \im m_i(\ii\eta)\,\lesssim\, \frac{\eta}{\eps}\,.
}
The first inequality follows from \eqref{bound of m in terms of im m} and \eqref{eta inequality for wt D eps}, the second inequality from \eqref{eq:bound_Mf_inverse} and the third from \eqref{eta inequality for wt D eps} and the bounded support of $v_i$. 
In particular, by the formula \eqref{eq:norm_mF} for the norm of $\mF$ we have 
\bels{eps stability}{
1-\normsp{\mF(\ii\eta)}\,\gtrsim\, \eps^4\,.
}
To see \eqref{eps stability} we simply follow the calculation in the proof of Lemma~\ref{lem:properties_F_operator} but instead of using the bounds \eqref{eq:Mf_upper_bound_eta}, \eqref{eq:bound_Mf_inverse} and \eqref{eq:Im_Mf_upper_bound_away_support}  on $\norm{\mf}$ and $\norm{\mf^{-1}}$ and $\im m_i$ we use \eqref{eps bounds}. Similarly we find
\[
\norm{\cal{C}_{\Wf}}\norm{\cal{C}_{\Wf}^{-1}}\,\lesssim\, \frac{1}{\eps^3}\,,\qquad 
\norm{\cal{C}_{\sqrt{\im\Mf}}}\norm{\cal{C}_{\sqrt{\im\Mf}}^{-1}}\,\lesssim\,\frac{1}{\eps}\,.
\]
By \eqref{eq:decomposition_Lf} we conclude 
\[
\norm{\lopb^{-1}}_{\rm{sp}}\,\lesssim\, \frac{1}{\eps^8}\,.
\]
Using \eqref{eq:lopb_bounded} and the bound on $\norm{\mf}$ in \eqref{eps bounds}  we improve this bound on the $\norm{\2\cdot\2}_{\rm{sp}}$-norm  to a bound on the $\norm{\2\cdot\2}$-norm,
\[
\norm{\lopb^{-1}}\,\lesssim\, \frac{1}{\eps^{12}}\,.
\]
We are therefore in the linear stability regime of the Dyson equation and 
from the stability equation (cf. \eqref{eq:stability_equation}) for the difference $\Delta\,\defeq\mf(z)-\mf(\ii\eta)$, i.e. from
\bels{stability equation for d=mz}{
\scr{L}[\Delta]\,=\, (z-\ii\eta)\mf(\ii\eta)^2 + \frac{1}{2}\pB{\mf(\ii\eta)\scr{S}[\Delta]\Delta +\Delta\scr{S}[\Delta]\mf(\ii\eta) }\,,
}
we infer 
\[
\norm{\mf(z)-\mf(\ii\eta)}\,\lesssim\, \norm{\lopb^{-1}}\norm{\mf}^2\abs{z-\ii\eta}\,\lesssim\, \frac{\abs{z-\ii\eta}}{\eps^{14}}\,,
\]
for any $z \in \Hb$ with 
\[
\abs{z-\ii\eta}\,\le\, \frac{C}{\norm{\lopb^{-1}}^2\norm{\mf}^3}\,\lesssim\, \eps^{27}\,,
\]
where $C\sim1$ is a constant depending only on model parameters. Note that in \eqref{stability equation for d=mz} we symmetrized the quadratic term in $\Delta$ which can always be done since every other term of the equation is invariant under taking the Hermitian conjugate.
In fact, we see that $\mf$ can be extended analytically to an $\eps^{27}$-neighborhood of $\ii \eta$. Since $\eta$ can be chosen arbitrarily small we find an analytic extension of $\mf$ to all $z \in \C$ with $\abs{z} \le c\2 \eps^{27}$ for some constant $c \sim 1$. We denote this extension by the same symbol $\mf=(m_1, \dots,m_N)$ as the solution to the Dyson equation.  By definition of $\wt{\supnon}_\eps$ we have $\im m_i(0)=0$ and it is easy to see by the following argument that for any $z \in \R$ the imaginary part still vanishes as long as we are in the  linear stability regime. Thus $\rho^\zeta([-c\2 \eps^{27},c\2 \eps^{27}])=0$: The stability equation \eqref{stability equation for d=mz} evaluated at $\eta=0$ and $z \in \R$ is an equation on the space $\{\Delta \in (\C^{K \times K})^N: \Delta_i^* =\Delta_i, \, i=1, \dots,N\}$, i.e. for any $\Delta$ in this space both sides of the equation remain inside this space. Thus by the implicit function theorem applied within this subspace of $(\C^{K \times K})^N$ we conclude that the solution to \eqref{stability equation for d=mz} satisfies $\Delta=\Delta^*$, or equivalently $\im \Delta=0$, for $z \in \R$ inside the linear stability regime.  
Since $\rho^\zeta([-c\2 \eps^{27},c\2 \eps^{27}])=0$ we thus obtain $\zeta \in \C \setminus \supnon_{c \eps^{27}}$ which yields the missing inclusion. 
\end{proof}

\section{Proofs of Theorem~\ref{thm:global_law} and Lemma \ref{lem:local_law_aux_bound}} \label{app:proof_local_law}

For the reader's convenience, 
we now state and prove the local law for $\Hf$, Lemma~\ref{lem:local_law} below.
Its first part is designed for all spectral parameters $z$, where the
Dyson equation, \eqref{eq:Dyson}, is stable and its solution $\mf$ is bounded; here the  local law holds down to
the scale  $\eta =\im z \ge N^{-1+\gamma}$ that is optimal near the self-consistent spectrum.
The second part is valid away from the self-consistent  spectrum; in this regime the Dyson equation is always  stable and the local
law holds down to the real line, however the dependence of our  estimate on the distance from the spectrum is not
optimized. For the proof of Lemma~\ref{lem:local_law_aux_bound}, the second part 
is sufficient, but we also give the first part for completeness. 
For simplicity we state the first part under the  condition that $\Af=\sum_i a_i \otimes E_{ii}$ is bounded;  in the second  part 
we relax this condition to include the assumptions of Lemma~\ref{lem:local_law_aux_bound}.
From now on, we will also consider $\kappa_4, \ldots, \kappa_9$ 
from \eqref{eq:A_two_bounded}, \eqref{eq:def_Hboutone}, \eqref{eq:def_Hbouttwo} and \eqref{eq:def_Hbin} 
below, respectively, as model parameters.

\begin{lem}[Local law] \label{lem:local_law}
Fix $K\in \N$. 
Let $\Af=\sum_{i=1}^N a_i \otimes E_{ii}$ be a deterministic Hermitian matrix. Let $\Hf$ be a Hermitian random matrix as in~\eqref{eq:def_Hf} 
satisfying Assumptions \ref{assums:Hermitian_local_law}, i.e., \eqref{eq:upper_bound_variances}, \eqref{eq:moments_bounds} and \eqref{eq:bound_coefficients} hold true. 
\begin{enumerate}[(i)]
\item (Stable regime) 
Let $\gamma, \kappa_4, \kappa_5, \kappa_6>0$. Assume that $\| \Af\|_2  \le \kappa_4$ and  define
\begin{equation} \label{eq:def_Hbin}
 \Hbin \defeq \bigg\{w \in \Hb \colon \sup_{s \geq 0 } \norm{\mf(w+\ii s)} \leq \kappa_5, \;
  \sup_{s\geq 0} \normsp{\Lf^{-1}(w+\ii s) } \leq \kappa_6 \quad  \text{and}\;  \im w  \geq N^{-1+\gamma}\ \bigg\}. 
\end{equation} 
Then, we have
\begin{equation} \label{eq:local_law_inside_entry}
\max_{i,j=1}^N\absb{G_{ij}(z) -m_i(z)\delta_{ij}} \prec \, \frac{1}{1 + \eta}\sqrt{\frac{\norm{\Im \mf(z)}}{{N\eta}}} + 
\frac{1}{(1 + \eta^2) \sqrt N} + \frac{1}{(1+\eta^2) N\eta}
\end{equation}
uniformly for $z \in\Hbin$. 
Moreover, 
if $c_1, \ldots, c_N \in \C^{K\times K}$ are deterministic and satisfy $\max_{i=1}^N\abs{c_i} \leq 1$ then we have
\begin{equation} \label{eq:local_law_inside_average}
\absB{\frac{1}{N}\sum_{i=1}^N \left[c_i\left( G_{ii}(z) - m_i(z) \right)\right] } \prec \frac{1}{1+\eta}\Big(\frac{1}{N\eta} +\frac{1}{N} \Big) 
\end{equation}
uniformly for $z \in\Hbin$. 
\item (Away from the spectrum) Let $\kappa_7, \kappa_8, \kappa_9 >0$ be fixed. 
Assume that \eqref{eq:A_two_bounded} holds true
and $\Hboutone$ and $\Hbouttwo$ are defined as in \eqref{eq:def_Hbout}. 
Then there are universal  constants $\delta>0$ and $P \in \N$ such that  
\begin{equation} \label{eq:local_law_outside_entry}
\max_{i,j=1}^N\absb{G_{ij}(z) -m_i(z)\delta_{ij}} \prec \, \max\bigg\{\frac{1}{\disvz[2]}, \frac{1}{\disvz[P]}\bigg\} \frac{1}{\sqrt N} 
\end{equation}
uniformly for $z \in \big(\Hboutone \cap \{ w \in \Hb \colon \drho(w) \geq N^{-\delta}\}\big) \cup \Hbouttwo$.

Moreover, if $c_1, \ldots, c_N \in \C^{K\times K}$ are deterministic  and satisfy $\max_{i=1}^N\abs{c_i} \leq 1$ then we have
\begin{equation} \label{eq:local_law_outside_average}
\absB{\frac{1}{N}\sum_{i=1}^N \left[c_i\left( G_{ii}(z) - m_i(z) \right)\right] } \prec \max\bigg\{\frac{1}{\disvz[2]}, \frac{1}{\disvz[P]}\bigg\} \frac{1}{N}
\end{equation}
uniformly for $z \in \big(\Hboutone \cap \{ w \in \Hb \colon \drho(w) \geq N^{-\delta}\}\big)\cup \Hbouttwo$.
\end{enumerate}
\end{lem}

The local laws  \eqref{eq:local_law_outside_entry} and \eqref{eq:local_law_outside_average} hold as stated with 
the alternative definitions of the sets $\Hboutone$ and $\Hbouttwo$ given after Lemma~\ref{lem:local_law_aux_bound}.
\begin{proof}[Proof of Theorem~\ref{thm:global_law}]
Let $\mf$ be the unique solution of \eqref{eq:Dyson} with positive imaginary part, where $\alpha_\mu\defeq\wt{\alpha}_\mu$,
 $\beta_\nu \defeq 2\wt{\beta}_\nu = \wt{\beta}_\nu + \wt{\gamma}_\nu^*$ and $a_j \defeq \wt{a}_j$. 
Defining $\rho_N$ as in \eqref{eq:def_self_cons_dens_states}, it is now a standard exercise to obtain \eqref{eq:global_law} 
from~\eqref{eq:local_law_outside_average}, 
since $z \mapsto (NL)^{-1} \Tr((\Hf_N-z)^{-1})$ is the Stieltjes transform of $\mu_{\Hf_N}$. 
\end{proof}

\begin{proof}[Proof of Lemma~\ref{lem:local_law}]
We start with the proof of part (i). 
 For later use, we will present the proof for all spectral parameters $z$ in a slightly larger set than $\Hbin$, namely in the set
\begin{equation} \label{eq:def_Hbinprime}
\begin{aligned}
\Hbin' \defeq \bigg\{w \in \Hb \colon \sup_{s \geq 0 } \big(1+ \| \Af-w-\ii s\|_2\big)\norm{\mf(w+\ii s)} \leq \kappa_5, \hspace{4cm} \\
  \hspace{4cm}\sup_{s\geq 0} \normsp{\Lf^{-1}(w+\ii s) } \leq \kappa_6 \text{ and }  \im w  \geq N^{-1+\gamma}\bigg\}. 
\end{aligned}
\end{equation}
Under the condition $\normtwo{\Af}\le \kappa_4$, it is easy to see $\Hbin\subset \Hbin'$ 
perhaps with somewhat larger
$\kappa$-parameters. 
Furthermore, we relax the condition $\| \Af\|_2\le \kappa_4$ to $\|\Af\|_2\le N^{\kappa_7}$ with some positive  constant $\kappa_7$.
We also restrict our attention to the regime $|E|\le N^{\kappa_7+1}$ since the complementary regime
 will be covered by the regime \eqref{eq:def_Hbouttwo} in part (ii). 
Let $\varphi$ and $\psi$ be defined as in part (iii) of Lemma~\ref{lem:Gap_Lemma} and recall the
definition of $\vartheta$ from \eqref{def of vartheta}.

{\emph{Proof of \eqref{eq:local_law_inside_entry}}}:
We first show that 
\begin{equation} \label{eq:gap_lemma_no_char_function}
 \Lambda(E+\ii\eta) \prec \varphi 
\end{equation}
uniformly for $E +\ii\eta \in\Hbin'$ and $|E|\le N^{\kappa_7+1}$. 

We start with   some auxiliary estimates. 
By the definition of $\Hbin'$ in \eqref{eq:def_Hbinprime} and setting $\af \defeq (a_1, \ldots, a_N)$, we have
\begin{equation} 
\label{eq:M_two_bounded}
\norm{\mf(z)} \, \lesssim \, \frac{1}{ 1 + \norm{\af-z}} \, \lesssim \, 1,
\end{equation}
uniformly for $z \in \Hbin'$. We remark that $\norm{\af} = \normtwo{\Af}$. 

We now verify that, uniformly for $z \in \Hbin'$, we have
\begin{equation}\label{eq:product_condition_mf}
\norm{\mf(z)} \norm{\mf^{-1}(z)}\lesssim 1.
\end{equation}
Applying $\norm{\genarg}$ to \eqref{eq:Dyson} as well as using \eqref{eq:norm_Sf} and \eqref{eq:M_two_bounded}, we get that
\begin{equation}
 \norm{\mf^{-1}(z)} \, \lesssim \, \norm{\af-z} + 1 \lesssim 1 + \abs{z} + \norm{\af}  
\label{eq:mf_inverse_bounded}
\end{equation}
for $z \in \Hbin'$. Thus, combining the first bounds in \eqref{eq:M_two_bounded} and in \eqref{eq:mf_inverse_bounded} yields \eqref{eq:product_condition_mf}.

From the definition of $\Hbin'$ in \eqref{eq:def_Hbinprime},  using \eqref{eq:M_two_bounded}, \eqref{eq:lopb_bounded} and \eqref{eq:comparing_norms_of_Lf_and_lopb}, we obtain 
\begin{equation}
\norm{\Lf^{-1}} \, \lesssim \, 1, \qquad  \norm{(\Lf^{-1})^*} \, \lesssim\,  1, \label{eq:Lf_bounded}
\end{equation}
where the adjoint is introduced above \eqref{eq:def_matrix_valued_scalar_product}.

We will now use  part  (iii) of Lemma \ref{lem:Gap_Lemma} to prove \eqref{eq:gap_lemma_no_char_function}.
To check the condition $\psi(\eta) \leq N^{-\delta}$ in that lemma,
 we use \eqref{eq:M_two_bounded}, \eqref{eq:Lf_bounded} and \eqref{eq:product_condition_mf} to obtain 
$\psi(\eta) \lesssim 1/(N\eta)$. Hence, $\psi(\eta) \leq N^{-\gamma/2}$ for $\eta \geq N^{-1+\gamma}$ and we choose $\delta =\gamma/2$ in~\eqref{eq:gap_lemma_bound_hermitian}. 

We now estimate $\varphi$ and $\vartheta$ in our setting. 
From \eqref{eq:product_condition_mf}, \eqref{eq:M_two_bounded} and \eqref{eq:Lf_bounded}, we conclude that $\varphi \lesssim \norm{\mf} \Psi$, where we introduced the control parameter
\[  \Psi \defeq \sqrt{\frac{\norm{\Im \mf}}{N\eta}} + \frac{\norm{\mf}}{\sqrt{N}} + \frac{\norm{\mf}}{N\eta} . \]
We note that the factor $\norm{\mf}$ is kept in the bound $\varphi \lesssim \norm{\mf} \Psi$ and the definition of $\Psi$ to control $\norm{\mf^{-1}}$ factors 
via \eqref{eq:product_condition_mf} later and to track 
the correct dependence of the right-hand sides of \eqref{eq:local_law_inside_entry} 
and \eqref{eq:local_law_inside_average} on $\eta$.
For the second purpose, we will use the following estimate. 
Combined with \eqref{eq:Mf_upper_bound_eta}, the bound \eqref{eq:M_two_bounded} yields 
\begin{equation}
\norm{\mf} \, \lesssim \, \frac{1}{1 + \disvz}. \label{eq:m_bound_aux1} 
\end{equation}
For $\vartheta$, we claim that 
\begin{equation} \label{eq:estimate_theta}
 \vartheta \gtrsim (1+ \abs{z}+ \norm{\af})^{-1}, \qquad \vartheta \gtrsim \norm{\mf}. 
\end{equation}
Indeed, for the first bound, we apply \eqref{eq:norm_Sf}, \eqref{eq:M_two_bounded}, \eqref{eq:Lf_bounded} and  the second bound in \eqref{eq:mf_inverse_bounded} to  
the definition of $\vartheta$, \eqref{def of vartheta}. Using \eqref{eq:product_condition_mf} instead of \eqref{eq:M_two_bounded} and \eqref{eq:mf_inverse_bounded} yields the second bound. 

Now, to prove \eqref{eq:gap_lemma_no_char_function}, we show that $\char(\Lambda \leq \vartheta) =1$ a.w.o.p. for $\eta \geq N^{-1+\gamma}$ on the left-hand side of \eqref{eq:gap_lemma_bound_hermitian}. 
The first step is to establish $\Lambda \leq \vartheta$ for large $\eta$.
For $\eta \geq \max\{1,\abs{E}, \normtwo{\Af}\}$, we have $\Lambda \prec \eta^{-2}$ by \eqref{eq:behaviour_lambda_large_eta}. 
By \eqref{eq:estimate_theta}, we have $\vartheta \gtrsim \eta^{-1}$ for $\eta \geq \max\{1,\abs{E}, \normtwo{\Af}\}$. 
Therefore, there is $\kappa>\kappa_7 +1$ such that $\Lambda (\eta) \leq \vartheta(\eta)$ a.w.o.p. for all $\eta \geq N^{\kappa}$. 
Together with \eqref{eq:gap_lemma_bound_hermitian}, this proves \eqref{eq:gap_lemma_no_char_function} for $\eta \geq N^{\kappa}$. 

The second step is a stochastic continuity argument to reduce $\eta$ for the domain of 
validity of  \eqref{eq:gap_lemma_no_char_function}. 
The estimate \eqref{eq:gap_lemma_bound_hermitian} asserts that $\Lambda$ cannot take on any value between $\varphi$ and $\vartheta$ 
with very high probability. Since $\eta \mapsto \Lambda(\eta)$ is continuous, $\Lambda$ remains bounded by $\varphi$ for all values of $\eta$ as long as $\varphi$ is smaller than $ \vartheta$. The precise formulation 
of this procedure is found e.g. in Lemma A.2 of \cite{Ajankirandommatrix} and we leave the straightforward 
check of its conditions to the reader.  
The bound \eqref{eq:gap_lemma_no_char_function} yields \eqref{eq:local_law_inside_entry} in the regime $|E|\le N^{\kappa_7+1}$. 

\emph{Proof of \eqref{eq:local_law_inside_average}}: We apply Lemma \ref{lem:avg_tech_local_law} with $\Phi \defeq \norm{\mf^{-1}}\varphi$.
The condition \eqref{eq:assumption_averaged_local_law} is satisfied by the definition of $\Phi$ and~\eqref{eq:gap_lemma_no_char_function}. 
Since $\Phi \lesssim \Psi$ it is easily checked that all terms on the right-hand side of 
\eqref{eq:local_law_average_general} are bounded by $\norm{\mf}\max\{N^{-1/2}, \Psi\}\Psi$.
Therefore, using \eqref{eq:Lf_bounded} and \eqref{eq:m_bound_aux1}, the averaged local law, \eqref{eq:local_law_average_general}, yields 
\begin{equation} \label{eq:average_local_law_aux1}
 \absB{\frac{1}{N} \sum_{i=1}^N c_i(G_{ii} -m_i )} \prec \norm{\mf}\max\Big\{ \frac{1}{\sqrt N},\Psi\Big\} \Psi 
\lesssim \frac{1}{1 + \disvz} \Big(  \frac{\norm{\Im \mf(z)}}{N\eta}+ \frac{1}{N} + \frac{1}{1 + \disvz[2]}\frac{1}{(N\eta)^2} \Big)
\end{equation}
for any $c_1, \ldots, c_N \in \C^{K\times K}$ such that $\max_i \abs{c_i} \leq 1$. Owing to $\norm{\Im\mf} \lesssim 1$ by \eqref{eq:M_two_bounded}, the bound \eqref{eq:local_law_inside_average} follows. 

We now turn to the proof of (ii) which  is divided into two steps. 
In the first step, we show Lemma \ref{lem:local_law_aux_bound}. Therefore, we will 
 follow the proof of \eqref{eq:average_local_law_aux1} with the bounds \eqref{eq:m_bound_aux1} and \eqref{eq:Lf_bounded} replaced by 
their weaker analogues \eqref{eq:m_bound_aux2} and \eqref{eq:Lf_bounded2} below
that deteriorate as $\disvz$ becomes small.
After having completed Lemma~\ref{lem:local_law_aux_bound}, we immediately get Theorem~\ref{thm:no_eigenvalues} 
via the proof given in Section \ref{subsec:local_law}. 
Finally, in the second step, proceeding similarly as in the proof of (i), the bounds \eqref{eq:local_law_outside_entry} and \eqref{eq:local_law_outside_average} will be obtained from Theorem~\ref{thm:no_eigenvalues}. 

\begin{proof}[Step 1: Proof of Lemma \ref{lem:local_law_aux_bound}]
We first give the replacements for the bounds \eqref{eq:m_bound_aux1} and \eqref{eq:Lf_bounded} 
that served as inputs for the previous proof of part (i).
The replacement for \eqref{eq:m_bound_aux1} is a direct consequence of \eqref{eq:Mf_upper_bound_eta}:
\begin{equation} \label{eq:m_bound_aux2}
 \norm{\mf} \leq \frac{1}{\disvz}.
\end{equation}
The replacement of \eqref{eq:Lf_bounded} is the bound 
\begin{equation} \label{eq:Lf_bounded2}
 \norm{\Lf^{-1}} + \norm{(\Lf^{-1})^*} \lesssim 1 + \frac{1}{\disvz[26]},
\end{equation}
 which is obtained by distinguishing the regimes $\normtwo{\Mf}^2 \norm{\sopb} > 1/2$ and $\normtwo{\Mf}^2 \norm{\sopb} \leq 1/2$. 
In the first regime, we conclude from \eqref{eq:lopb_bounded_sp} and \eqref{eq:lopb_bounded} that 
\[ \norm{\lopb^{-1}} + \norm{(\lopb^{-1})^*}\lesssim 1 + \normtwo{\Mf}^2 + \frac{\normtwo{\Mf}^9\normtwo{\Mf^{-1}}^9}{\normtwo{\Mf}^4\disvz[8]} \lesssim 1 + \frac{1}{\disvz[26]}, \]
where we used the lower bound on $\Mf$ given by the definition of the regime and $\norm{\sopb} \lesssim 1$ as well as the bound $\normtwo{\Mf}\normtwo{\Mf^{-1}}\lesssim1/ \disvz[2]$ that is proven as \eqref{new product} below.
In the second case, we use the simple bound 
$\norm{\lopb^{-1}} + \norm{(\lopb^{-1})^*} \leq 2/(1- \normtwo{\Mf}^2 \norm{\sopb}) \leq 4$. 
Thus, \eqref{eq:comparing_norms_of_Lf_and_lopb} yields \eqref{eq:Lf_bounded2}.

Next, we will check that the following weaker version of \eqref{eq:product_condition_mf} holds 
\begin{equation}\label{new product}
 \norm{\mf(z +\ii s )} \norm{\mf^{-1}(z +\ii s )} \lesssim 1+ \frac{1}{d^2_\rho(z+\ii s)} 
 \end{equation}
for all $z \in \Hboutone\cup\Hbouttwo$ and $s \geq 0$. This is straightforward for $z \in \Hboutone$ since 
in this case $|z|$, $\| \Af \|_2$ and $\supmeas$ all remain bounded (see \eqref{eq:support_V}), so similarly to \eqref{eq:mf_inverse_bounded}
we have $ \norm{\mf^{-1}(z +\ii s )}\lesssim 1+ s +  \norm{\mf(z +\ii s )}$. For $|s|\le C$
\eqref{new product} directly follows from \eqref{eq:m_bound_aux2}, while for large $s$ we have $\norm{\mf(z +\ii s )}\lesssim s^{-1}$
and $ \norm{\mf^{-1}(z +\ii s )}\lesssim s$, so \eqref{new product} also holds.

Suppose now that $z \in\Hbouttwo$. In this regime $z$ is far away from the
spectrum of $\Af$, so  by \eqref{eq:support_V} we know that $\dist( z + \ii s, \spec \bs{A})\sim \dist( z+ \ii s, \supmeas)\ge 1$.
This means that 
\begin{equation}\label{m1}
\norm{\mf (z+\ii s)}\lesssim \frac{1}{\dist( z+ \ii s, \supmeas)} \,\sim\, \frac{1}{\dist( z+ \ii s, \spec \bs{A})}\,=\, \frac{1}{\smin(\Af-(z+\ii s)\id)},
\end{equation}
 and  hence from the Dyson equation
\begin{equation}\label{m2}
   \Big\| \frac{1}{\mf(z+\ii s)}\Big\|\le \| \Af-(z+\ii s)\id\|_2  + \norm{\sopb} \lesssim  \| \Af-(z+\ii s)\id\|_2.
\end{equation}
Since $\Af$ is Hermitian, we have the bound
\begin{equation}\label{comp}
  \frac{\norm{\Af-(z+\ii s)\id}_2}{\smin(\Af-(z+\ii s)\id)}\le \frac{\norm{\Af-z\id}_2}{\smin(\Af-z\id)} \le\kappa_9
 \end{equation}
 for any $s\ge 0$, where the first inequality  comes from the spectral theorem and the
 second bound is from 
  the definition of $\Hbouttwo$.
Therefore $\smin(\Af-(z+\ii s)\id)\sim\norm{\Af-(z+\ii s)\id}_2 $, and 
thus  \eqref{new product} follows from \eqref{m1} and~\eqref{m2}.

Now we can complete Step 1 by following the proof of part (i) but using 
 \eqref{eq:m_bound_aux2}, \eqref{eq:Lf_bounded2}  and \eqref{new product}
 instead of \eqref{eq:m_bound_aux1}, \eqref{eq:Lf_bounded} and  \eqref{eq:product_condition_mf}, 
 respectively.  It is easy to see that only
 these three estimates on $\|\mf\|$, $\|\mf\|\|\mf^{-1}\|$ and $\|\Lf^{-1}\|$   were used as inputs in this argument.
 The  resulting  estimates are weaker by  multiplicative factors involving  certain power of $1+ 1/\disvz$. 
We thus obtain 
a version of \eqref{eq:average_local_law_aux1} for $\eta \geq N^{-1+\gamma}(1+\disvz[-p])$ with $(1+\disvz)^{-1}$ replaced by $\max\{1,\disvz[-P]\}$ for some explicit $p, P \in\N$. 
Thus, applying \eqref{eq:Im_Mf_upper_bound_away_support} to estimate $\Im \mf$ in \eqref{eq:average_local_law_aux1} instead of $\norm{\Im \mf} \lesssim 1$ and possibly increasing $P$
yields \eqref{eq:average_local_law_aux2}.
\end{proof}

\emph{Step 2:} Continuing the proof of part (ii) of Lemma~\ref{lem:local_law}, we draw two consequences from Theorem~\ref{thm:no_eigenvalues} and the fact that $\Gf$ is the Stieltjes transform 
of a positive semidefinite matrix-valued measure $V_{\Gf}$ supported on $\spec \Hf$ with $V_{\Gf}(\spec \Hf) = \id$. 
Let $\delta>0$ be chosen as in Theorem~\ref{thm:no_eigenvalues}.
Since the spectrum of $\Hf$ is contained in $\{ \omega \in\R \colon \dist(\omega,\supmeas) \leq N^{-\delta}\}$ a.w.o.p.
by Theorem~\ref{thm:no_eigenvalues}, we have 
 \[ \normtwo{\Gf} \lesssim \frac{1}{\disvz}, \qquad \Im \Gf \lesssim \frac{\eta}{\disvz[2]}\id   \]
a.w.o.p. for all $z\in\Hb$ satisfying $\disvz\geq N^{-\delta/2}$. 
Therefore, \eqref{eq:ward_identity} implies for all $z\in\Hb$ satisfying $\disvz\geq N^{-\delta/2}$ that 
\begin{equation} \label{eq:estimate_Lambda_w}
\Lambda_{\rm{hs}} + \Lambda_{\rm{w}} \prec \frac{1}{\disvz\sqrt{N}}.
\end{equation}

Since $\Mf$ is the Stieltjes transform of $V_\Mf$ defined in \eqref{eq:def_V_Mf} and $V_\Mf(\R) = \id$ and $\Gf$ is the Stieltjes transform of $V_\Gf$ 
we conclude that there is $\kappa>0$ such that 
\begin{equation} \label{eq:Gf_minus_Mf_large_z}
 \Lambda \lesssim \normtwo{\Gf-\Mf} \lesssim \abs{z}^{-2} 
\end{equation}
a.w.o.p.  uniformly for all $z\in\Hb$ satisfying $\abs{z} \geq N^{\kappa}$.
Here, we used that $\supp V_\Mf \subset \supmeas$ and hence $\diam(\supp V_\Mf) \lesssim N^{\kappa_7+1}$ by 
\eqref{eq:A_two_bounded} and \eqref{eq:support_V} as well as
 $\diam(\supp V_\Gf) \leq \diam(\spec \Hf) \lesssim N^{\kappa_7+1}$ a.w.o.p.
by Theorem~\ref{thm:no_eigenvalues}.  

Hence, owing to \eqref{eq:estimate_theta} and \eqref{eq:Gf_minus_Mf_large_z}, by possibly increasing $\kappa>0$, we can assume 
that $\Lambda\leq \vartheta$ a.w.o.p. for all $z\in\Hboutone\cup\Hbouttwo$ satisfying $\abs{z} \geq N^{\kappa}$. 
Thus, to estimate $\norm{\gf-\mf}$ we start from \eqref{Lambda gap bound} and use \eqref{eq:Lf_bounded2}, \eqref{eq:m_bound_aux2}, \eqref{eq:estimate_Lambda_w} 
and \eqref{eq:product_condition_mf} to obtain an explicit $P \in \N$ such that $\norm{\gf-\mf} \prec \norm{\mf} \max\{\disvz[-1], \disvz[-P]\} N^{-1/2}$ a.w.o.p. 
For the offdiagonal terms of $\Gf$, we apply \eqref{eq:estimate_Lambda_w} to \eqref{eq:initial_bound_offdiag}. 
This yields
\begin{equation} \label{eq:local_law_outside_aux}
  \Lambda \prec \norm{\mf}\max\bigg\{\frac{1}{\disvz},\frac{1}{\disvz[P]} \bigg\} \frac{1}{\sqrt{N}} 
\end{equation} 
for $z\in\Hboutone\cup\Hbouttwo$ satisfying $\abs{z} \geq N^{\kappa}$.
Employing the stochastic continuity argument from Lemma A.2 in \cite{Ajankirandommatrix} as before, we obtain \eqref{eq:local_law_outside_aux} for all $z \in\Hboutone\cup\Hbouttwo$ satisfying $\disvz \geq N^{-\delta/2}$. 
We use \eqref{eq:m_bound_aux2} in \eqref{eq:local_law_outside_aux}, replace $P$ by $P+1$ and $\delta$ by $\delta/2$. Thus, we have proven \eqref{eq:local_law_outside_entry} 
for all $z\in\Hboutone\cup\Hbouttwo$ satisfying $\disvz \geq N^{-\delta}$. 
 Notice that this argument covers the case $|E|\ge N^{\kappa_7+1}$ as well that was left open in Step 1. 

For the proof of \eqref{eq:local_law_outside_average}, we set $\Phi \defeq (\disvz\sqrt{N})^{-1}$ and apply Lemma \ref{lem:avg_tech_local_law}. 
Its assumption $\Lambda \prec \Phi/\norm{\mf^{-1}}$ is satisfied by~\eqref{eq:local_law_outside_aux} and \eqref{eq:product_condition_mf}. Using \eqref{eq:Lf_bounded2}, \eqref{eq:m_bound_aux2}, \eqref{eq:product_condition_mf}
 and \eqref{eq:estimate_Lambda_w}, this proves \eqref{eq:local_law_outside_average} and hence concludes the proof of Lemma~\ref{lem:local_law}.
\end{proof}

\providecommand{\bysame}{\leavevmode\hbox to3em{\hrulefill}\thinspace}
\providecommand{\MR}{\relax\ifhmode\unskip\space\fi MR }
\providecommand{\MRhref}[2]{%
  \href{http://www.ams.org/mathscinet-getitem?mr=#1}{#2}
}
\providecommand{\href}[2]{#2}

\end{document}